\documentclass[11pt, reqno]{amsart}
\usepackage{txfonts}
\usepackage{amsmath,amssymb,amsthm,mathrsfs,enumerate,bm,xcolor,multirow,pbox}
\usepackage{mathrsfs}
\usepackage{algorithm,algorithmic,float}
\usepackage{graphicx,color,framed,tikz,caption,subcaption}
\usepackage{enumitem}
\setlist{leftmargin=8mm}
\usepackage[colorlinks,linkcolor=black,citecolor=black,urlcolor=black]{hyperref}
\allowdisplaybreaks[4]
\numberwithin{equation}{section}
\newcommand{\N}{\mathbb{N}}
\newcommand{\R}{\mathbb{R}}

\newcommand{\pnorm}[2]{\lVert #1\rVert_{#2}}
\newcommand{\bigpnorm}[2]{\big\lVert#1\big\rVert_{#2}}

\newcommand{\abs}[1]{\lvert#1\rvert}
\newcommand{\bigabs}[1]{\big\lvert#1\big\rvert}
\newcommand{\biggabs}[1]{\bigg\lvert#1\bigg\rvert}
\newcommand{\iprod}[2]{\langle#1,#2\rangle}
\newcommand{\bigiprod}[2]{\big\langle#1,#2\big\rangle}

\renewcommand{\epsilon}{\varepsilon}

\renewcommand{\d}[1]{\mathrm{d}#1}

\newcommand{\smallo}{\mathfrak{o}}
\newcommand{\bigo}{\mathcal{O}}

\newcommand{\equald}{\stackrel{d}{=}}
\renewcommand{\hat}{\widehat}
\renewcommand{\tilde}{\widetilde}

\DeclareMathOperator{\E}{\mathbb{E}}
\DeclareMathOperator{\Prob}{\mathbb{P}}

\DeclareMathOperator{\sign}{\texttt{sgn}}

\DeclareMathOperator{\var}{Var}
\DeclareMathOperator{\cov}{Cov}

\DeclareMathOperator{\op}{op}

\DeclareMathOperator{\err}{\texttt{err}}

\DeclareMathOperator{\prox}{\mathsf{prox}}

\let\limsup\relax
\DeclareMathOperator*\limsup{\overline{lim}}

\DeclareMathOperator*{\argmin}{arg\,min\,}

\theoremstyle{definition}\newtheorem{problem}{Problem}[section]
\theoremstyle{definition}\newtheorem{definition}[problem]{Definition}
\theoremstyle{remark}\newtheorem{assumption}{Assumption}

\theoremstyle{remark}\newtheorem{remark}{Remark}
\theoremstyle{definition}\newtheorem{example}[problem]{Example}
\theoremstyle{plain}\newtheorem{theorem}[problem]{Theorem}
\theoremstyle{plain}\newtheorem{question}{Question}
\theoremstyle{plain}\newtheorem{lemma}[problem]{Lemma}
\theoremstyle{plain}\newtheorem{proposition}[problem]{Proposition}
\theoremstyle{plain}
\theoremstyle{plain}

\AtBeginDocument{%
	\def\MR#1{}
}

\graphicspath{{Fig/}}

\begin{document}

\title[Gradient descent inference in ERM]{Gradient descent inference in empirical risk minimization}
\thanks{The research of Q. Han is partially supported by NSF grant DMS-2143468.}

\author[Q. Han]{Qiyang Han}

\address[Q. Han]{
Department of Statistics, Rutgers University, Piscataway, NJ 08854, USA.
}
\email{qh85@stat.rutgers.edu}

\author[X. Xu]{Xiaocong Xu}

\address[X. Xu]{
	Data Sciences and Operations Department, Marshall School of Business, University of Southern California, Los Angeles, CA 90089, USA.
}
\email{xuxiaoco@marshall.usc.edu}

\date{\today}


\begin{abstract}
Gradient descent is one of the most widely used iterative algorithms in modern statistical learning. However, its precise algorithmic dynamics in high-dimensional settings remain only partially understood, which has limited its broader potential for statistical inference applications.

This paper provides a precise, non-asymptotic joint distributional characterization of gradient descent iterates and their debiased statistics in a broad class of empirical risk minimization problems, in the so-called mean-field regime where the sample size is proportional to the signal dimension. Our non-asymptotic state evolution theory holds for both general non-convex loss functions and non-Gaussian data, and reveals the central role of two Onsager correction matrices that precisely characterize the non-trivial dependence among all gradient descent iterates in the mean-field regime. 

Leveraging the joint state evolution characterization, we show that the gradient descent iterate retrieves approximate normality after a debiasing correction via a linear combination of observable loss derivative directions from all past iterates. Crucially, the debiasing coefficients are directly linked to the Onsager correction matrices which can be estimated in a fully data-driven manner via the proposed \emph{gradient descent inference algorithm}. This leads to a new algorithmic statistical inference framework based on debiased gradient descent, which (i) applies to a broad class of models with both convex and non-convex losses, (ii) remains valid at each iteration without requiring algorithmic convergence, and (iii) exhibits a certain robustness to possible model misspecification. As a by-product, our framework also provides algorithmic estimates of the generalization error at each iteration.

We demonstrate our theory and inference methods in the canonical single-index regression model and a generalized logistic regression model, where the natural loss functions may exhibit arbitrarily non-convex landscapes. Our analysis further shows that, in linear regression with squared loss, the proposed debiased gradient descent iterate eventually coincides with the debiased convex regularized estimator in a mean-field distributional sense, and the quality of statistical inference for the unknown signal aligns exactly with the generalization error achieved along the algorithmic trajectory.
\end{abstract}

\maketitle

\vspace{-2em}
\setcounter{tocdepth}{1}
\tableofcontents

\sloppy

\section{Introduction}

Suppose we observe i.i.d. data $\{(A_i,Y_i)\}_{i \in [m]}\subset \R^n\times \R$, where $A_i$'s are features/covariates and $Y_i$'s are labels related via the model
\begin{align}\label{def:model_general}
Y_i=\mathcal{F}(\iprod{A_i}{\mu_\ast},\xi_i),\quad i \in [m].
\end{align}
Here $\mathcal{F}:\R^2\to \R$ is a model-specific function, $\mu_\ast \in \R^n$ is the unknown signal, and $\xi_i$'s are statistical noises independent of $A_i$'s. For notational convenience, we write $A \in \R^{m\times n}$ whose rows collect $A_i^\top$'s, and $Y=(Y_i)_{i \in [m]}, \xi \equiv (\xi_i)_{i \in [m]}\in \R^m$.

While (\ref{def:model_general}) covers a broad range of concrete models, here we have in mind the following two prominent examples:
\begin{example}[Single-index regression model]\label{model:single_index}
We observe $(Y_i,A_i) \in \R\times \R^n$ according to $
Y_i = \varphi_\ast\big(\iprod{A_i}{\mu_\ast}\big)+\xi_i$, $i \in [m]$, where $\varphi_\ast:\R\to \R$ is a link function possibly without any apriori convexity. This single-index regression model can be identified as (\ref{def:model_general}) by setting $\mathcal{F}(z,\xi)\equiv \varphi_\ast(z)+\xi$. The special case of linear regression amounts to taking $\varphi_\ast =\mathrm{id}$. 
\end{example}

\begin{example}[Generalized logistic regression]\label{model:1bit}
We observe $(Y_i,A_i)\in \{\pm 1\}\times \R^n$ according to the model $
Y_i = \sign\big(\iprod{A_i}{\mu_\ast}+\xi_i\big)$, $i \in [m]$. Here for definiteness, we interpret $\sign(x)=2\cdot\bm{1}_{x\geq 0}-1$ for $x\in \R$. This model can be recast into (\ref{def:model_general}) by setting $\mathcal{F}(z,\xi) = \sign(z+\xi)$, and is also known under the name of the noisy one-bit compressed sensing model. The special case of logistic regression can be recovered upon suitable specification of the error distribution for $\{\xi_i\}$'s. 
\end{example}

A major goal of statistician is to make inference about the unknown signal $\mu_\ast \in \R^n$ based on the observed data $\{(A_i,Y_i)\}_{i \in [m]}$ from (\ref{def:model_general}). In this paper, we will be mainly interested in the so-called `proportional regime' (or `mean-field regime', cf. \cite{montanari2018mean}), where
\begin{align}\label{def:mean_field}
\hbox{the sample size $m$ and the signal dimension $n$ are of the same order}.
\end{align}
The mean-field regime (\ref{def:mean_field}) is particularly challenging for statistical inference of $\mu_\ast$, as consistent estimation of $\mu_\ast$ is in general impossible in this regime.

\subsection{Review: Mean-field debiasing methods for convex regularized estimators}

In convex models, a popular statistical paradigm for the inference purpose is to use a suitable debiased version of the empirical risk minimizer $\hat{\mu}$ obtained from solving the empirical risk minimization problem
\begin{align}\label{def:ERM_general}
\hat{\mu}\in \argmin_{\mu \in \R^n} \sum_{i \in [m]} \mathsf{L}\big(\iprod{A_i}{\mu}, Y_i\big)+ \sum_{j \in [n]} \mathsf{f}(\mu_j).
\end{align}
Here $\mathsf{L}:\R^2 \to \R_{\geq 0}$ is a loss function, and $\mathsf{f}:\R\to \R_{\geq 0}$ is a (convex) regularizer designed to promote the structure of $\mu_\ast$. As a canonical example, in the linear model $\mathcal{F}(x,y)=x+y$ under Gaussian i.i.d. design with entrywise variance $1/n$, the regularized estimator $\hat{\mu}^{\mathrm{ls}}$ under the squared loss $\mathsf{L}(x,y)=(x-y)^2/2$ admits the following distributional characterization in the mean-field regime (\ref{def:mean_field}): for some $\alpha_\ast \in \mathbb{R}_{>0}$ and Gaussian noise $\mathsf{W}^{\mathrm{ls}}\in \R^n$, it holds in an averaged sense\footnote{In the introduction, for two random vectors $X, Y \in \mathbb{R}^n$ defined possibly on different probability spaces, we say that $X \stackrel{d}{\approx} Y$ in an averaged sense if, for any sufficiently regular test function $\psi$, the relation $n^{-1} \sum_{i \in [n]} \psi(X_i) \approx n^{-1} \sum_{i \in [n]} \E \psi(Y_i)$ holds with high probability.} that
\begin{align}\label{eqn:reg_est}
\hat{\mu}^{\mathrm{ls}} \stackrel{d}{\approx} \prox_{\alpha_\ast\mathsf{f}}\big(\mu_\ast+\mathsf{W}^{\mathrm{ls}}\big).
\end{align}
We refer the reader to, e.g.,  \cite{elkaroui2013asymptotic,stojnic2013framework,donoho2016high,elkaroui2018impact,thrampoulidis2015regularized,thrampoulidis2018precise,sur2019likelihood,sur2019modern,miolane2021distribution,celentano2023lasso,han2023noisy,han2023universality,han2023distribution} for precise statements of (\ref{eqn:reg_est}) for various concrete regularizers, and a formal definition of the proximal operator $\prox_\cdot$ can be found in (\ref{def:prox_op}).

The core to debiasing $\hat{\mu}^{\mathrm{ls}} $ is to expose the Gaussian noise $\mathsf{W}^{\mathrm{ls}}$ in (\ref{eqn:reg_est}) through a correction to $\hat{\mu}^{\mathrm{ls}}$ along its own loss derivative direction (cf.~\cite{javanmard2014confidence,javanmard2014hypothesis,miolane2021distribution,celentano2023lasso,bellec2023debias,bellec2025observable}): for some scalar $\omega^{\mathrm{ls}} \in \mathbb{R}$, the oracle debiased regularized estimator 
	\begin{align}\label{eqn:debiased_reg_est}
	\mu^{\mathrm{ls}}_{\mathrm{db}} \equiv \hat{\mu}^{\mathrm{ls}} + \omega^{\mathrm{ls}} \cdot A^\top (A\hat{\mu}^{\mathrm{ls}} - Y) \stackrel{d}{\approx} \mathcal{N}\big(\mu_\ast, \sigma_{\mathrm{db}}^2 I_n\big)
	\end{align}
	is approximately normal (typically in an averaged sense). Statistical inference for $\mu_\ast$ can then be performed using a data-driven version of the debiased estimator $\hat{\mu}^{\mathrm{ls}}_{\mathrm{db}}$, in which the scalar $\omega^{\mathrm{ls}}$ is replaced by a consistent estimator $\hat{\omega}^{\mathrm{ls}} = \hat{\omega}^{\mathrm{ls}}\big(\{(A_i, Y_i)\}_{i \in [m]}\big) \in \mathbb{R}$. We refer readers to \cite{bellec2023debias,bellec2025observable} for a general `degrees-of-freedom adjustment' methodology to construct such a consistent estimator $\hat{\omega}^{\mathrm{ls}}$ in the mean-field regime (\ref{def:mean_field}).
	
	An important theoretical property of the debiased regularized estimator $\mu^{\mathrm{ls}}_{\mathrm{db}} $ lies in the direct connection between its variance $\sigma_{\mathrm{db}}^2$ and the generalization error $\mathscr{E}_{\mathrm{sq}}$ achieved by $\hat{\mu}^{\mathrm{ls}}$: 
	\begin{align}\label{eqn:var_gen_intro}
	\sigma_{\mathrm{db}}^2 \approx \kappa \cdot \mathscr{E}_{\mathrm{sq}},\quad \kappa \equiv \phi^{-1}\equiv (m/n)^{-1}.
	\end{align}
	Consequently, the quality of statistical inference for $\mu_\ast$, as measured by the length of the confidence intervals produced by the debiased estimator $\hat{\mu}^{\mathrm{ls}}_{\mathrm{db}}$, is precisely aligned with the generalization error achieved by the regularized estimator $\hat{\mu}^{\mathrm{ls}}$.

\subsection{Motivation and the goal of this paper}
In practice, since the optimization problem (\ref{def:ERM_general}) generally lacks a closed-form solution, the empirical risk minimizer $\hat{\mu}$ in (\ref{def:ERM_general}) and its debiased version $\hat{\mu}^{\mathrm{ls}}_{\mathrm{db}}$ in (\ref{eqn:debiased_reg_est}) are typically computed via iterative algorithms. One of the simplest yet broadly applicable methods for this purpose is the proximal gradient descent algorithm: starting with an initialization $\mu^{(0)}\in \R^n$ and a pre-specified step size $\eta>0$, the algorithm iteratively updates $\mu^{(t)}$ via:
\begin{align}\label{def:gd_intro}
\mu^{(t)}=\prox_{\eta \mathsf{f}}\big( \mu^{(t-1)}-\eta\cdot A^\top \partial_1 \mathsf{L} (A\mu^{(t-1)},Y)\big),\quad t=1,2,\ldots.
\end{align}
Here $\partial_1 \mathsf{L}(x,y)\equiv (\partial \mathsf{L}/\partial x)(x,y)$ is understood as applied row-wise.

When the loss function $x \mapsto \mathsf{L}(x,y)$ is convex and the gradient descent iterates $\mu^{(t)}$ are provably close to the empirical risk minimizer $\hat{\mu}$, one may directly use the algorithmic output $\mu^{(t)}$ for large $t$ to construct an accurate approximation of the debiased estimator $\hat{\mu}^{\mathrm{ls}}_{\mathrm{db}}$ in (\ref{eqn:debiased_reg_est}) for statistical inference of $\mu_\ast$. However, even in this favorable setting, it is often preferable in practice to early-stop the gradient descent due to its implicit regularization effect, which typically leads to smaller generalization error. This benefit is already evident in the simplest linear model with squared loss and Ridge regularization; cf.~\cite{ali2020implicit}.
	
A much more challenging scenario relevant to our examples arises when the loss function is non-convex, in which case the resulting proximal gradient descent (\ref{def:gd_intro}) may fail to converge to the empirical risk minimizer (\ref{def:ERM_general}). For instance, in the single-index regression model from Example~\ref{model:single_index}, the canonical squared loss $x \mapsto \mathsf{L}(x,y) = (\varphi_\ast(x) - y)^2$ can already exhibit \emph{arbitrary} non-convex landscape due to the presence of a general link function $\varphi_\ast$. Furthermore, even if the algorithm converges, the existing debiasing methodology (\ref{eqn:debiased_reg_est}) is not directly applicable without convexity assumptions. As a result, statistical inference for $\mu_\ast$ based on existing debiasing paradigms for empirical risk minimizers becomes infeasible.
	
The foregoing discussion naturally leads to the following question:

\begin{question}\label{question}
Can the gradient descent iterate $\mu^{(t)}$ itself, rather than the empirical risk minimizer $\hat{\mu}$, be directly used for the purpose of statistical inference of $\mu_\ast$ in the mean-field regime (\ref{def:mean_field}), both in convex and non-convex settings?
\end{question}

A recent series of works \cite{bellec2024uncertainty,tan2024estimating} demonstrates the feasibility of statistical inference using $\mu^{(t)}$ in the convex case of the linear model with squared loss and Gaussian data. These works introduce a data-driven iterative debiasing methodology that leverages the derivatives of all past gradient descent mappings, and thus providing a promising approach.

The goal of this paper is to provide a systematic solution to Question~\ref{question} by developing a statistical inference framework that can be applied during gradient descent training for the general class of models in \eqref{def:model_general}, including both convex and non-convex loss cases.
	
Our approach builds on the state evolution formalism from the Dynamical Mean Field Theory (DMFT) \cite{mignacco2020dynamical,altieri2020dynamical,celentano2021high,mignacco2022effective,gerbelot2024rigorous}, particularly in the recently developed form presented in \cite{han2025entrywise}. We characterize the joint distribution of $\mu^{(t)}$, $A\mu^{(t)}$ and their debiased statistics, from which we show that the gradient iterate $\mu^{(t)}$ retrieves approximate normality after a debiasing correction via a linear combination of observable loss derivative directions $\{\partial_1 \mathsf{L}(A\mu^{(s)},Y)\}_{s \in [0:t-1]}$ from all past iterates. Furthermore, we show that the debiasing coefficients in this linear combination can be estimated in a data-driven manner using the proposed \emph{gradient descent inference algorithm} (cf.~Algorithm~\ref{def:alg_tau_rho}).

Taken together, this leads to a new algorithmic statistical inference framework based on debiased gradient descent. As will be clear from below, our algorithmic inference framework (i) applies to a broad class of models in (\ref{def:model_general}) with both convex and non-convex losses, (ii) remains valid at each iteration without requiring algorithmic convergence, and (iii) exhibits a certain robustness to possible model misspecification. As a by-product of our inference framework, we also obtain algorithmic estimates of the generalization error of $\mu^{(t)}$ at each iteration.

While we only treat the gradient descent algorithm (\ref{def:gd_intro}) and its immediate variants (cf., Eqn. (\ref{def:grad_descent})) in this paper, our theory and inference methods can be readily extended to other variations such as accelerated or noisy gradient descent. For clarity and to emphasize the key ideas of our theory and inference methods, we do not detail these extensions here.

\subsection{Mean-field dynamics of (debiased) gradient descent}

As mentioned earlier, the main theoretical tool underlying our inference method is a joint distributional characterization for $\mu^{(t)}, A\mu^{(t)}$ and their debiased statistics at each iteration $t$. Suppose $A$ has independent, mean-zero and sub-gaussian entries with variance $1/n$, a simplified version of Theorem \ref{thm:gd_se} shows that in the mean-field regime (\ref{def:mean_field}), the following holds both in an entrywise and an averaged sense for $t=1,2,\ldots$:
\begin{align}\label{ineq:intro_gd_dist}
\begin{cases}
A\mu^{(t-1)} \stackrel{d}{\approx} -\eta \sum_{s \in [1:t-1]}  \rho_{t-1,s} \cdot \partial_1\mathsf{L}\big(A\mu^{(s-1)},Y\big)+ \mathsf{Z}^{(t)},\\
\mu^{(t)} \stackrel{d}{\approx} \prox_{\eta \mathsf{f}}\big[(1+\tau_{t,t})\cdot  \mu^{(t-1)}+\sum_{s \in [1:t-1]} \tau_{t,s}\cdot  \mu^{(s-1)}+ \delta_t\cdot \mu_\ast+\mathsf{W}^{(t)}\big].
\end{cases}
\end{align}
Here $\mathsf{Z}^{(t)}$ and $\mathsf{W}^{(t)}$ are centered Gaussian vectors, $\bm{\tau}^{[t]}=(\tau_{r,s})_{r,s \in [t]}\in \R^{t\times t},\bm{\rho}^{[t-1]}=(\rho_{r,s})_{r,s \in [t-1]}\in \R^{(t-1)\times (t-1)}$ are called \emph{Onsager correction matrices}\footnote{In the statistical physics literature \cite{cugliandolo1993analytical,agoritsas2018out}, the elements of these matrices are referred to as \emph{linear response functions}. To avoid confusion with statistical terminology, we do not adopt this terminology from statistical physics in this paper.} that quantify how the current gradient descent iterates $\mu^{(t)}$ and $A\mu^{(t-1)}$ depend on past iterates, and $\delta_t\in \R$ is called \emph{information parameter} that measures the amount of information about $\mu_\ast$ contributed by the gradient descent iterate $\mu^{(t)}$ at iteration $t$. These parameters can be determined recursively for $t=1,2,\ldots$, via a specific state evolution detailed in Definition \ref{def:gd_se}.

The distributional characterization of $\mu^{(t)}$ in (\ref{ineq:intro_gd_dist}) takes a form similar to (\ref{eqn:reg_est}) with a single additional Gaussian noise term $\mathsf{W}^{(t)}$, thereby suggesting a general debiasing framework via $\mu^{(t)}$ in the same spirit as (\ref{eqn:debiased_reg_est}). We show in Theorem \ref{thm:db_gd_oracle} that this is indeed possible: with $(\omega_{r,s})_{r,s \in [t]}\equiv \bm{\omega}^{[t]}\equiv (\bm{\tau}^{[t]})^{-1}\in \R^{t\times t}$ and suitable bias-variance parameters $\big(b^{(t)}_{\mathrm{db}}, (\sigma^{(t)}_{\mathrm{db}})^2\big)\in \R\times \R_{\geq 0}$, the debiased gradient descent iterate $\mu^{(t)}_{\mathrm{db}}$ below is approximately normal in the mean-field regime (\ref{def:mean_field}), both in an entrywise and an averaged sense for $t=1,2,\ldots$:
\begin{align}\label{def:db_gd_intro}
\mu^{(t)}_{\mathrm{db}}\equiv \mu^{(t-1)}+\eta  \sum_{s \in [1:t]} \omega_{t,s} A^\top \partial_1 \mathsf{L}(A\mu^{(s-1)},Y)\stackrel{d}{\approx} \mathcal{N}\Big( b^{(t)}_{\mathrm{db}}\cdot \mu_\ast, (\sigma^{(t)}_{\mathrm{db}})^2 I_n \Big).
\end{align}
In fact, the above display follows by an `inversion' of (a suitable form of) the second display in (\ref{ineq:intro_gd_dist}); some heuristics can be found in (\ref{def:debias_gd_oracle}).

Compared to the debiased regularized estimator $\mu^{\mathrm{ls}}_{\mathrm{db}}$ in (\ref{eqn:debiased_reg_est}), the debiased gradient descent iterate $\mu^{(t)}_{\mathrm{db}}$ in (\ref{def:db_gd_intro}) exhibits a clear structural similarity: the iterate $\mu^{(t-1)}$ can be debiased to recover normality through a linear combination of observable loss derivative directions from all past iterates. A key advantage of the debiased gradient descent iterate (\ref{def:db_gd_intro}), however, is that it entirely avoids any convexity requirement on the loss landscape.

In addition, the debiased gradient descent iterate (\ref{def:db_gd_intro}) enjoys a theoretical property analogous to~(\ref{eqn:var_gen_intro}): In linear regression under the squared loss, with the same constant $\kappa$ as in (\ref{eqn:var_gen_intro}), 
\begin{align}\label{eqn:CI_len_gen_intro}
 (\sigma^{(t)}_{\mathrm{db}})^2 \approx \kappa \cdot \mathscr{E}_{\mathrm{sq}}^{(t)},\quad t=1,2,\ldots.
\end{align}
The consequences of (\ref{eqn:CI_len_gen_intro}) are two-fold. First, as $b^{(t)}_{\mathrm{db}}=1$ in linear regression, it reveals a stronger alignment between the quality of statistical inference via the debiased gradient descent iterate $\mu^{(t)}_{\mathrm{db}}$ and the generalization error achieved along the entire gradient descent trajectory. Second, (\ref{eqn:var_gen_intro}) and (\ref{eqn:CI_len_gen_intro}) imply that if the proximal gradient descent iterate $\mu^{(t)}$ converges to the regularized least squares estimator $\hat{\mu}^{\mathrm{ls}}$ as $t \to \infty$, then the debiased gradient descent iterate $\mu^{(t)}_{\mathrm{db}}$ must also converge to the debiased regularized estimator $\mu^{\mathrm{ls}}_{\mathrm{db}}$ in a mean-field distributional sense, cf. Eqn. (\ref{eqn:debias_gd_cvx}).

From these perspectives, our proposed debiased gradient descent iterate $\mu^{(t)}_{\mathrm{db}}$ in~(\ref{def:db_gd_intro}) can be viewed as a canonical extension of the debiased regularized least squares estimator $\mu^{\mathrm{ls}}_{\mathrm{db}}$ in~(\ref{eqn:debiased_reg_est}), now situated within an iterative algorithmic framework that does not require convexity of the loss landscape.

We mention that the property (\ref{eqn:CI_len_gen_intro}) extends beyond linear regression. In fact, (\ref{eqn:CI_len_gen_intro}) also holds for logistic regression with the (mis-specified) squared loss for a different iteration-independent constant $\kappa>0$; see, e.g., Proposition \ref{prop:1bit_cs_bias_var}.

\subsection{Debiased statistical inference via gradient descent inference algorithm}

In order to use~(\ref{def:db_gd_intro}) for statistical inference of $\mu_\ast$, an essential challenge lies in obtaining data-driven estimates of the Onsager correction matrices $\bm{\tau}^{[t]},\bm{\rho}^{[t]}$ (and therefore $\bm{\omega}^{[t]}$). The difficulty stems from the fact that these matrices are defined recursively via state evolution, and their exact analytical forms are typically mathematically intractable.

We propose the \emph{gradient descent inference algorithm} (cf. Algorithm \ref{def:alg_tau_rho}) to construct consistent estimators $\hat{\bm{\tau}}^{[t]},\hat{\bm{\rho}}^{[t]}$ for $\bm{\tau}^{[t]},\bm{\rho}^{[t]}$. This algorithm can be naturally embedded in the gradient descent iterate (\ref{def:gd_intro}), which, at iteration $t$, simultaneously outputs $\hat{\bm{\tau}}^{[t]},\hat{\bm{\rho}}^{[t]}$ and $\mu^{(t)}$. At a high level, the construction of this algorithm relies on the recursive structure of the state evolution mechanism, which propagates  $\bm{\tau}^{[t]},\bm{\rho}^{[t]}$ through the chain $\bm{\tau}^{[1]} \to \bm{\rho}^{[1]}\to \cdots \to \bm{\tau}^{[t]} \to \bm{\rho}^{[t]}$. The special coupling structure in the state evolution mean-field functions allows us to efficiently construct estimators $\hat{\bm{\tau}}^{[t]},\hat{\bm{\rho}}^{[t]}$ along this chain using only $\hat{\bm{\tau}}^{[t-1]},\hat{\bm{\rho}}^{[t-1]}$ from the previous iterate.

From a conceptual standpoint, the role of our proposed gradient descent inference algorithm is analogous to that of constructing a data-driven estimate of $\omega^{\mathrm{ls}}$ for the oracle debiased regularized estimator in~(\ref{eqn:debiased_reg_est}). A well-understood methodology for the latter is the `degrees-of-freedom (DoF) adjustment' systematically studied in \cite{bellec2023debias,bellec2025observable}. We summarize this conceptual correspondence below:

\begin{table}[H]
	\centering
	\begin{tabular}{c|cc}
		& \textbf{Debiased form} & \textbf{Est. method of debias. coef.} \\
		\hline
		\textbf{Cvx. reg. est.}
		& Eqn. (\ref{eqn:debiased_reg_est})
		& DoF adjustment \cite{bellec2023debias,bellec2025observable} \\
		\textbf{Grad. descent  }
		& Eqn. (\ref{def:db_gd_intro})
		& \emph{Grad. descent inf. alg.} in Sec. \ref{section:iterative_inf} \\
	\end{tabular}
\end{table}

With $\hat{\bm{\tau}}^{[t]}$ and $\hat{\bm{\rho}}^{[t]}$ computed from the gradient descent inference algorithm, confidence intervals for the unknown signal $\mu_\ast$ may be constructed using the debiased gradient descent iterate $\hat{\mu}^{(t)}_{\mathrm{db}}$ in (\ref{def:db_gd_intro}), provided that the bias and variance parameters $b^{(t)}_{\mathrm{db}}$ and $(\sigma^{(t)}_{\mathrm{db}})^2$ can be estimated. As we will see, the variance parameter $(\sigma^{(t)}_{\mathrm{db}})^2$ can be readily estimated from the observed data, whereas the bias parameter $b^{(t)}_{\mathrm{db}}$ may depend on oracle information about $\mu_\ast$ through the information parameters $\{\delta_t\}$, and must therefore leverage model-specific structure.
	
As a by-product of our general debiasing methodology, the 	`generalization error' of $\mu^{(t)}$ can be estimated at no additional cost at each iteration $t$. Specifically, we show in Theorem \ref{thm:gen_error} that the generalization error $\mathscr{E}_{\mathsf{H}}^{(t)}$ of $\mu^{(t)}$ under a given loss function $\mathsf{H}:\R^2 \to \R$ can be estimated via:
\begin{align}\label{def:gen_err_est_intro}
\hat{\mathscr{E}}_{\mathsf{H}}^{(t)}\equiv \frac{1}{m}\sum_{k \in [m]}\mathsf{H}\,\bigg[\bigg(A\mu^{(t)}+\eta \sum_{s \in [1:t]}\hat{\rho}_{t,s}  \partial_1\mathsf{L}\big(A\mu^{(s-1)},Y\big), Y \bigg)_k\bigg]\approx \mathscr{E}_{\mathsf{H}}^{(t)}.
\end{align}
The specific form of the generalization error estimate $\hat{\mathscr{E}}_{\mathsf{H}}^{(t)}$ follows from an `inversion' of the first display in~(\ref{ineq:intro_gd_dist}), together with an identity that relates the variance of $\mathsf{Z}^{(t)}$ to the generalization error $\mathscr{E}_{\mathsf{H}}^{(t)}$; some heuristics can be found in (\ref{ineq:gen_err_est_reason}). From a practical perspective, the sequence of estimates $\{\hat{\mathscr{E}}_{\mathsf{H}}^{(t)}\}$ remains valid under a non-convex loss landscape and provides a direct criterion for determining whether the gradient descent algorithm should be early stopped. This is particularly relevant when the goal is to minimize generalization error.

It is worth noting that our debiased gradient descent iterate (\ref{def:db_gd_intro}) and the generalization error estimate (\ref{def:gen_err_est_intro}) are robust to a certain degree of model misspecification. For example, even when the data are generated from a single-index regression model (cf. Example~\ref{model:single_index}) with an unknown nonlinear link function~$\varphi_\ast$, valid inference can still be obtained using (\ref{def:db_gd_intro}) and (\ref{def:gen_err_est_intro}) designed for linear regression; the readers are referred to Appendix~\ref{subsection:loo_gen_est} for numerical validation. Fundamentally, this robustness is possible because the loss function~$\mathsf{L}$ used in the debiased gradient descent (\ref{def:db_gd_intro}) need not correctly reflect the model structure~$\mathcal{F}$ in (\ref{def:model_general}).

\subsection{Applications to the two leading examples}
We further illustrate our general theory and inference methods in the two leading examples mentioned in the beginning of Introduction:
\begin{enumerate}
	\item[(i)] In the single-index regression model in Example~\ref{model:single_index}, although the loss landscape may exhibit arbitrary non-convexity, our proposed debiased gradient descent inference remains valid upon numerical estimation of the bias parameter $b^{(t)}_{\mathrm{db}}$. In the special case of the linear model, the bias parameter $b^{(t)}_{\mathrm{db}}=1$ regardless of the loss-regularization pair $(\mathsf{L},\mathsf{f})$; this means statistical inference for $\mu_\ast$ via debiased gradient descent in the linear model is almost as easy as running the gradient descent algorithm (\ref{def:gd_intro}) itself. 
	\item[(ii)] In the generalized logistic regression model in Example~\ref{model:1bit}, valid debiased gradient descent inference can again be performed by numerically estimating the bias parameter $b^{(t)}_{\mathrm{db}}$. Interestingly, using the mis-specified squared loss for the standard logistic regression leads to major computational gains, while producing qualitatively similar confidence intervals for $\mu_\ast$ to those computed from the standard logistic/cross-entropy loss.
\end{enumerate} 
It should be mentioned that in both examples above, the quality of inference need not align exactly with the generalization error achieved along the gradient descent trajectory for general/non-convex loss functions; see, e.g., Section~\ref{sec:numerics} for several numerical experiments in this direction. While a complete theoretical understanding remains open, our framework provides a practical path forward: since both the confidence intervals based on $\hat{\mu}^{(t)}_{\mathrm{db}}$ and the generalization error estimate $\hat{\mathscr{E}}_{\mathsf{H}}^{(t)}$ remain valid at each iteration, practitioners can select learning procedures and potentially early stopping times, based on task-specific goals, such as minimizing generalization error or confidence interval length, under the design conditions considered in this paper.

\subsection{Further related literature}

\subsubsection{Comparison to existing DMFT characterizations}

	We compare (\ref{ineq:intro_gd_dist})-(\ref{def:db_gd_intro}) with existing DMFT characterizations in \cite{celentano2021high,gerbelot2024rigorous}, which typically show that, for some deterministic functions $\{\Theta_s:\mathbb{R}^{m\times [0:s]}\to \mathbb{R}^m\}_{s \in [1:t]}$ and $\{\Omega_s: \mathbb{R}^{n\times [1:s]} \to \mathbb{R}^n\}_{s \in [1:t]}$, it holds in an averaged sense that
	\begin{align}\label{ineq:intro_gd_dist_2}
	\big(A\mu^{(s-1)}\big)_{s \in [1:t]} \stackrel{d}{\approx} \big(\Theta_s(\mathsf{Z}^{([0:s])})\big)_{s \in [1:t]}, \quad 
	\big(\mu^{(s)}\big)_{s \in [1:t]} \stackrel{d}{\approx} \big(\Omega_s(\mathsf{W}^{([1:s])})\big)_{s \in [1:t]}.
	\end{align}
	For example, \cite[Lemma 6.2]{celentano2021high} provides asymptotic Gaussian characterizations of the form (\ref{ineq:intro_gd_dist_2}) for a class of gradient descent with Ridge regularization, via a reduction to the so-called Approximate Message Passing (AMP) algorithms, cf. \cite{bayati2011dynamics,bayati2015universality,berthier2020state}. Likewise, \cite[Theorem 3.2]{gerbelot2024rigorous} derives asymptotic characterizations of the form (\ref{ineq:intro_gd_dist_2}) for a class of (stochastic) gradient descent algorithms, by directly applying the Gaussian conditioning technique developed in \cite{bayati2011dynamics} for the AMP.  
	
	However, DMFT characterizations of the type~(\ref{ineq:intro_gd_dist_2}) are generally not readily applicable for statistical inference of $\mu_\ast$, since the mean-field functions $\{\Theta_s, \Omega_s\}_{s \in [1:t]}$ and the Gaussian laws $\mathsf{Z}^{([0:t])},\mathsf{W}^{([1:t])}$ typically depend on the unknown parameter $\mu_\ast$ in a highly nonlinear manner and are therefore not amenable to numerical approximation based on the observable data $\{(A_i, Y_i)\}_{i \in [m]}$ and the gradient descent trajectory $\{\mu^{(t)}\}$. 

Differently from the existing DMFT characterizations in (\ref{ineq:intro_gd_dist_2}), the debiased characterization~(\ref{ineq:intro_gd_dist}) reveals the Gaussian vectors $\mathsf{Z}^{(t)}$ and $\mathsf{W}^{(t)}$ without involving unknown nonlinear transformations. This structure enables the construction of the debiased gradient descent iterate~(\ref{def:db_gd_intro}), and reduces the challenge of statistical inference from estimating the entire state evolution to consistently estimating the debiasing coefficients $\{\omega_{t,s}\}$, which can be achieved using the proposed \emph{gradient descent inference algorithm} without knowledge of $\mu_\ast$.

From a technical perspective, since our debiased characterization \eqref{ineq:intro_gd_dist} can be viewed as a certain inversion of the DMFT theory \eqref{ineq:intro_gd_dist_2}, its formal validation therefore relies crucially on certain `stability' properties for the mean-field functions $\{\Theta_s, \Omega_s\}_{s \in [1:t]}$, the Gaussian laws $\mathsf{Z}^{([0:t])}, \mathsf{W}^{([1:t])}$, and other state evolution parameters, in addition to the theory already developed in \cite{han2025entrywise}; the readers are referred to the technical Lemma~\ref{lem:rho_tau_bound} and the subsequent proofs in Section \ref{section:proof_gd_dynamics} for mathematical details.

\subsubsection{Other mean-field theory of iterative algorithms}

We review some other literature directly related to our theory in (\ref{ineq:intro_gd_dist}). Under the squared loss without regularization, the algorithmic evolution of gradient descent (\ref{def:gd_intro}) has been analyzed directly using random matrix methods thanks to a direct reduction to the spectrum of $A^\top A$, cf. \cite{ali2019continuous,ali2020implicit}. 

In a related direction, a significant body of recent work has characterized the algorithmic dynamics of stochastic gradient descent (SGD) under the squared loss \cite{paquette2021sgd,paquette2021dynamics,balasubramanian2025high} and for more general non-convex losses \cite{benarous2021online,benarous2024high,collins2024hitting}. In particular, \cite{benarous2021online} characterizes the precise sample complexity of SGD for strong signal recovery in the regime $m \gg n$, in terms of the so-called `information exponent' associated with the single-index model function in (\ref{def:model_general}). In contrast, our work operates under the mean-field regime (\ref{def:mean_field}) where strong recovery is generally impossible, and, more importantly, studies the full gradient-descent trajectory whose sample-complexity behavior may differ qualitatively from that of SGD due to its dependence on all past iterates, cf. \cite{dandi2024benefits,lee2024neural}. We also mention that while our theory can cover some mini-batch SGD settings (where batch sizes are proportional to $m$ or $n$), the dynamics of the fully online SGD are of a different nature and fall out of the scope of our approach. 


\subsubsection{Statistical inference via gradient descent}

Statistical inference via gradient descent algorithms in the mean-field regime (\ref{def:mean_field}) was initiated in \cite{bellec2024uncertainty} in the specific linear model under the squared loss, and has been further extended to general losses in \cite{tan2024estimating}. In contrast, our generic inference methods in (\ref{def:db_gd_intro}) and (\ref{def:gen_err_est_intro}) are broadly applicable to the general class of models in (\ref{def:model_general}) with possibly non-convex losses.

In a different direction, statistical inference is studied for stochastic gradient descent (SGD) in convex problems under (effectively) low-dimensional settings. A key approach involves using averaged SGD iterates which are known to obey a normal limiting law \cite{ruppert1988efficient,polyak1992acceleration}. Inference is then feasible once the limiting covariance is accurately estimated. We refer the readers to \cite{fang2018online,chen2020statistical,zhu2023online} for several recent proposals along this line; much more references can be found therein. It remains open to extend our inference methods (\ref{def:db_gd_intro}) and (\ref{def:gen_err_est_intro}) to the fully online SGD setting in the mean-field regime (\ref{def:mean_field}).

\subsection{Organization}

The rest of the paper is organized as follows. In Section~\ref{section:gd_dynamics}, we formalize a more comprehensive version of theory~(\ref{ineq:intro_gd_dist}) in the mean-field regime~(\ref{def:mean_field}), and show how it leads to the construction of the oracle debiased gradient descent iterate~(\ref{def:db_gd_intro}) and an oracle version of~(\ref{def:gen_err_est_intro}). Section~\ref{section:iterative_inf} presents our gradient descent inference algorithm for computing the estimates $\hat{\bm{\tau}}^{[t]}$ and $\hat{\bm{\rho}}^{[t]}$, and describes the resulting fully data-driven inference procedures. Applications of our theory and inference method to the single-index regression and generalized logistic regression models are provided in Sections~\ref{section:example_linear} and~\ref{section:example_1bit}, respectively. Numerical experiments for our debiased gradient descent inference proposal in both convex and non-convex settings are presented in Section~\ref{sec:numerics}, with additional simulation results provided in Appendix~\ref{section:additional_simulation}. All technical proofs are deferred to Sections~\ref{section:proof_gd_dynamics}-\ref{section:proof_logistic} and Appendices~\ref{section:GFOM_se}-\ref{section:aux_result}.

\subsection{Notation}
For any two integers $m,n$, let $[m:n]\equiv \{m,m+1,\ldots,n\}$. We sometimes write for notational convenience $[n]\equiv [1:n]$. When $m>n$, it is understood that $[m:n]=\emptyset$.  

For $a,b \in \R$, $a\vee b\equiv \max\{a,b\}$ and $a\wedge b\equiv\min\{a,b\}$. For $a \in \R$, let $a_\pm \equiv (\pm a)\vee 0$. For a multi-index $a \in \mathbb{Z}_{\geq 0}^n$, let $\abs{a}\equiv \sum_{i \in [n]}a_i$. For $x \in \R^n$, let $\pnorm{x}{p}$ denote its $p$-norm $(0\leq p\leq \infty)$, and $B_{n;p}(R)\equiv \{x \in \R^n: \pnorm{x}{p}\leq R\}$. We simply write $\pnorm{x}{}\equiv\pnorm{x}{2}$ and $B_n(R)\equiv B_{n;2}(R)$. For $x \in \R^n$, let $\mathrm{diag}(x)\equiv (x_i\bm{1}_{i=j})_{i,j \in [n]} \in \R^{n\times n}$. For $x, y \in \R^n$, let $x\odot y \equiv (x_i y_i)_{i \in [n]}$.

For a matrix $M \in \R^{m\times n}$, let $\pnorm{M}{\op},\pnorm{M}{F}$ denote the spectral and Frobenius norm of $M$, respectively. $I_n$ is reserved for an $n\times n$ identity matrix, written simply as $I$ (in the proofs) if no confusion arises. For a general $n\times n$ matrix $M$, let
\begin{align}\label{def:mat_O}
\mathfrak{O}_{n+1}(M)\equiv 
\begin{pmatrix}
0_{1\times n} & 0\\
M & 0_{n\times 1}
\end{pmatrix} \in \R^{(n+1)\times (n+1)}.
\end{align}
For notational consistency, we write $\mathfrak{O}_1(\emptyset)=0$.

We use $C_{x}$ to denote a generic constant that depends only on $x$, whose numeric value may change from line to line unless otherwise specified. $a\lesssim_{x} b$ and $a\gtrsim_x b$ mean $a\leq C_x b$ and $a\geq C_x b$, abbreviated as $a=\bigo_x(b), a=\Omega_x(b)$ respectively;  $a\asymp_x b$ means $a\lesssim_{x} b$ and $a\gtrsim_x b$. $\bigo$ and $\smallo$ (resp. $\mathcal{O}_{\mathbf{P}}$ and $\mathfrak{o}_{\mathbf{P}}$) denote the usual big and small O notation (resp. in probability). By convention, sum and product over an empty set are understood as $\Sigma_{\emptyset}(\cdots)=0$ and $\Pi_{\emptyset}(\cdots)=1$. 

For a random variable $X$, we use $\Prob_X,\E_X$ (resp. $\Prob^X,\E^X$) to indicate that the probability and expectation are taken with respect to $X$ (resp. conditional on $X$).

For $\Lambda>0$ and $\mathfrak{p}\in \N$, a measurable map $f:\R^n \to \R$ is called \emph{$\Lambda$-pseudo-Lipschitz of order $\mathfrak{p}$} iff 
\begin{align}\label{cond:pseudo_lip}
\abs{f(x)-f(y)}\leq \Lambda\cdot  (1+\pnorm{x}{}+\pnorm{y}{})^{\mathfrak{p}-1}\cdot\pnorm{x-y}{},\quad \forall x,y \in \R^{n}.
\end{align}
Moreover, $f$ is called \emph{$\Lambda$-Lipschitz} iff $f$ is $\Lambda$-pseudo-Lipschitz of order $1$, and in this case we often write $\pnorm{f}{\mathrm{Lip}}\leq L$, where $\pnorm{f}{\mathrm{Lip}}\equiv \sup_{x\neq y} \abs{f(x)-f(y)}/\pnorm{x-y}{}$. For a proper, closed convex function $f$ defined on $\R^n$, its \emph{proximal operator} $\prox_f(\cdot)$  is defined by 
\begin{align}\label{def:prox_op}
\prox_f(x)\equiv \argmin_{z \in \R^n} \big\{\pnorm{x-z}{}^2/2+f(z)\big\}.
\end{align}

\section{Mean-field dynamics of (debiased) gradient descent}\label{section:gd_dynamics}

\subsection{Basic setups and assumptions}

We consider a class of generalized proximal gradient descent algorithms: starting from an initialization $\mu^{(0)} \in \R^n$, for $t=1,2,\ldots$ and using step sizes $\{\eta_t\}\subset \R_{>0}$, the iterates are computed as follows:
\begin{align}\label{def:grad_descent}
\mu^{(t)} = \mathsf{P}_{t}\big( \mu^{(t-1)}-\eta_{t-1}\cdot A^\top \partial_1 \mathsf{L}_{t-1} (A\mu^{(t-1)},Y)\big).
\end{align}
Here $\mathsf{P}_{t}:\R^n\to \R^n$ and $\mathsf{L}_{t-1}:\R^m\to \R^m$ are row-separable functions, and  $\partial_1 \mathsf{L}_{t-1}(x,y)\equiv (\partial \mathsf{L}_{t-1}/\partial x)(x,y)$ is understood as applied row-wise. 

For the canonical proximal gradient descent method, we may take $\mathsf{P}_{t}\equiv \prox_{\eta_{t-1} \mathsf{f}}$. The  generalization to iteration-dependent and vector-valued loss functions $\mathsf{L}_{t-1}$ can naturally accommodate other variants such as stochastic gradient descent (SGD). For example, in SGD, only a subsample $S_t \subset [m]$ is used at iteration $t$, so we may take $\mathsf{L}_{t-1,\cdot}(u)\equiv \mathsf{L}(u)\circ \bm{1}_{\cdot \in S_t}$. Importantly, at the level of our abstract theory, we do not assume that $\mathsf{P}_t$ is the proximal operator of a convex function, nor do we assume the convexity of $\mathsf{L}_{t-1}$. 

We list a set of common assumptions that will be used throughout the paper:
\begin{assumption}\label{assump:setup}
	Suppose the following hold for some $K,\Lambda\geq 2$:
	\begin{enumerate}
		\item[(A1)] The aspect ratio $\phi \equiv m/n \in [1/K,K]$.
		\item[(A2)] The matrix $A\equiv A_0/\sqrt{n}$, where the entries of $A_0\in \R^{m\times n}$ are independent mean $0$, unit variance variables such that\footnote{Here $\pnorm{\cdot}{\psi_2}$ is the standard Orlicz-2/subgaussian norm; see, e.g., \cite[Section 2.1]{van1996weak} for a precise definition.} $\max_{i,j \in [n]}\pnorm{A_{0,ij}}{\psi_2}\leq K$.
		\item [(A3)] The step sizes satisfy $\max_{s \in [0:t-1]} \eta_{s}\leq \Lambda$ at iteration $t$.
	\end{enumerate}
\end{assumption}

Assumption (A1) formalizes the proportional/mean-field regime (\ref{def:mean_field}) of our main interest here. Assumption (A2) requires that the design matrix $A$ is normalized with entries having variance $1/n$. If the variance is instead normalized as $1/m$, the state evolution below need be adjusted accordingly. Notably, (A2) does not require $A$ to be Gaussian. This means that our results hold universally for all random matrix models satisfying (A2). Assumption (A3) imposes a mild constraint on the magnitude of the step sizes. Here we use the constant $\Lambda$ (rather than $K$) for conditions on the gradient descent algorithm (\ref{def:grad_descent}).

\subsection{State evolution}

The state evolution for describing the mean-field behavior of the gradient descent iterate $\{\mu^{(t)}\}$ consists of three major components:

\begin{enumerate}
	\item A sequence of functions $\{\Upsilon_t: \mathbb{R}^{m\times [0:t]}\to \R^m\}$, a Gaussian law $\mathfrak{Z}^{([0:\infty))}\in \R^{[0:\infty)}$ that describes the distributions of (a transform of) $\{A \mu^{(t)}\}$, and a matrix $\bm{\rho}\in \R^{\infty\times \infty}$ that characterizes inter-correlation between $\{A\mu^{(t)}\}$.	
	\item A sequence of functions $\{\Omega_t: \mathbb{R}^{n\times [1:t]}\to \R^n\}$, a Gaussian law $\mathfrak{W}^{([1:\infty))}\in \R^{[1:\infty)}$ that describe the distributions of (a transform of) $\{\mu^{(t)}\}$, and a matrix $\bm{\tau}\in \R^{\infty\times \infty}$ that characterizes inter-correlation between $\{\mu^{(t)}\}$.	
	\item An information vector $\bm{\delta}\in \R^{[0:\infty)}$ that characterizes the amount of information about the true signal $\mu_\ast$ contained in the gradient descent iterates $\{\mu^{(t)}\}$.
\end{enumerate}

To formally describe the recursive relation for these components, we need some additional notation:
\begin{itemize}
	\item Let $
	\sigma_{\mu_\ast}^2 \equiv \pnorm{\mu_\ast}{}^2/n$
	be the signal strength.
	\item Let $\pi_m$ (resp. $\pi_n$) denote the uniform distribution on $[1:m]$ (resp. $[1:n]$), independent of all other variables. 
	\item In our probabilistic statements, we usually treat the true signal $\mu_\ast$, the initialization $\mu^{(0)}$ and the noise $\xi$ as fixed, and use $
	\E^{(0)}[\cdot]\equiv \E[\cdot |\mu_\ast,\mu^{(0)},\xi]$
	to denote the expectation over all other sources. 
\end{itemize}

\begin{definition}\label{def:gd_se}
	Initialize with (i) two formal variables $\Omega_{-1}\equiv \mu_\ast \in \R^n$ and $\Omega_0 \equiv \mu^{(0)}\in \R^n$, and (ii) a Gaussian random variable $\mathfrak{Z}^{(0)}\sim \mathcal{N}(0,\sigma_{\mu_\ast}^2)$. For $t=1,2,\ldots$, we execute the following steps:
	\begin{enumerate}
		\item[(S1)] Let $\Upsilon_t: \mathbb{R}^{m\times [0:t]}\to \R^m$ be defined as follows: 
		\begin{align*}
		\Upsilon_t(\mathfrak{z}^{([0:t])})\equiv - \eta_{t-1}  \partial_1 \mathsf{L}_{t-1}\bigg(\mathfrak{z}^{(t)} + \sum_{s \in [1:t-1]}\rho_{t-1,s}\Upsilon_s(\mathfrak{z}^{([0:s])}) ,\mathcal{F}(\mathfrak{z}^{(0)},\xi)\bigg)\in \R^m.
		\end{align*}
		Here the coefficients are defined via
		\begin{align*}
		\rho_{t-1,s}\equiv \E^{(0)} \partial_{ \mathfrak{W}^{(s)}} \Omega_{t-1;\pi_n}(\mathfrak{W}^{([1:t-1])})\in \R,\quad s \in [1:t-1].
		\end{align*}
		\item[(S2)] Let $\mathfrak{Z}^{([0:t])}\in \R^{[0:t]}$ and $\mathfrak{W}^{([1:t])}\in \R^{[1:t]}$ be centered Gaussian random vectors whose laws at iteration $t$ are determined via the correlation specification: 
		\begin{align*}
		\cov(\mathfrak{Z}^{(t)},\mathfrak{Z}^{(s)})
		& \equiv \E^{(0)} \prod\limits_{\ast \in \{s-1,t-1\}} \Omega_{*;\pi_n} (\mathfrak{W}^{([1:*])}),\quad s \in [0:t];\\
		\cov(\mathfrak{W}^{(t)},\mathfrak{W}^{(s)})
		& \equiv \phi\cdot  \E^{(0)} \prod_{\ast \in \{s,t\}} \Upsilon_{*;\pi_m}(\mathfrak{Z}^{([0:*])}),\quad s \in [1:t].
		\end{align*}
		\item[(S3)] Let $\Omega_t: \R^{n\times [1:t]}\to \R^n$ be defined as follows:
		\begin{align*}
		\Omega_t\big(\mathfrak{w}^{([1:t])}\big)\equiv \mathsf{P}_t\bigg(\mathfrak{w}^{(t)}+\sum_{s \in [1:t]} (\tau_{t,s}+\bm{1}_{t=s})\cdot  \Omega_{s-1}(\mathfrak{w}^{([1:s-1])}) + \delta_t\cdot \mu_\ast\bigg).
		\end{align*}
		Here the coefficients are defined via
		\begin{align*}
		\tau_{t,s} &\equiv \phi\cdot \E^{(0)}\partial_{\mathfrak{Z}^{(s)}} \Upsilon_{t;\pi_m}(\mathfrak{Z}^{([0:t])})\in \R,\quad s \in [1:t];\\
		\delta_t &\equiv \phi \cdot \E^{(0)}\partial_{\mathfrak{Z}^{(0)}} \Upsilon_{t;\pi_m}(\mathfrak{Z}^{([0:t])}) \in \R.
		\end{align*}
	\end{enumerate}
\end{definition}

For notational convenience, we shall sometimes write $\Sigma_{\mathfrak{Z}}^{[t]} \in \R^{[0:t]\times [0:t]}$ for the covariance of  $\mathfrak{Z}^{([0:t])}$ and $\Sigma_{\mathfrak{W}}^{[t]} \in \R^{[1:t]\times [1:t]}$ for the covariance of  $\mathfrak{W}^{([1:t])}$.

\begin{remark}
Some technical and notational remarks:
\begin{enumerate}
	\item We use lowercase $\mathfrak{z}$ and $\mathfrak{w}$ in the definitions of the mean-field functions $\Upsilon_\cdot$ in (S1) and $\Omega_\cdot$ in (S3), and uppercase $\mathfrak{Z}$ and $\mathfrak{W}$ for the corresponding Gaussian random vectors.
	\item Since $\Upsilon_{t;k}(\mathfrak{z}^{([0:t])})$ depends on $\mathfrak{z}^{([0:t])}$ only through its $k$-th row $\mathfrak{z}_{k\cdot}^{([0:t])}$, we identify $\Upsilon_{t;k}$ as a mapping from $\mathbb{R}^{[0:t]}$ to $\mathbb{R}$. A similar convention applies to $\Omega_{t;\ell}$.
	\item In Definition \ref{def:gd_se} above, we have not specified the precise conditions on the regularity of the loss functions $\{\mathsf{L}_\cdot\}$, the model function $\mathcal{F}$ and the (proximal) operators $\{\mathsf{P}_\cdot\}$. The precise conditions will be specified in the theorems ahead. 
\end{enumerate}
\end{remark}

\subsubsection{Onsager correction matrices}
For any $t\geq 1$, let
\begin{align}\label{def:tau_rho_mat}
\bm{\tau}^{[t]}\equiv (\tau_{r,s})_{r,s \in [t]}\in \R^{t\times t},\quad \bm{\rho}^{[t]}\equiv (\rho_{r,s})_{r,s \in [t]}\in \R^{t\times t}.
\end{align}
Both $\bm{\tau}^{[t]}$ and $\bm{\rho}^{[t]}$ are lower triangular matrices. As will be clear below, these matrices play a crucial role in describing the interactions across the iterates for $\{A\mu^{(t)}\}$ and $\{\mu^{(t)}\}$. Following terminology from the AMP literature \cite{bayati2011dynamics,javanmard2013state,bayati2015universality,berthier2020state,fan2022approximate,bao2025leave}, we refer to $\bm{\tau}^{[t]}$ and $\bm{\rho}^{[t]}$  as \emph{Onsager correction matrices}, inspired by the `Onsager correction coefficients' used to describe how the current AMP iterate restores approximate normality using previous iterates. The main difference is that under the random matrix model (A2), the current AMP iterate may restore normality using only the two most recent iterations, whereas in gradient descent, the dependency is global spanning all past iterations.

\subsubsection{Alternative formulations for $\{\Upsilon_\cdot\}$ and $\{\Omega_\cdot\}$}

As we will prove below, the state evolution in Definition \ref{def:gd_se} accurately captures the behavior of $\{\partial_1 \mathsf{L}_{t-1}(A\mu^{(t-1)},Y)\}$ and $\{\mu^{(t)}\}$ in the sense that
\begin{align}
\big(-\eta_{t-1}\partial_1 \mathsf{L}_{t-1}(A\mu^{(t-1)},Y)\big)\stackrel{d}{\approx} \big(\Upsilon_t(\mathfrak{Z}^{([0:t])})\big),\quad \big(\mu^{(t)}\big) \stackrel{d}{\approx} \big(\Omega_t(\mathfrak{W}^{([1:t])})\big).
\end{align}
In order to describe the behavior of $\{A\mu^{(t)}\}$ and $\{A^\top \partial_1 \mathsf{L}_t (A\mu^{(t)},Y)\}$, it is also convenient to work with some equivalent transformations of $\{\Upsilon_\cdot\}$ and $\{\Omega_\cdot\}$.
\begin{itemize}
	\item Let $\Theta_t: \mathbb{R}^{m\times [0:t]}\to \R^m$ be defined recursively via
	\begin{align}\label{eqn:Theta_recur}
	\Theta_t(\mathfrak{z}^{([0:t])})\equiv \mathfrak{z}^{(t)}- \sum_{s \in [1:t-1]}\eta_{s-1} \rho_{t-1,s} \cdot   \partial_1 \mathsf{L}_{s-1}\big(\Theta_{s}(\mathfrak{z}^{([0:s])}),\mathcal{F}(\mathfrak{z}^{(0)},\xi)\big).
	\end{align}
	The functions $\{\Theta_t\}$ and $\{\Upsilon_t\}$ are equivalent via the relation
	\begin{align}\label{def:Theta_fcn}
	\begin{cases}
	\Theta_t(\mathfrak{z}^{([0:t])})= \mathfrak{z}^{(t)} + \sum_{s \in [1:t-1]}\rho_{t-1,s} \Upsilon_s(\mathfrak{z}^{([0:s])}),\\
	\Upsilon_t(\mathfrak{z}^{([0:t])})= - \eta_{t-1}  \partial_1 \mathsf{L}_{t-1}\big(\Theta_t(\mathfrak{z}^{([0:t])}) ,\mathcal{F}(\mathfrak{z}^{(0)},\xi)\big).
	\end{cases}
	\end{align}
	With $\{\Theta_t\}$, we may describe the behavior of $\{A\mu^{(t-1)}\}$ via
	\begin{align}\label{eqn:Theta_dist_heuristic}
	\big(A\mu^{(t-1)}\big) \stackrel{d}{\approx} \big(\Theta_t(\mathfrak{Z}^{([0:t])})\big).
	\end{align}
	\item Let $\Delta_t: \R^{n\times [1:t]}\to \R^n$ be defined recursively as follows:
	\begin{align}\label{def:Delta_fcn}
	\Delta_t\big(\mathfrak{w}^{([1:t])}\big)\equiv \mathfrak{w}^{(t)}+\sum_{s \in [1:t]} (\tau_{t,s}+\bm{1}_{t=s})\cdot  \mathsf{P}_{s-1}\big(\Delta_{s-1}(\mathfrak{w}^{([1:s-1])})\big) + \delta_t\cdot \mu_\ast.
	\end{align}
	The functions $\{\Delta_t\}$ and $\{\Omega_t\}$ are equivalent via the relation
	\begin{align}
	\begin{cases}
	\Delta_t\big(\mathfrak{w}^{([1:t])}\big)= \mathfrak{w}^{(t)}+\sum_{s \in [1:t]} (\tau_{t,s}+\bm{1}_{t=s})\cdot  \Omega_{s-1}(\mathfrak{w}^{([1:s-1])}) + \delta_t\cdot \mu_\ast, \\
	\Omega_t\big(\mathfrak{w}^{([1:t])}\big)= \mathsf{P}_t\big(\Delta_t\big(\mathfrak{w}^{([1:t])}\big)\big).
	\end{cases}
	\end{align}
	With $\{\Delta_t\}$, we may describe 
	\begin{align}\label{eqn:Delta_dist_heuristic}
	\big(\mu^{(t-1)}-\eta_{t-1}\cdot A^\top \partial_1 \mathsf{L}_{t-1} (A\mu^{(t-1)},Y) \big) \stackrel{d}{\approx} \big(\Delta_t(\mathfrak{W}^{([1:t])})\big).
	\end{align}
\end{itemize}

\subsubsection{Alternative definition of the information parameter $\delta_t$}
In some examples, the quantity $\E^{(0)}\partial_{\mathfrak{Z}^{(0)}} \Upsilon_{t;\pi_m}(\mathfrak{Z}^{([0:t])})$ may not be well-defined due to the strict non-differentiability of $\mathcal{F}$. In such cases, we shall interpret the definition of $\delta_t$ via the Gaussian integration-by-parts formula:
\begin{align}\label{def:delta_t_alternative}
\delta_t \equiv \frac{\phi}{ \sigma_{\mu_\ast}^2}\bigg(\E^{(0)} \mathfrak{Z}^{(0)} \Upsilon_{t;\pi_m}(\mathfrak{Z}^{([0:t])})-\phi^{-1}\sum_{s \in [1:t]} \tau_{t,s}\cov(\mathfrak{Z}^{(0)},\mathfrak{Z}^{(s)}) \bigg).
\end{align}
Here for $\mu_\ast=0$, the right hand side is interpreted as the limit as $\pnorm{\mu_\ast}{}\to 0$ whenever well-defined. It is easy to check for regular enough $\{(u_1,u_2)\mapsto \E_{\pi_m}\mathsf{L}_{s-1}(u_1,\mathcal{F}(u_2,\xi_{\pi_m}))\}_{s \in [1:t-1]}$, the above definition of $\delta_t$ coincides with Definition \ref{def:gd_se}.

\subsection{Distributional characterizations for (debiased) gradient descent}

With the state evolution in Definition \ref{def:gd_se}, we shall now formally describe the joint distributional behavior of $\{A\mu^{(t)}\}$ and $\{\mu^{(t)}\}$ with their debiased versions, defined as
\begin{align}\label{def:Z_W}
{Z}^{(t)}&\equiv A\mu^{(t)}+\sum_{s \in [1:t]} \eta_{s-1} \rho_{t,s} \cdot \partial_1\mathsf{L}_{s-1}\big(A\mu^{(s-1)},Y\big) \in \R^m, \nonumber\\
{W}^{(t)}&\equiv -\delta_t\cdot \mu_{\ast}-\sum_{s \in [1:t]} \tau_{t,s}\cdot  \mu^{(s-1)} -\eta_{t-1}\cdot  A^\top\partial_1 \mathsf{L}_{t-1}(A \mu^{(t-1)},Y)\in \R^n.
\end{align}
The form of $Z^{(t)}$ is motivated by plugging the heuristic (\ref{eqn:Theta_dist_heuristic}) into (\ref{eqn:Theta_recur}), while the form of $W^{(t)}$ is informed by similarly plugging the heuristic (\ref{eqn:Delta_dist_heuristic}) into (\ref{def:Delta_fcn}).

For notational convenience, we define the constant
\begin{align}\label{def:L_mu}
L_\mu\equiv 1+\pnorm{\mu^{(0)}}{\infty}+\pnorm{\mu_\ast}{\infty}.
\end{align}

\begin{theorem}\label{thm:gd_se}
Suppose Assumption \ref{assump:setup} holds for some $K,\Lambda \geq 2$. 
\begin{enumerate}
	\item (\textbf{Entrywise characterization}). Further suppose that for $s \in [0:t-1]$:
	\begin{enumerate}
		\item[(A4)] For all $\ell \in [n]$, 
		$\mathsf{P}_{s+1;\ell}\in C^3(\R)$ and $\abs{\mathsf{P}_{s+1;\ell}(0)}\vee \max_{q\in [1:3]}\pnorm{\mathsf{P}_{s+1;\ell}^{(q)}}{\infty}\leq \Lambda$.
		\item[(A5)] Both $\pnorm{\partial_1 \mathsf{L}_s\big(0, \mathcal{F}(0,\xi)\big)}{\infty}$ and 
		\begin{align*}
		& \max_{k \in [m]}\sup_{u_1,u_2 \in \R} \max_{0\neq \alpha \in \mathbb{Z}_{\geq 0}^2: \abs{\alpha}\leq 3} \biggabs{\frac{\partial^\alpha}{\partial u_1^{\alpha_1}\partial u_2^{\alpha_2}} \partial_1 \mathsf{L}_{s;k}\big(u_1, \mathcal{F}(u_2,\xi_k)\big) }
		\end{align*}
		are bounded by $\Lambda$.
	\end{enumerate} 
	Fix any test function $\Psi \in C^3(\R^{2t+1})$ such that for some $\Lambda_{\Psi}\geq 2$ and $\mathfrak{p}\in \N$,
	\begin{align}\label{cond:test_fcn}
	\max_{\alpha\in \mathbb{Z}_{\geq 0}^{2t+1}: \abs{\alpha}\leq 3}\sup_{x\in \R^{2t+1}} \big(1+\pnorm{x}{}\big)^{-\mathfrak{p}}\cdot \abs{\partial_{\alpha} \Psi(x)}\leq \Lambda_{\Psi}.
	\end{align}
	Then there exists some $c_t\equiv c_t(t,\mathfrak{p})>1$ such that
	\begin{align*}
	&\max_{k \in [m]} \bigabs{\E^{(0)} \Psi\big(\big\{(A \mu^{(s-1)})_k, {Z}_k^{(s-1)}\big\}, (A \mu_\ast)_k\big)-\E^{(0)} \Psi\big(\big\{\Theta_{s;k}(\mathfrak{Z}^{([0:s])}), \mathfrak{Z}^{(s)}\big\}, \mathfrak{Z}^{(0)}\big)}\\
	&\quad \vee \max_{\ell \in [n]} \bigabs{ \E^{(0)} \Psi\big(\big\{\mu^{(s)}_\ell, {W}_\ell^{(s)}\big\},\mu_{\ast,\ell}\big)-\E^{(0)} \Psi\big(\big\{\Omega_{s;\ell}(\mathfrak{W}^{([1:s])}),\mathfrak{W}^{(s)}\big\},\mu_{\ast,\ell}\big)}\\
	& \leq \big(K\Lambda \Lambda_\Psi  L_\mu \big)^{c_t} \cdot n^{-1/c_t}.
	\end{align*}
	\item  (\textbf{Averaged characterization}). Further suppose that for $s \in [0:t-1]$:
	\begin{enumerate}
		\item[(A4')] $\max_{\ell \in [n]} \big\{\abs{\mathsf{P}_{s+1;\ell}(0)} \vee \pnorm{\mathsf{P}_{s+1;\ell}}{\mathrm{Lip}}\big\}\leq \Lambda$.
		\item[(A5')] $\max_{k \in [m]}\big\{ \abs{\partial_1 \mathsf{L}_{s;k}\big(0, \mathcal{F}(0,\xi_k)\big)} \vee  \pnorm{\partial_1 \mathsf{L}_{s;k}\big(\cdot, \mathcal{F}(\cdot,\xi_k)}{\mathrm{Lip}} \big\} \leq \Lambda$.
	\end{enumerate}
	Fix a sequence of $\Lambda_\psi$-pseudo-Lipschitz functions $\{\psi_k:\R^{2t+1} \to \R\}_{k \in [m\vee n]}$ of order $\mathfrak{p}$, where $\Lambda_\psi\geq 2$. Then for any $q \in \N$, there exists some constant $c_t'=c_t'(t,\mathfrak{p},q)>1$ such that
	\begin{align*}
	&\E^{(0)} \biggabs{\frac{1}{m}\sum_{k \in [m]}\Big( \psi_{k}\big(\big\{(A \mu^{(s-1)})_{k}, {Z}_{k}^{(s-1)}\big\}, (A \mu_\ast)_{k}\big)- \E^{(0)}\psi_{k}\big(\big\{\Theta_{s;k}(\mathfrak{Z}^{([0:s])}), \mathfrak{Z}^{(s)}\big\}, \mathfrak{Z}^{(0)}\big) \Big) }^{q}\\
	&\quad \vee \E^{(0)}\biggabs{ \frac{1}{n}\sum_{\ell \in [n]} \Big(\psi_{\ell}\big(\big\{\mu^{(s)}_{\ell}, {W}_{\ell}^{(s)}\big\},\mu_{\ast,\ell}\big)- \E^{(0)} \psi_{\ell}\big(\big\{\Omega_{s;\ell}(\mathfrak{W}^{([1:s])}),\mathfrak{W}^{(s)}\big\},\mu_{\ast,\ell}\big)\Big) }^{q}\\
	& \leq \big(K\Lambda \Lambda_\psi L_\mu\big)^{c_t'} \cdot n^{-1/c_t'}.
	\end{align*}
\end{enumerate}
In both error estimates, the index $s$ in the brackets all run over $s \in [1:t]$.
\end{theorem}

As mentioned in the Introduction, averaged distributional characterizations for $\{A\mu^{(t)}\}$ and $\{\mu^{(t)}\}$ are known for other variants of gradient descent algorithms and fall within the standard DMFT framework in a technically weaker asymptotic form or under Gaussian data assumptions or both, cf. \cite{celentano2021high,gerbelot2024rigorous}. Our Theorem~\ref{thm:gd_se} here provides a stronger, non-asymptotic, entrywise distributional characterization for $\{A\mu^{(t)}\}$ and $\{\mu^{(t)}\}$ that holds under non-Gaussian data.

\subsection{Oracle debiased gradient descent iterates}
The main advantage of Theorem~\ref{thm:gd_se} over the standard DMFT formalism lies in its utility for statistical inference applications via a joint distributional characterization involving the debiased statistics $Z^{(t)}$ and $W^{(t)}$.

To this end, we first consider $W^{(t)}$. Let
\begin{itemize}
	\item $\bm{\delta}^{[t]}\equiv (\delta_s)_{s \in [t]} \in \R^t$,
	\item  $\bm{W}^{[t]} \equiv  (W^{(1)},\ldots,W^{(t)}) \in \R^{n\times t}$,
	\item $\bm{\mu}^{[t]} \equiv  (\mu^{(0)},\ldots,\mu^{(t-1)}) \in \R^{n\times t}$, 
	\item $\bm{g}^{[t]}\equiv \big(\eta_{s-1}A^\top \partial_1 \mathsf{L}_{s-1}(A\mu^{(s-1)},Y)\big)_{s \in [1:t]}\in \R^{n\times t}$. 
\end{itemize} 
Using these notation, the second line of (\ref{def:Z_W}) can be rewritten as
\begin{align*}
\bm{W}^{[t]} = -  \mu_{\ast}\bm{\delta}^{[t],\top} - \bm{\mu}^{[t]}\bm{\tau}^{[t],\top}  - \bm{g}^{[t]}.
\end{align*}
Suppose the lower triangular matrix $\bm{\tau}^{[t]}  $ is invertible (i.e., when $\tau_{rr}\neq 0$ for $r \in [1:t]$). Then from the above display, with $\bm{\omega}^{[t]} \equiv (\bm{\tau}^{[t]} )^{-1}$,
\begin{align*}
\bm{\mu}^{[t]} + \bm{g}^{[t]}\bm{\omega}^{[t],\top}  &= - \mu_{\ast}  \big(\bm{\omega}^{[t]} \bm{\delta}^{[t]}\big)^\top - \bm{W}^{[t]} \bm{\omega}^{[t],\top}.  
\end{align*}
On the right hand side of the display above, the first term $- \mu_{\ast}  \big(\bm{\omega}^{[t]} \bm{\delta}^{[t]}\big)^\top$ contains the information for the true signal $\mu_\ast$ up to a multiplicative factor, whereas the second term is approximately Gaussian as $\bm{W}^{[t]}$ is.

This motivates us to define the (oracle) debiased gradient descent iterate
\begin{align}\label{def:debias_gd_oracle}
\mu^{(t)}_{\mathrm{db}}&\equiv \big(\bm{\mu}^{[t]} + \bm{g}^{[t]}\bm{\omega}^{[t],\top} \big)_t=\mu^{(t-1)}+\sum_{s \in [1:t]} \omega_{t,s}\cdot \eta_{s-1}A^\top \partial_1 \mathsf{L}_{s-1}(A\mu^{(s-1)},Y).
\end{align}
By setting 
\begin{align}\label{def:db_gd_bias_var}
b^{(t)}_{\mathrm{db}}\equiv - \iprod{\bm{\omega}^{[t]} \bm{\delta}^{[t]}}{e_t},\quad \sigma^{(t)}_{\mathrm{db}}\equiv \pnorm{\Sigma_{\mathfrak{W}}^{[t],1/2} \bm{\omega}_{t\cdot}^{[t],\top} }{},
\end{align} 
where recall $\Sigma_{\mathfrak{W}}^{[t]}= \big(\cov(\mathfrak{W}^{(r)},\mathfrak{W}^{(s)})\big)_{r,s\in [t]}$, we now expect that
\begin{align}\label{eqn:gd_db_normality}
\mu^{(t)}_{\mathrm{db}}\stackrel{d}{\approx} \mathcal{N}\big(b^{(t)}_{\mathrm{db}}\cdot \mu_\ast,(\sigma^{(t)}_{\mathrm{db}})^2\cdot I_n\big).
\end{align}
To formalize this heuristic, we need the quantity
\begin{align}\label{def:tau_ast}
\tau_\ast^{(t)} \equiv \min_{s \in [1:t]} \abs{\tau_{s,s}}=\min_{s \in [1:t]}\bigabs{\eta_{s-1}  \E^{(0)} \partial_{11} \mathsf{L}_{s-1;\pi_m}\big(\Theta_{s;\pi_m}(\mathfrak{Z}^{([0:s])}),\mathcal{F}(\mathfrak{Z}^{(0)},\xi_{\pi_m})\big)},
\end{align}
where the second identity is a consequence of Lemma \ref{lem:tau_diag}.

\begin{theorem}\label{thm:db_gd_oracle}
	The following hold:
	\begin{enumerate}
		\item Under the assumptions in Theorem \ref{thm:gd_se}-(1), fix any test function $\psi \in C^3(\R)$ such that $
		\max_{q\in [0:3]}\sup_{x \in \R }(1+\abs{x})^{-\mathfrak{p}} \abs{\psi^{(q)}(x)}\leq \Lambda_\psi$
		holds for some $\Lambda_\psi\geq 2$ and $\mathfrak{p}\in \N$. Then there exists some $c_t=c_t(t,\mathfrak{p})>1$ such that
		\begin{align*}
		&\max_{\ell \in [n]}\bigabs{\E^{(0)} \psi \big(\mu^{(t)}_{\mathrm{db};\ell}\big)- \E^{(0)} \psi\big(b^{(t)}_{\mathrm{db}}\cdot \mu_{\ast,\ell}+ \sigma^{(t)}_{\mathrm{db}}\mathsf{Z}\big)}\\
		&\leq   \big((1\wedge \tau_\ast^{(t)})^{-1} K\Lambda \Lambda_\psi L_\mu \big)^{c_t} \cdot n^{-1/c_t}.
		\end{align*}
		\item Under the assumptions in Theorem \ref{thm:gd_se}-(2), fix a sequence of $\Lambda_\psi$-pseudo-Lipschitz functions $\{\psi_\ell:\R \to \R\}_{\ell \in [n]}$ of order $\mathfrak{p}$ where $\Lambda_\psi\geq 2$. Then	for any $q \in \N$, there exists some $c_t'=c_t'(t,\mathfrak{p},q)>1$ such that
		\begin{align*}
		&\E^{(0)} \bigabs{\E_{\pi_n} \psi_{\pi_n} \big(\mu^{(t)}_{\mathrm{db};\pi_n}\big)- \E^{(0)} \psi_{\pi_n}\big(b^{(t)}_{\mathrm{db}}\cdot \mu_{\ast,\pi_n}+ \sigma^{(t)}_{\mathrm{db}}\mathsf{Z}\big)}^q\\
		&\leq \big((1\wedge \tau_\ast^{(t)})^{-1} K\Lambda \Lambda_\psi L_\mu\big)^{c_t'} \cdot n^{-1/c_t'}.
		\end{align*}
	\end{enumerate}
	Here $\mathsf{Z}\sim \mathcal{N}(0,1)$ is independent of all other variables.
\end{theorem}

Next, the normality of $Z^{(t)}$ can be used for constructing general consistent estimators for the `generalization error' of $\mu^{(t)}$, formally defined as follows.
\begin{definition}
	The \emph{generalization error} $\mathscr{E}_{\mathsf{H}}^{(t)}(A,Y)$ for the gradient descent iterate $\mu^{(t)}$ under a given loss function $\mathsf{H}:\R^2\to \R$ is defined as 
	\begin{align}\label{def:gen_err}
	\mathscr{E}_{\mathsf{H}}^{(t)}\equiv \mathscr{E}_{\mathsf{H}}^{(t)}(A,Y)\equiv \E\big[\mathsf{H}\big(\iprod{A_{\mathrm{new}}}{\mu^{(t)}},\mathcal{F}(\iprod{A_{\mathrm{new}}}{\mu_\ast},\xi_{\pi_m})\big)|(A,Y)\big].
	\end{align}
	Here the expectation $\E$ is taken jointly over $A_{\mathrm{new}}\equald A_1$ and $\pi_m$.
\end{definition}
Let the oracle generalization error estimate $\overline{\mathscr{E}}_{\mathsf{H}}^{(t)}(A,Y)$ be defined by
\begin{align}\label{def:gen_err_est_oracle}
\overline{\mathscr{E}}_{\mathsf{H}}^{(t)}(A,Y)\equiv m^{-1} \iprod{\mathsf{H} (Z^{(t)},Y)}{\bm{1}_m}= \frac{1}{m}\sum_{k \in [m]} \mathsf{H}(Z^{(t)}_k,Y_k),
\end{align}
To provide some intuition for the above proposal, as conditional on data $(A,Y)$,
\begin{align*}
\binom{\iprod{A_{\mathrm{new}}}{\mu^{(t)}}}{\iprod{A_{\mathrm{new}}}{\mu_\ast}}\stackrel{d}{\approx} \mathcal{N}\left(0_2, \frac{1}{n}
\begin{bmatrix}
\pnorm{\mu^{(t)}}{}^2 & \iprod{ \mu^{(t)}}{\mu_\ast}\\
\iprod{ \mu^{(t)}}{\mu_\ast} & \pnorm{\mu_\ast}{}^2
\end{bmatrix}
\right) \stackrel{d}{\approx} \binom{ \mathfrak{Z}^{(t+1)} }{ \mathfrak{Z}^{(0)} },
\end{align*}
we then expect
\begin{align}\label{ineq:gen_err_est_reason}
\mathscr{E}_{\mathsf{H}}^{(t)}(A,Y)&\equiv \E\big[\mathsf{H}\big(\iprod{A_{\mathrm{new}}}{\mu^{(t)}},\mathcal{F}(\iprod{A_{\mathrm{new}}}{\mu_\ast},\xi_{\pi_m})\big)|(A,Y)\big]\nonumber\\
& \approx \E\big[\mathsf{H}\big(\mathfrak{Z}^{(t+1)},\mathcal{F}(\mathfrak{Z}^{(0)},\xi_{\pi_m})\big)\big]\nonumber\\
& \approx m^{-1} \iprod{\mathsf{H} ({Z}^{(t)},Y)}{\bm{1}_m}=\overline{\mathscr{E}}_{\mathsf{H}}^{(t)}(A,Y).
\end{align}
The following theorem makes this heuristic precise. For notational simplicity, for $k \in [m]$, let $\mathsf{H}_{\mathcal{F};k}(u_1,u_2)\equiv \mathsf{H}(u_1,\mathcal{F}(u_2,\xi_k))$, and $\mathsf{H}_{\mathcal{F}}(u_1,u_2)\equiv \E_{\pi_m} \mathsf{H}_{\mathcal{F};\pi_m}(u_1,u_2)$.

\begin{theorem}\label{thm:gen_error_oracle}
	Suppose the assumptions in Theorem \ref{thm:gd_se}-(2) hold, and additionally $\{\mathsf{H}_{\mathcal{F};k}\}_{k \in [m]}\subset C^3(\R^2)$ admits mixed derivatives of order 3 all bounded by $\Lambda$. Then for any $q\in \N$, there exists some $c_t=c_t(t,q)>1$ such that
	\begin{align*}
	\E^{(0)}\abs{\overline{\mathscr{E}}_{\mathsf{H}}^{(t)}(A,Y)- \mathscr{E}_{\mathsf{H}}^{(t)}(A,Y) }^q\leq (K\Lambda L_\mu )^{c_t}\cdot n^{-1/c_t}.
	\end{align*}
\end{theorem}
We have not pursued the weakest possible conditions on, e.g., $\{\mathsf{H}_{\mathcal{F};k}\}_{k \in [m]}$, as this can usually be weakened via technical modifications in specific models.

From Theorems~\ref{thm:db_gd_oracle} and~\ref{thm:gen_error_oracle}, to use (i) the oracle debiased gradient descent iterate $\mu^{(t)}_{\mathrm{db}}$ for statistical inference of $\mu_\ast$, and (ii) the oracle generalization estimate $\overline{\mathscr{E}}_{\mathsf{H}}^{(t)}(A,Y)$ for consistent estimation of $\mathscr{E}_{\mathsf{H}}^{(t)}(A,Y)$, it remains to provide data-driven estimates of the Onsager correction matrices $\bm{\tau}^{[t]}$ and $\bm{\rho}^{[t]}$. In Section~\ref{section:iterative_inf}, we propose a \emph{gradient descent inference algorithm} that can be naturally embedded within vanilla gradient descent, and produces consistent estimators $\hat{\bm{\tau}}^{[t]}$ and $\hat{\bm{\rho}}^{[t]}$ for $\bm{\tau}^{[t]}$ and $\bm{\rho}^{[t]}$ at each iteration $t$.

A particularly important feature of Theorems \ref{thm:db_gd_oracle} and \ref{thm:gen_error_oracle} is that it does not require any convexity assumption. In fact, all Theorems~\ref{thm:gd_se}, \ref{thm:db_gd_oracle} and \ref{thm:gen_error_oracle} allow for \emph{arbitrary} non-convexity, subject to the smoothness conditions specified therein. In particular, the debiased gradient descent iterate $\mu^{(t)}_{\mathrm{db}}$ and the generalization error estimate $\overline{\mathscr{E}}_{\mathsf{H}}^{(t)}(A,Y)$, when coupled with the \emph{gradient descent inference algorithm} to be detailed in Section \ref{section:iterative_inf}, can be used for inference of $\mu_\ast$ and estimation of the generalization error in a broad class of non-convex problems. Several such examples will be presented in Section~\ref{section:example_linear}.

\begin{remark}\label{rmk:gd_dyn_main_thm}
Some technical remarks on Theorems \ref{thm:gd_se}, \ref{thm:db_gd_oracle} and \ref{thm:gen_error_oracle}:
\begin{enumerate}
	\item The dependence of $c_t$ on $t$ can be tracked explicitly in the proof (for instance, $c_t \leq t^{c_0 t}$ for some universal constant $c_0 > 0$), but we have omitted this explicit dependence as these estimates are likely suboptimal. Note that at the level of the general theory, any further improvement must encounter a barrier at $t \asymp \log n$ (cf. \cite{rush2018finite}). Such a barrier can be improved only in concrete models; see, e.g. \cite{li2022non,jones2024diagram}, for results in this direction for specialized AMP algorithms under Gaussian/Rademacher designs, and \cite{han2025long} for the vanilla gradient descent algorithm in the regime $\phi\gg 1$.
	\item The error bounds require all quantities $K,\Lambda$ and $L_\mu$ to grow slowly, for instance at a rate of at most $n^\epsilon$ for some small $\epsilon > 0$. In particular, this implies that the initialization $\mu^{(0)}$ must satisfy $\pnorm{\mu^{(0)}}{\infty} \leq n^\epsilon$.
	\item In applications, condition (A4) (resp. (A4')) is satisfied when $\mathsf{P}_t: \R^n \to \R^n$ has coordinate mappings given by the same function $x \mapsto \prox_{\eta_{t-1} \mathsf{f}}(x)$, where $\mathsf{f}$ is a sufficiently smooth convex regularizer such that $\max_{q \in [2:4]} \pnorm{\mathsf{f}^{(q)}}{\infty} < \infty$ (resp. $\pnorm{\mathsf{f}'}{\mathrm{Lip}} < \infty$).
	\item Conditions (A5) and (A5') are designed to reflect the interplay between the loss function and the noise distribution. For example, in the linear model $\mathcal{F}(x, y) = x + y$ with a symmetric loss function $\mathsf{L}(x, y) = \mathsf{L}_\ast(x - y)$, both (A5) and (A5') primarily require a slowly growing bound on $\pnorm{\mathsf{L}_\ast'(\xi)}{\infty}$, which directly captures the well-known interaction between the tail behavior of the noise and the choice of loss function.
	\item In specific examples, regularity conditions on these theorems may not hold. Nonetheless, our theory can often be adapted with technical modifications. For example, in Section \ref{section:example_1bit} ahead, we demonstrate how our framework applies to generalized logistic regression, even though $\mathcal{F}$ is not globally continuous. 		
\end{enumerate}
\end{remark}

\section{Debiased inference via gradient descent inference algorithm}\label{section:iterative_inf}

\subsection{Gradient descent inference algorithm}

As highlighted in the previous section, the key to using the debiased gradient descent iterate $\mu^{(t)}_{\mathrm{db}}$ in (\ref{def:debias_gd_oracle}) and the generalization error estimate $\overline{\mathscr{E}}_{\mathsf{H}}^{(t)}(A,Y)$ in (\ref{def:gen_err_est_oracle}) for statistical inference lies in providing data-driven estimates of the Onsager correction matrices $\bm{\tau}^{[t]}$ and $\bm{\rho}^{[t]}$ defined in (\ref{def:tau_rho_mat}).

While the exact forms of these matrices are generally complex and not analytically tractable, we present in Algorithm \ref{def:alg_tau_rho} below a general iterative procedure for computing their data-driven estimates, which we refer to as the \emph{gradient descent inference algorithm}. Recall the notation $\mathfrak{O}_t(M)\in \R^{t\times t}$ defined in (\ref{def:mat_O}) for a general matrix $M\in \R^{(t-1)\times (t-1)}$.

\begin{algorithm}
	\caption{Gradient descent inference algorithm}\label{def:alg_tau_rho}
	\begin{algorithmic}[1]
		\STATE Input data $A \in \R^{m\times n}$, $Y \in \R^n$, step sizes $\{\eta_t\}\subset \R_{>0}$.
		\STATE Initialize with $\mu^{(0)}\in \R^n$,  $\hat{\bm{\rho}}^{[0]}\equiv \emptyset$, $\{\hat{\bm{\rho}}_\ell^{[0]}\}_{\ell \in [n]}\equiv \emptyset$.
		\FOR{$ t = 1,2,\ldots$}
		\STATE Compute the coefficient matrices $\{\hat{\bm{\tau}}^{[t]}_k\}_{k \in [m]} \subset \R^{t\times t}$ by
		\begin{align*}
		\hat{\bm{L}}^{[t]}_k &\equiv \mathrm{diag}\Big(\Big\{-\eta_{s-1} \iprod{e_k}{\partial_{11}\mathsf{L}_{s-1}(A\mu^{(s-1)},Y)}\Big\}_{s \in [1:t]}\Big),\\
		\hat{\bm{\tau}}^{[t]}_k &\equiv \phi\cdot  \big[I_t - \hat{\bm{L}}^{[t]}_k \mathfrak{O}_t(\hat{\bm{\rho}}^{[t-1]})\big]^{-1} \hat{\bm{L}}^{[t]}_k,\,\hbox{ for all } k \in [m].
		\end{align*}
		The average $\hat{\bm{\tau}}^{[t]}$ is computed as $\hat{\bm{\tau}}^{[t]}\equiv m^{-1}\sum_{k \in [m]} \hat{\bm{\tau}}^{[t]}_{k}\in \R^{t\times t}$. 
		\STATE Compute the coefficient matrices $\{\hat{\bm{\rho}}^{[t]}_\ell\}_{\ell \in [n]} \subset \R^{t\times t}$ by
		\begin{align*}
		\hat{\bm{P}}^{[t]}_\ell&\equiv \mathrm{diag}\Big(\Big\{\mathsf{P}_{s;\ell}'\big(\bigiprod{e_\ell}{\mu^{(s-1)}-\eta_{s-1} A^\top \partial_1 \mathsf{L}_{s-1}(A \mu^{(s-1)},Y) }\big)\Big\}_{s \in [1:t]}\Big),\\
		\hat{\bm{\rho}}_\ell^{[t]}&\equiv \hat{\bm{P}}_\ell^{[t]}\big[I_t+(\hat{\bm{\tau}}^{[t]}+I_t)\mathfrak{O}_t(\hat{\bm{\rho}}_\ell^{[t-1]})\big],\,\hbox{ for all } \ell \in [n].
		\end{align*}
		The average $\hat{\bm{\rho}}^{[t]}$ is computed as $\hat{\bm{\rho}}^{[t]}\equiv n^{-1}\sum_{\ell \in [n]} \hat{\bm{\rho}}^{[t]}_\ell \in \R^{t\times t}$.
		\STATE Compute the gradient descent iterate $\mu^{(t)}$ by (\ref{def:grad_descent}).
		\ENDFOR
	\end{algorithmic}
\end{algorithm}

The algorithmic form for computing $\hat{\bm{\tau}}^{[\cdot]}$ and $\hat{\bm{\rho}}^{[\cdot]}$ is directly inspired by the state evolution in Definition \ref{def:gd_se}:
\begin{itemize}
	\item By taking derivatives on both sides of (S1), with
	\begin{align*}
	\bm{\Upsilon}_{k}^{';[t]}(z^{([0:t])}) &\equiv \big(\partial_{z^{(s)}}\Upsilon_{r;k}(z^{([0:r])})\big)_{r,s \in [1:t]} \in \R^{t\times t},\\
	\bm{L}^{[t]}_{k}(z^{([0:t])}) &\equiv \mathrm{diag}\big(\big\{-\eta_{s-1}  \partial_{11} \mathsf{L}_{s-1}\big(\Theta_{s;k}(z^{([0:s])}) ,\mathcal{F}(z^{(0)},\xi_k)\big)\big\}_{s \in [1:t]}\big) \in \R^t,
	\end{align*}
	we have the derivative formula
	\begin{align*}
	&\bm{\Upsilon}_{k}^{';[t]}(z^{([0:t])}) =  \bm{L}^{[t]}_{k}(z^{([0:t])}) + \bm{L}^{[t]}_{k}(z^{([0:t])}) \mathfrak{O}_{t}(\bm{\rho}^{[t-1]}) \bm{\Upsilon}_{k}^{';[t]}(z^{([0:t])}),\\
	\implies \,\, & \bm{\Upsilon}_{k}^{';[t]}(z^{([0:t])}) = \big[I_t- \bm{L}^{[t]}_{k}(z^{([0:t])}) \mathfrak{O}_{t}(\bm{\rho}^{[t-1]}) \big]^{-1}\bm{L}^{[t]}_{k}(z^{([0:t])}).
	\end{align*}
	Thus, $\hat{\bm{\tau}}^{[t]}$ can be viewed as a data-driven version of the averaged solutions $\{\bm{\Upsilon}_{k}^{';[t]}(z^{([0:t])})\}_{k \in [m]}$ to the above equation.
	\item By taking derivatives on both sides of (S3), with 
	\begin{align*}
	\bm{\Omega}_{\ell}^{';[t]}(w^{([1:t])})&\equiv \big(\partial_{w^{(s)}}\Omega_{r;\ell}(w^{([1:r])})\big)_{r,s \in [1:t]} \in \R^{t\times t},\\
	\bm{P}^{[t]}_{\ell}(w^{([1:t])}) &\equiv \mathrm{diag}\big(\big\{\mathsf{P}_{s;\ell}'(\Delta_{s;\ell} (w^{([1:t])}) )\big\}_{s \in [1:t]}\big) \in \R^t,
	\end{align*} 
	we have the derivative formula
	\begin{align*}
	\bm{\Omega}_{\ell}^{';[t]}(w^{([1:t])}) = \bm{P}^{[t]}_{\ell}(w^{([1:t])})\big[I_t+(\bm{\tau}^{[t]}+I_t)\mathfrak{O}_t\big(\bm{\Omega}_{\ell}^{';[t-1]}(w^{([1:t-1])}) \big)\big].
	\end{align*}
	Thus, $\hat{\bm{\rho}}^{[t]}$ can be viewed as a data-driven version of the averaged solutions $\{\bm{\Omega}_{\ell}^{';[t]}(w^{([1:t])}) \}_{\ell \in [n]}$ to the above equation.
\end{itemize}
Here we follow the convention that boldface fonts are used to denote matrices that collect elements from the corresponding non-bold quantities.

\begin{remark}
	Some remarks on Algorithm \ref{def:alg_tau_rho}:
\begin{enumerate}
	\item A stopping time is not explicitly included, and it should be understood that if stopped at iteration $t$, the algorithm outputs (i) estimates $\hat{\bm{\tau}}^{[t]}, \hat{\bm{\rho}}^{[t]}$ for the Onsager correction matrices $\bm{\tau}^{[t]},\bm{\rho}^{[t]}$, and (ii) the gradient descent iterate $\mu^{(t)}$.
	\item As the inversion of a lower triangular matrix can be computed in $\mathcal{O}(t^2)$ operations, updating the $t$-th rows for $\hat{\bm{\tau}}^{[t]}, \hat{\bm{\rho}}^{[t]}$ in Algorithm \ref{def:alg_tau_rho} at iteration $t$ requires additionally $\mathcal{O}(n t^2)$ operations on top of $\mathcal{O}(n^2)$ many operations that the gradient descent iterate  (\ref{def:grad_descent}) requires in general. The computational complexity can be further reduced when stochastic gradient methods are employed. 
\end{enumerate}
\end{remark}

The next theorem shows that the outputs $\hat{\bm{\tau}}^{[t]}, \hat{\bm{\rho}}^{[t]}$ of Algorithm \ref{def:alg_tau_rho} are close to their `population' versions $\bm{\tau}^{[t]},\bm{\rho}^{[t]}$ (defined in (\ref{def:tau_rho_mat})) in the state evolution.

\begin{theorem}\label{thm:consist_tau_rho}
	Suppose Assumption \ref{assump:setup} holds for some $K,\Lambda \geq 2$. Moreover, suppose that for $s \in [0:t-1]$:
	\begin{enumerate}
		\item[(A4*)] $\max_{\ell \in [n]} \big\{\abs{\mathsf{P}_{s+1;\ell}(0)} \vee \pnorm{\mathsf{P}_{s+1;\ell}}{\mathrm{Lip}}\vee \pnorm{\mathsf{P}_{s+1;\ell}'}{\mathrm{Lip}}\big\}\leq \Lambda$.
		\item[(A5*)] $\max_{k \in [m]}\big\{ \abs{\partial_1 \mathsf{L}_{s;k}\big(0, \mathcal{F}(0,\xi_k)\big)} \vee \max_{*=1,11} \pnorm{\partial_{*} \mathsf{L}_{s;k}\big(\cdot, \mathcal{F}(\cdot,\xi_k)}{\mathrm{Lip}} \big\} \leq \Lambda$.
	\end{enumerate}
    Then for any $q>1$, there exists some constant $c_t=c_t(t,q)>1$ such that 
    \begin{align*}
    & \E^{(0)}  \pnorm{\hat{\bm{\tau}}^{[t]}-\bm{\tau}^{[t]} }{\op}^q\vee \E^{(0)}  \pnorm{\hat{\bm{\rho}}^{[t]}-\bm{\rho}^{[t]} }{\op}^q  \leq   (K\Lambda L_\mu )^{ c_t} \cdot n^{-1/c_t}.
    \end{align*} 
\end{theorem}

Similar to Remark \ref{rmk:gd_dyn_main_thm} for the results in Section \ref{section:gd_dynamics}, the technical conditions in the above theorem are not the weakest possible and are chosen for clean presentation. In concrete applications, these regularity conditions can possibly be further relaxed in a case-by-case manner.

\subsection{Application I: Inference via debiased gradient descent iterate $\mu^{(t)}$}

With the output $\hat{\bm{\tau}}^{[t]}$ from Algorithm \ref{def:alg_tau_rho}, let 
\begin{align*}
\hat{\bm{\omega}}^{[t]}\equiv (\hat{\bm{\tau}}^{[t]})^{-1}
\end{align*}
(whenever invertible) be an estimator for $\bm{\omega}^{[t]}$. The oracle debiased gradient descent iterate $\mu^{(t)}_{\mathrm{db}}$ in (\ref{def:debias_gd_oracle}) then admits a data-driven version:
\begin{align}\label{def:debias_gd}
\hat{\mu}^{(t)}_{\mathrm{db}}&\equiv  \mu^{(t-1)}+\sum_{s \in [1:t]} \hat{\omega}_{t,s}\cdot \eta_{s-1}A^\top \partial_1 \mathsf{L}_{s-1}(A\mu^{(s-1)},Y).
\end{align}
Note that $\hat{\mu}^{(t)}_{\mathrm{db}}$ can be computed using iterates $\{\mu^{(0)},\ldots,\mu^{(t-1)}\}$, and we retain the index $t$ to indicate that this may be naturally incorporated in the $t$-th iteration of Algorithm \ref{def:alg_tau_rho}.

The following is an immediate consequence of Theorems \ref{thm:db_gd_oracle} and \ref{thm:consist_tau_rho}.
\begin{theorem}
Suppose the assumptions in Theorem \ref{thm:consist_tau_rho} hold. Fix a sequence of $\Lambda_\psi$-pseudo-Lipschitz functions $\{\psi_\ell:\R \to \R\}_{\ell \in [n]}$ of order $\mathfrak{p}$ where $\Lambda_\psi\geq 2$. Then	for any $q \in \N$, there exists some $c_t=c_t(t,\mathfrak{p},q)>1$ such that
\begin{align*}
&\E^{(0)} \bigabs{\E_{\pi_n} \psi_{\pi_n} \big(\hat{\mu}^{(t)}_{\mathrm{db};\pi_n}\big)- \E^{(0)} \psi_{\pi_n}\big(b^{(t)}_{\mathrm{db}}\cdot \mu_{\ast,\pi_n}+ \sigma^{(t)}_{\mathrm{db}}\mathsf{Z}\big)}^q\\
&\leq \big((1\wedge \tau_\ast^{(t)})^{-1} K\Lambda \Lambda_\psi L_\mu\big)^{c_t} \cdot n^{-1/c_t}.
\end{align*} 
\end{theorem}
The proof is omitted for simplicity. In order to use $\hat{\mu}^{(t)}_{\mathrm{db}}$ for statistical inference of the unknown $\mu_\ast$, it remains to estimate the bias $b^{(t)}_{\mathrm{db}}$ and the variance $\big({\sigma}^{(t)}_{\mathrm{db}}\big)^2$:

\begin{itemize}
	\item Estimating the variance $\big({\sigma}^{(t)}_{\mathrm{db}}\big)^2$ is fairly easy. From the state evolution (S2), the covariance $\Sigma_{\mathfrak{W}}^{[t]}$ in the general empirical risk minimization problem can be naturally estimated by 
	\begin{align*}
	\hat{\Sigma}_{\mathfrak{W}}^{[t]}\equiv \bigg(\phi\cdot \eta_{r-1}\eta_{s-1} \cdot \frac{1}{m}\bigiprod{\partial_1 \mathsf{L}_{r-1}(A\mu^{(r-1)},Y) }{ \partial_1 \mathsf{L}_{s-1}(A\mu^{(s-1)},Y) }\bigg)_{r,s \in [t]}.
	\end{align*}
	Therefore, a natural variance estimator is
	\begin{align}\label{def:db_gd_var_est}
	\big(\hat{\sigma}^{(t)}_{\mathrm{db}}\big)^2\equiv \hat{\bm{\omega}}_{t\cdot}^{[t]}\hat{\Sigma}_{\mathfrak{W}}^{[t]}\hat{\bm{\omega}}_{t\cdot}^{[t],\top}.
	\end{align}
	In specific models, simpler methods may exist for estimating $\big({\sigma}^{(t)}_{\mathrm{db}}\big)^2$.
	\item Estimating the bias parameter $b^{(t)}_{\mathrm{db}}$ is more challenging and requires leveraging model-specific features. This is expected, as $b^{(t)}_{\mathrm{db}}$ involves $(\delta_s)_{s \in [t]}$ which are derivatives of $\Upsilon_{t}$ with respect to $\mathfrak{Z}^{(0)}$ that contains purely oracle information $\mu_\ast$. 
	If the signal strength $\sigma_{\mu_\ast}$ can be estimated by $\hat{\sigma}_{\mu_\ast}$, then we may invert the approximate normality (\ref{eqn:gd_db_normality}) to construct a generic bias estimator
	\begin{align}\label{def:db_gd_b_est}
	\abs{\hat{b}^{(t)}_{\mathrm{db}}}&\equiv \big[\pnorm{\hat{\mu}^{(t)}_{\mathrm{db}}}{}^2/n-\big(\hat{\sigma}^{(t)}_{\mathrm{db}}\big)^2\big]_+^{1/2}/\hat{\sigma}_{\mu_\ast}.
	\end{align}
\end{itemize}
With a bias estimator $\hat{b}^{(t)}_{\mathrm{db}}$ and a variance estimator $\big(\hat{\sigma}^{(t)}_{\mathrm{db}}\big)^2$, we may construct  $(1-\alpha)$ confidence intervals (CIs) for $\{\mu_{\ast,\ell}\}_{\ell \in [n]}$ as follows:
\begin{align}\label{def:db_gd_CI_general}
\mathsf{CI}_\ell^{(t)}(\alpha)\equiv \bigg[  \frac{\hat{\mu}^{(t)}_{\mathrm{db};\ell}}{ \hat{b}^{(t)}_{\mathrm{db}}} \pm \frac{ \hat{\sigma}^{(t)}_{\mathrm{db}} \cdot z_{\alpha/2} }{\abs{\hat{b}^{(t)}_{\mathrm{db}}} } \bigg].
\end{align}
Here for CI's, we write $z_\alpha$ as the solution to $\alpha\equiv \Prob(\mathcal{N}(0,1)>z_\alpha)$. The coverage validity of these CI's can be easily justified for each $\ell \in [n]$ under the conditions in Theorem \ref{thm:db_gd_oracle}-(1), and in an averaged sense under the conditions in Theorem \ref{thm:db_gd_oracle}-(2). We omit these routine technical details.

\subsection{Application II: Estimation of the generalization error}

With the output $\hat{\bm{\rho}}^{[t]}$ from Algorithm \ref{def:alg_tau_rho}, the oracle generalization error estimate $\overline{\mathscr{E}}_{\mathsf{H}}^{(t)}(A,Y)$ in (\ref{def:gen_err_est_oracle}) has the following data-driven version:
\begin{align}\label{def:gen_err_est}
\hat{\mathscr{E}}_{\mathsf{H}}^{(t)}(A,Y)\equiv m^{-1} \iprod{\mathsf{H} (\hat{Z}^{(t)},Y)}{\bm{1}_m}= \frac{1}{m}\sum_{k \in [m]} \mathsf{H}(\hat{Z}^{(t)}_k,Y_k).
\end{align}
Here $\hat{Z}^{(t)}$ uses the output $ \hat{\bm{\rho}}^{[t]}$ of Algorithm \ref{def:alg_tau_rho} to estimate $Z^{(t)}$ defined in (\ref{def:Z_W}):
\begin{align}\label{def:Z_hat}
\hat{Z}^{(t)}&\equiv A\mu^{(t)}+\sum_{s \in [1:t]} \eta_{s-1} \hat{\rho}_{t,s} \cdot \partial_1\mathsf{L}_{s-1}\big(A\mu^{(s-1)},Y\big) \in \R^m.
\end{align}
In fact, computation of $\hat{Z}^{(t)}$ above and therefore $\hat{\mathscr{E}}_{\mathsf{H}}^{(t)}(A,Y)$ can be immediately embedded into Algorithm \ref{def:alg_tau_rho}, so that at $t$-th iteration, the algorithm outputs the estimate $\hat{\mathscr{E}}_{\mathsf{H}}^{(t)}(A,Y)$ for the unknown generalization error $\mathscr{E}_{\mathsf{H}}^{(t)}(A,Y)$.

The following theorem formally validates (\ref{def:Z_hat}).

\begin{theorem}\label{thm:gen_error}
Suppose the assumptions in Theorem \ref{thm:consist_tau_rho} hold, and additionally $\{\mathsf{H}_{\mathcal{F};k}\}_{k \in [m]}\subset C^3(\R^2)$ admits mixed derivatives of order 3 all bounded by $\Lambda$. Then for any $q\in \N$, there exists some $c_t=c_t(t,q)>1$ such that
\begin{align*}
\E^{(0)}\abs{\hat{\mathscr{E}}_{\mathsf{H}}^{(t)}(A,Y)- \mathscr{E}_{\mathsf{H}}^{(t)}(A,Y) }^q\leq (K\Lambda L_\mu )^{c_t}\cdot n^{-1/c_t}.
\end{align*}
\end{theorem}

\begin{remark}[\emph{Comparison to LOOCV}]
A different heuristic approach, leave-one-out cross-validation (LOOCV), can be applied to gradient descent to produce a generalization error estimate at each iteration. While this procedure is generic, its theoretical validity must be verified on a case-by-case basis and has so far only been established for vanilla gradient descent under the simplest linear regression setting with squared loss, as shown in \cite{patil2024failures}.

A well-known practical issue with the LOOCV method is that it is computationally demanding and does not scale well to large-scale problems. In the context of gradient descent, at iteration $t$, LOOCV typically requires $\bigo(n^3)$ operations, whereas our proposed method requires  at most $\bigo(nt^2+n^2)$. In particular, in the common regime where $t \ll n$, our method offers substantial computational advantages over LOOCV. In Appendix~\ref{subsection:loo_gen_est}, we compare the performance of our method and LOOCV using vanilla gradient descent with squared loss, and find that while both methods remain valid even under possible model mis-specification, our proposal is hundreds of times faster than LOOCV for a moderate number of iterations.

\end{remark}

\begin{remark}[\emph{On the initialization scheme}]
Our inference proposals in the preceding subsections do not \emph{apriori} specify whether an uninformative or informative initialization $\mu^{(0)}$ should be used, as long as it is independent of the data matrix $X$ and $\pnorm{\mu^{(0)}}{\infty}$ grows mildly (cf., Remark \ref{rmk:gd_dyn_main_thm}-(2)). However, additional care in the choice of initialization scheme may be needed depending on the inference task:
\begin{itemize}
	\item For the task of consistent estimation of the generalization error, our theory in Theorem~\ref{thm:gen_error_oracle} and the inference proposal~\eqref{def:gen_err_est} hold without further conditions on the initialization scheme.
	\item The task of constructing confidence intervals for $\mu_\ast$ is more subtle. In particular, while our abstract theory in Theorem~\ref{thm:db_gd_oracle} holds without additional conditions on the initialization scheme, the proposed confidence intervals~\eqref{def:db_gd_CI_general} remain valid only if the `bias parameter' $b_{\mathrm{db}}^{(t)}$ is of order~1.
\end{itemize}
From a practical point of view, although the exact expression of $b_{\mathrm{db}}^{(t)}$ can be computed in closed form only in special cases (for instance, linear regression in Section~\ref{section:example_linear} and logistic regression under squared loss in Section~\ref{section:example_1bit}), whether it is of order~1 can be checked numerically via the data-driven estimate~\eqref{def:db_gd_b_est}. Therefore, for the purpose of inference on $\mu_\ast$, one should proceed with an initialization scheme that renders the bias estimator in~\eqref{def:db_gd_b_est} not exceedingly small.
\end{remark}

\section{Example I: Single-index regression}\label{section:example_linear}

In this section we consider single-index regression in Example \ref{model:single_index}. Suppose the model is correctly specified, and consider the loss function $\mathsf{L}_{\cdot}(x,y)\equiv \mathsf{L}_\ast(\varphi_\ast(x)-y)$ for some symmetric function $\mathsf{L}_\ast:\R \to \R_{\geq 0}$. We are interested in solving the non-convex ERM problem (\ref{def:ERM_general}) in the above single-index regression model with gradient descent. As $\varphi_\ast$ does not have any convexity structure, the ERM problem (\ref{def:ERM_general}) may exhibit arbitrary non-convexity, and therefore the gradient descent algorithm is not guaranteed to converge. For simplicity of discussion, we take a fixed step size $\eta_t\equiv \eta>0$, and the resulting gradient descent algorithm reads
\begin{align}\label{def:grad_descent_linear}
\mu^{(t)}=\prox_{\eta \mathsf{f}}\Big(\mu^{(t-1)}-\eta\cdot A^\top  \big[\mathsf{L}_{\ast}' (\varphi_\ast(A\mu^{(t-1)})-Y)\odot \varphi_\ast'(A\mu^{(t-1)}) \big]\Big).
\end{align}

\subsection{Distributional characterizations}

Distributional theory for the gradient descent iterate (\ref{def:grad_descent_linear}) and the validity of Algorithm \ref{def:alg_tau_rho} follow immediately from our theory in the previous sections. We state these results without a formal proof.

\begin{theorem}\label{thm:linear}
	Suppose Assumption \ref{assump:setup} holds for some $K,\Lambda \geq 2$, $
	\pnorm{\mathsf{L}_\ast'(\xi)}{\infty}\leq \Lambda$ and the maps $\{(x,y)\mapsto \mathsf{L}_\ast'\big(\varphi_\ast(x)-\varphi_\ast(y)+\xi_i\big)\cdot \varphi_\ast'(x)\}_{i \in [m]}$ are $\Lambda$-Lipschitz.
	\begin{enumerate}
		\item Further suppose (A4') holds. The averaged distributional characterizations in Theorem \ref{thm:gd_se}-(2) hold for the gradient descent iterates in (\ref{def:grad_descent_linear}).
		\item Further suppose (A4*) holds and the maps $\{(x,y)\mapsto \mathsf{L}_\ast''\big(\varphi_\ast(x)-\varphi_\ast(y)+\xi_i\big) (\varphi_\ast'(x))^2+ \mathsf{L}_\ast'\big(\varphi_\ast(x)-\varphi_\ast(y)+\xi_i\big) \varphi_\ast''(x)\}_{i \in [m]}$ are $\Lambda$-Lipschitz. Then Theorem \ref{thm:consist_tau_rho} holds for the output $(\hat{\bm{\tau}}^{[t]},\hat{\bm{\rho}}^{[t]})$ in Algorithm \ref{def:alg_tau_rho} using gradient descent iterates in (\ref{def:grad_descent_linear}).
	\end{enumerate}
\end{theorem}

The information parameter $\delta_t$ takes a simple form in the single-index model:
\begin{lemma}\label{lem:linear_delta}
Suppose $\varphi_\ast \in C^1(\R)$. Then the information parameter is $
\delta_t = -  \sum_{s \in [1:t]} \phi\cdot  \E^{(0)}\partial_{\mathfrak{Z}^{(s)}} \Upsilon_{t;\pi_m}(\mathfrak{Z}^{([0:t])}) \varphi_\ast'(\mathfrak{Z}^{(0)})$. 
For the linear model, further simplification is possible:  $\delta_t=-\sum_{s \in [1:t]} \tau_{t,s}$.
\end{lemma}

The simplicity of $\delta_t$ in the linear model arises from a key structural property: the function $\Upsilon_t(\mathfrak{z}^{([0:t])})$ depends on $\mathfrak{z}^{([0:t])}$ only through $\{\mathfrak{z}^{(s)}-\mathfrak{z}^{(0)}\}_{s \in [1:t]}$. This structure directly leads to the simplified formula for $\delta_t$ in the linear model.

\subsection{Mean-field statistical inference}

\subsubsection{Estimation of generalization error}
Consider the generalization error (\ref{def:gen_err}) with $\mathsf{H}(x,y)=\mathsf{H}_\ast(\varphi_\ast(x)-y)$ for a symmetric loss function $\mathsf{H}_\ast:\R\to \R_{\geq 0}$. The proposed estimator (\ref{def:gen_err_est}) simplifies to
	\begin{align}\label{def:linear_gen_err_est}
	\hat{\mathscr{E}}_{\mathsf{H}}^{(t)}&\equiv \frac{1}{m}\sum_{k \in [m]} \mathsf{H}_\ast\,\bigg[\varphi_\ast\bigg(\big(A\mu^{(t)}\big)_k\nonumber\\
	&\qquad+\eta\sum_{s \in [1:t]}  \hat{\rho}_{t,s} \cdot \mathsf{L}_\ast'\big( (\varphi_\ast(A\mu^{(s-1)})-Y)_k\big)\cdot \varphi_\ast'\big((A\mu^{(s-1)})_k\big)\bigg)-Y_k \bigg],
	\end{align}
	where $\hat{\bm{\rho}}^{[t]}$ is the output of Algorithm \ref{def:alg_tau_rho}. The conditions in Theorem \ref{thm:gen_error} can be easily adapted to this setting. Under these conditions, $\hat{\mathscr{E}}_{\mathsf{H}}^{(t)}\approx \mathscr{E}_{\mathsf{H}}^{(t)}(A,Y)$ in the sense described therein.

\subsubsection{Inference via debiased gradient descent}

For the single-index model, the data-driven counterpart of the oracle debiased gradient descent iterate $\mu^{(t)}_{\mathrm{db}}$ takes the following form:
	\begin{align}\label{def:linear_debias_gd}
	\hat{\mu}^{(t)}_{\mathrm{db}}&\equiv  \mu^{(t-1)}+\eta \sum_{s \in [1:t]} \hat{\omega}_{t,s}\cdot A^\top  \Big[\mathsf{L}_{\ast}'(\varphi_\ast(A\mu^{(s-1)})-Y)\odot \varphi_\ast'(A\mu^{(s-1)})\Big].
	\end{align}
A special case is the linear model, where significant simplifications take place. 

\begin{proposition}\label{prop:linear_bias_var}
	Consider the linear model with $\varphi_\ast(x)=x$. Suppose that $\bm{\tau}^{[t]}$ is invertible. 
	\begin{enumerate}
		\item The bias $b^{(t)}_{\mathrm{db}}=1$. Moreover, for the squared loss $\mathsf{L}_\ast(x)=x^2/2$, the variance $
		(\sigma^{(t)}_{\mathrm{db}})^2=\phi^{-1}\cdot \big\{\E^{(0)}\big(\mathfrak{Z}^{(t)}-\mathfrak{Z}^{(0)}\big)^2+\sigma_m^2\big\}$, where $\sigma_m^2\equiv \E_{\pi_m} \xi_{\pi_m}^2$.
		\item Under the conditions in Theorem \ref{thm:linear}-(1), the averaged distributional characterizations for $\mu^{(t)}_{\mathrm{db}}$ in Theorem \ref{thm:db_gd_oracle}-(2) hold with $b^{(t)}_{\mathrm{db}},
		\sigma^{(t)}_{\mathrm{db}}$ specified as above.
	\end{enumerate}
	
\end{proposition}

From Proposition \ref{prop:linear_bias_var}-(1) and Theorem \ref{thm:db_gd_oracle}, we expect $
\hat{\mu}^{(t)}_{\mathrm{db}}\approx {\mu}^{(t)}_{\mathrm{db}}\stackrel{d}{\approx} \mu_\ast+ \sigma^{(t)}_{\mathrm{db}}\cdot  \mathsf{Z}$. Thus, with $\hat{\sigma}^{(t)}_{\mathrm{db}}$ defined in (\ref{def:db_gd_var_est}), in linear model the CI's in (\ref{def:db_gd_CI_general}) simplify to
\begin{align}\label{def:linear_CI}
\mathsf{CI}_\ell^{(t)}(\alpha)\equiv \big[  \hat{\mu}^{(t)}_{\mathrm{db};\ell} \pm  \hat{\sigma}^{(t)}_{\mathrm{db}} \cdot z_{\alpha/2}  \big].
\end{align}
Under the squared loss, the variance formula in Proposition \ref{prop:linear_bias_var}-(1) indicates a further simplification in variance estimation. Using the heuristics in (\ref{ineq:gen_err_est_reason}), 
\begin{align}\label{eqn:linear_var_gen}
(\sigma^{(t)}_{\mathrm{db}})^2\approx \phi^{-1} \mathscr{E}_{\abs{\cdot}^2}^{(t-1)}(A,Y),\quad t=1,2,\ldots.
\end{align}
By (\ref{eqn:linear_var_gen}), we may use the scaled generalization error estimate $\phi^{-1} \hat{\mathscr{E}}_{\abs{\cdot}^2}^{(t-1)}$ in (\ref{def:linear_gen_err_est}) as an estimator for $(\sigma^{(t)}_{\mathrm{db}})^2$. Consequently, under the squared loss, a further simplified CI can be devised in the linear model:
\begin{align}\label{def:linear_CI_sq}
\mathsf{CI}_{\mathrm{sq};\ell}^{(t)}(\alpha)\equiv \Big[  \hat{\mu}^{(t)}_{\mathrm{db};\ell} \pm  \big(\phi^{-1} \hat{\mathscr{E}}_{\abs{\cdot}^2}^{(t-1)}\big)^{1/2} \cdot z_{\alpha/2}  \Big].
\end{align}
The above CI's are valid in an averaged sense $\abs{\E_{\pi_n} \mathsf{CI}_{\mathrm{sq};\pi_n}^{(t)}(\alpha)-\alpha}\stackrel{\mathbb{P}}{\approx}0$, by an application of Proposition \ref{prop:linear_bias_var} coupled with a routine smoothing argument to lift the Lipschitz condition required for the test function therein. 

Note that if $\mu^{(t)} \approx \hat{\mu}$ for large $t$, then combining (\ref{eqn:var_gen_intro}) and (\ref{eqn:linear_var_gen}) yields $(\sigma^{(t)}_{\mathrm{db}})^2 \approx \sigma_{\mathrm{db}}^2$. This implies that the debiased gradient descent iterate $\hat{\mu}^{(t)}_{\mathrm{db}}$ in linear regression with squared loss has approximately the same Gaussian law as the debiased convex regularized estimator $\mu^{\mathrm{ls}}_{\mathrm{db}}$ defined in (\ref{eqn:debiased_reg_est}), in the sense that for regular enough test functions $\{\psi_j\}$,
\begin{align}\label{eqn:debias_gd_cvx}
\lim_{t \to \infty}\limsup_{n \to \infty}\E^{(0)}\biggabs{\frac{1}{n}\sum_{j \in [n]} \psi_j\big(\hat{\mu}^{(t)}_{\mathrm{db};j}\big)-\frac{1}{n}\sum_{j \in [n]} \psi_j\big(\mu^{\mathrm{ls}}_{\mathrm{db};j}\big)}^q = 0.
\end{align}
The above statement can be formally established under suitable conditions on the regularizer $\mathsf{f}$, with strong convexity being the simplest sufficient condition. We omit these technical details for brevity.

\section{Example II: Generalized logistic regression}\label{section:example_1bit}
In this section we consider generalized logistic regression in Example \ref{model:1bit}. Consider a loss function $\mathsf{L}(x,y)$ that does not depend on the iteration, and the associated proximal gradient descent algorithm with a fixed step size $\eta>0$:
\begin{align}\label{def:grad_descent_1bit_cs}
\mu^{(t)}= \prox_{\eta \mathsf{f}}\big(\mu^{(t-1)}-\eta \cdot A^\top \partial_1\mathsf{L}(A\mu^{(t-1)},Y)\big).
\end{align}
Note that the loss $\mathsf{L}$ is not required to align correctly with the model.

\subsection{Relation to logistic regression}
Let $\rho(t)=\log (1+e^t)$ and therefore $\rho'(t)=1/(1+e^{-t})$, $\rho'(-t)=1/(1+e^{t})$. In logistic regression, we observe i.i.d. data $(\tilde{Y}_i,A_i)\in \{0,1\}\times \R^n$ generated according to the model 
\begin{align*}
\Prob\big(\tilde{Y}_i=1|A_i\big)=1/(1+e^{- \iprod{A_i}{\mu_\ast} }) =\rho'(\iprod{A_i}{\mu_\ast}),\quad i \in [m].
\end{align*}
The maximum likelihood estimator $\hat{\mu}^{\mathrm{MLE}}$ solves
\begin{align*}
\hat{\mu}^{\mathrm{MLE}}\in \argmin_{\mu \in \R^n}-\sum_{i \in [m]}\Big\{\tilde{Y}_i\cdot  \log\rho'(\iprod{A_i}{\mu_\ast}) +(1-\tilde{Y}_i)\cdot \log\rho'(-\iprod{A_i}{\mu_\ast})\Big\}.
\end{align*}
It is natural to run gradient descent for the loss function above.

To setup its equivalence to Example \ref{model:1bit}, consider the reparametrization $Y_i\equiv 2 \tilde{Y}_i-1$. Then it is easy to verify that $\{(Y_i,A_i)\}_{i \in [m]}$ are i.i.d. observations from Example \ref{model:1bit} with the errors $\{\xi_i\}$ being i.i.d. random variables with c.d.f. $\Prob(\xi_1\leq t)=1/(1+e^{-t})=\rho'(t)$. Moreover, 
\begin{align*}
\hat{\mu}^{\mathrm{MLE}}\in \argmin_{\mu \in \R^n} \sum_{i \in [m]} \log\big(1+e^{-Y_i\cdot \iprod{A_i}{\mu}}\big) = \argmin_{\mu \in \R^n} \sum_{i \in [m]}  \rho\big(-Y_i\cdot \iprod{A_i}{\mu}\big).
\end{align*}
In other words, by setting the loss function as $\mathsf{L}(x,y)\equiv \rho(-xy)$, $\hat{\mu}^{\mathrm{MLE}}$ is stationary point to the gradient descent algorithm
\begin{align*}
\mu^{(t)}&\equiv \mu^{(t-1)}-\eta \cdot A^\top \partial_1 \mathsf{L}(A\mu^{(t-1)},Y).
\end{align*}
With this identification, it suffices to work with Example \ref{model:1bit} with a general loss function $\mathsf{L}(x,y)$.

\subsection{Distributional characterizations}

Since $\Upsilon_t(\mathfrak{z}^{[0:t]})$ is non-differentiable with respect to $\mathfrak{z}^{(0)}$, the regularity conditions (A5') in Theorem \ref{thm:gd_se}-(2) and (A5*) in Theorem \ref{thm:consist_tau_rho} are not satisfied. However, with a non-trivial smoothing technique, these distributional results continue to hold.

\begin{theorem}\label{thm:1bit_cs}
	Suppose Assumption \ref{assump:setup} and (A4*) hold for some $K,\Lambda \geq 2$, and the loss function $\mathsf{L}$ satisfies
	\begin{align}\label{ineq:111112bLound}
	 \abs{\partial_1 \mathsf{L}(0,0)}+\sup_{x \in \R, y \in [-1,1]}\max_{{\alpha}=2, 3}\abs{\partial_{\alpha}\mathsf{L}(x,y)} \leq \Lambda.
	\end{align}
	\begin{enumerate}
		\item 
		For a sequence of $\Lambda_\psi$-pseudo-Lipschitz functions $\{\psi_k:\R^{2t+1} \to \R\}_{k \in [m\vee n]}$ of order 2 where $\Lambda_\psi\geq 2$, the averaged distributional characterizations in Theorem \ref{thm:gd_se}-(2) hold for the gradient descent iterates in (\ref{def:grad_descent_1bit_cs}), with an error bound $
		\big( K\Lambda \Lambda_\psi L_\mu(1\wedge \sigma_{\mu_\ast})^{-1}\big)^{c_t}\cdot n^{-1/c_t}$. 
		\item Theorem \ref{thm:consist_tau_rho} holds for the output $(\hat{\bm{\tau}}^{[t]},\hat{\bm{\rho}}^{[t]})$ in Algorithm \ref{def:alg_tau_rho} using gradient descent iterates in (\ref{def:grad_descent_1bit_cs}), with an error bound $
		\big( K\Lambda  L_\mu(1\wedge \sigma_{\mu_\ast})^{-1}\big)^{c_t}\cdot n^{-1/c_t}$. 
	\end{enumerate}
\end{theorem}

It is easy to verify that the loss function $\mathsf{L}(x,y)\equiv \rho(-xy)$ used in logistic regression satisfies (\ref{ineq:111112bLound}) (details may be found in Section \ref{subsubsection:verification_logistic}). Moreover, 
for random noises $\{\xi_i\}$'s, the above theorem holds for every realization of these noises. 

At a high level, the smoothing technique used in the proof of Theorem \ref{thm:1bit_cs} proceeds as follows. We construct a sequence of smooth approximation functions $\{\varphi_\epsilon\}_{\epsilon>0}$ such that $\varphi_\epsilon\to \sign$ as $\epsilon \to 0$ in a suitable sense. Then for each $\epsilon>0$, we may compute the state evolution parameters $\texttt{SE}_\epsilon^{(t)}$ associated with the smoothed model according to Definition \ref{def:gd_se}. We also construct `smoothed data' $\{Y_{\epsilon;i}\equiv \varphi_\epsilon(\iprod{A_i}{\mu_\ast}+\xi_i)\}_{i \in [m]}$, and compute `smoothed gradient descent' $\mu^{(t)}_\epsilon$ according to (\ref{def:grad_descent}). Since our theory applies to $\mu^{(t)}_\epsilon$ via $\texttt{SE}_\epsilon^{(t)}$ for any $\epsilon>0$, the key task is to prove that these quantities remain stable as $\epsilon \to 0$. We will prove such stability estimates for $\texttt{SE}_\epsilon^{(t)}$ in Lemma \ref{lem:1bit_cs_se_smooth_diff}, and for $\mu^{(t)}_\epsilon$ in Lemma \ref{lem:1bit_cs_data_smooth_diff}.

Interestingly, the smoothing technique also provides a method to compute the information parameter $\delta_t$ via the limit of `smoothed information parameter' $\delta_{\epsilon;t}$ in Lemma \ref{lem:1bit_cs_delta_general_loss}, as well as the bias parameter $b^{(t)}_{\mathrm{db}}$ in Proposition \ref{prop:1bit_cs_bias_var} below.

\subsection{The information parameter}
Due to the non-differentiability issue of $\Upsilon_t$, it is understood that $\delta_t$ is defined in the sense of Gaussian integration-by-parts formula in (\ref{def:delta_t_alternative}), whenever well-defined. 

To state our formula for $\delta_t$, we need the some further notation. For a general loss function $\mathsf{L}$, let $\Theta_t^\circ: \R^{[0:t]}\to \R$ be defined recursively via the relation
\begin{align}\label{def:Theta_circ}
\Theta_t^\circ(z_{[0:t]})\equiv z_{t}- \sum_{s \in [1:t-1]}\eta_{s-1} \rho_{t-1,s} \cdot   \partial_1 \mathsf{L} \big(\Theta^\circ_s(z_{[0:s]}),z_0\big).
\end{align}
For $v>0$, let $\mathfrak{g}_v$ be the Lebesgue density of $\mathcal{N}(0,v^2)$.

\begin{lemma}\label{lem:1bit_cs_delta_general_loss}
	Suppose (A1), (A3) and (A4*), and (\ref{ineq:111112bLound}) hold. 
	Then with $(\mathsf{Z}_0,\mathsf{Z}_{[1:t]})\sim \mathcal{N}\big(0, \var\big(\E^{(0)}\big[\mathfrak{Z}^{([0:t])}|\mathfrak{Z}^{([1:t])}\big]\big)\big)$, if $\mathfrak{v}_t\equiv \E^{(0)}\var(\mathfrak{Z}^{(0)}|\mathfrak{Z}^{([1:t])})>0$, 
	\begin{align*}
	\delta_t&= -2\phi\eta\cdot \E^{(0)} \Big\{ \mathfrak{g}_{\mathfrak{v}_t}(\xi_{\pi_m}+\mathsf{Z}_0) \cdot  e_t^\top\big(I_t + \eta \cdot  \mathfrak{L}_1(U,\mathsf{Z}_{[1:t]}) \mathfrak{O}_t(\bm{\rho}^{[t-1]})\big)^{-1}\mathfrak{l}_2(U,\mathsf{Z}_{[1:t]})\Big\}. 
	\end{align*}
	Here $U\sim \mathrm{Unif}[-1,1]$ is independent of all other variables, and
	\begin{itemize}
		\item $\mathfrak{L}_1(U,\mathsf{Z}_{[1:t]})\equiv \mathrm{diag}\big(\big\{\partial_{11} \mathsf{L}\big( \Theta_t^\circ(U,\mathsf{Z}_{[1:s]}),U   \big)\}_{s \in [t]}\big) \in \R^{t\times t}$, 
		\item $\mathfrak{l}_2(U,\mathsf{Z}_{[1:t]})\equiv \big(\partial_{12} \mathsf{L}( \Theta_t^\circ(U,\mathsf{Z}_{[1:s]}),U)\big)_{s \in [t]} \in \R^t$.
	\end{itemize}
	For squared loss $\mathsf{L}(x,y)=(x-y)^2/2$, we have $\delta_t=-2 \E^{(0)} \mathfrak{g}_{\sigma_{\mu_\ast}}(\xi_{\pi_m})\cdot \sum_{s \in [1:t]} \tau_{t,s}$, provided that $\mu_\ast\neq 0$.
\end{lemma}

Under the squared loss, the complicated formula of $\delta_t$ simplifies significantly, revealing a structure reminiscent of the linear regression case (cf. Lemma \ref{lem:linear_delta}). Note that using the squared loss intrinsically mis-specifies the model by treating logistic regression as linear regression. Interestingly, \cite{huang2018robust} showed that the global least squares estimator (the convergent point of $\mu^{(t)})$ achieves a near rate-optimal convergence rate for a scaled $\mu_\ast$ in the low dimensional case $m\ll n$ under Gaussian noises.

\subsection{Mean-field statistical inference}

\subsubsection{Estimation of generalization error}
Consider the generalization error (\ref{def:gen_err}) under the loss function $\mathsf{H}(x,y)$. The generalization error estimator $\hat{\mathscr{E}}_{\mathsf{H}}^{(t)}$ takes the form as in (\ref{def:gen_err_est}). Its validity in the following proposition is formally justified by a smoothing argument similar to that of Theorem \ref{thm:1bit_cs}.

\begin{proposition}\label{prop:gen_error_linear}
	Assume the same conditions as in Theorem \ref{thm:1bit_cs}-(1). Suppose that $\mathsf{H}\in C^3(\R^2)$ has mixed derivatives of order 3 all bounded by $\Lambda$. Then  for any $q\in \N$, there exists some $c_t=c_t(t,q)>1$ such that
	\begin{align*}
	\E^{(0)}\abs{\hat{\mathscr{E}}_{\mathsf{H}}^{(t)}- \mathscr{E}_{\mathsf{H}}^{(t)}(A,Y) }^q\leq \big(K\Lambda L_\mu (1\wedge \sigma_{\mu_\ast})^{-1}\big)^{c_t}\cdot n^{-1/c_1}.
	\end{align*}
\end{proposition}

As the logistic loss $\mathsf{L}(x,y)=\log(1+e^{-xy})$ satisfies (\ref{ineq:111112bLound}), the above proposition holds for $\mathsf{H}=\mathsf{L}$ in logistic regression.

\subsubsection{Debiased gradient descent inference}

Consider the oracle debiased gradient descent iterate $\mu^{(t)}_{\mathrm{db}}$ or its data-driven version $\hat{\mu}^{(t)}_{\mathrm{db}}$ taking the form (\ref{def:debias_gd}). 

To state our result, recall $\bm{\delta}^{[t]}\equiv (\delta_s)_{s \in [1:t]}$ given by Lemma \ref{lem:1bit_cs_delta_general_loss}. Let $\bm{\omega}^{[t]}\equiv (\bm{\tau}^{[t]})^{-1}$, where with $\bm{L}_{11;k}^{[t]}(z_{[0:t]})\equiv \mathrm{diag}\big(\big\{ \partial_{11} \mathsf{L}(\Theta_{s;k}(z_{[0:t]}),\sign(z_0+\xi_k)) \big\}_{s\in [1:t]}\big)$,
\begin{align}\label{def:1bit_tau}
\bm{\tau}^{[t]}\equiv -\phi\eta \cdot \E^{(0)} \big(I_t + \eta\cdot  \bm{L}_{11;\pi_m}^{[t]}(\mathfrak{Z}^{([0:t])}) \mathfrak{O}_t(\bm{\rho}^{[t-1]})\big)^{-1}\bm{L}_{11;\pi_m}^{[t]}(\mathfrak{Z}^{([0:t])}).
\end{align}
Let the bias $b^{(t)}_{\mathrm{db}}\equiv -  \iprod{\bm{\omega}^{[t]}\bm{\delta}^{[t]} }{e_t}$, and the variance $(\sigma^{(t)}_{\mathrm{db}})^2\equiv \bm{\omega}_{t\cdot}^{[t]}\Sigma_{\mathfrak{W}}^{[t]}\bm{\omega}_{t\cdot}^{[t],\top}$.

\begin{proposition}\label{prop:1bit_cs_bias_var}
	Suppose Assumption \ref{assump:setup} and (A4*), and (\ref{ineq:111112bLound}) hold. 
	\begin{enumerate}
		\item For a sequence of $\Lambda_\psi$-pseudo-Lipschitz functions $\{\psi_\ell:\R \to \R\}_{\ell \in [n]}$ of order 2 where $\Lambda_\psi\geq 2$, the averaged distributional characterizations for $\mu^{(t)}_{\mathrm{db}}$ in Theorem \ref{thm:db_gd_oracle}-(2) hold with $b^{(t)}_{\mathrm{db}},
		\sigma^{(t)}_{\mathrm{db}}$ specified as above, and with an error bound $\big( K\Lambda \Lambda_\psi L_\mu(1\wedge \sigma_{\mu_\ast})^{-1}(1\wedge \tau_\ast^{(t)})^{-1}\big)^{c_t}\cdot n^{-1/c_t}$. 
		\item For the squared loss $\mathsf{L}(x,y)=(x-y)^2/2$, if $\sigma_{\mu_\ast}>0$, 
		\begin{align*}
		b^{(t)}_{\mathrm{db}}= 2 \E^{(0)} \mathfrak{g}_{\sigma_{\mu_\ast}}(\xi_{\pi_m}),\, (\sigma^{(t)}_{\mathrm{db}})^2= \phi^{-1} \E^{(0)}\big(\mathfrak{Z}^{(t)}-\sign(\mathfrak{Z}^{(0)}+\xi_{\pi_m})\big)^2.
		\end{align*}
	\end{enumerate}
	
\end{proposition}

In logistic regression, the generic proposal (\ref{def:db_gd_b_est}) can be used to estimate the bias parameter $b^{(t)}_{\mathrm{db}}$. Estimating the signal strength $\sigma_{\mu_\ast}$ in mean-field logistic regression is a non-trivial and separate problem; several notable methods include the \texttt{ProbeFrontier} method developed in \cite[Section 3.1]{sur2019modern} and the \texttt{SLOE} method developed in \cite{yadlowsky2021sloe}. However, estimation of $\sigma_{\mu_\ast}$ is unnecessary when the parameter of interest is $\mu_\ast/\pnorm{\mu_\ast}{}$ as in the single-index model \cite{bellec2025observable}. 

On the other hand, the variance parameter $(\sigma^{(t)}_{\mathrm{db}})^2$ can be estimated using the generic proposal (\ref{def:db_gd_var_est}). Under the squared loss, the variance parameter can also be estimated simply by the rescaled generalization error estimate $\phi^{-1} \hat{\mathscr{E}}_{\abs{\cdot}^2}^{(t-1)}$ as in linear regression. Since the bias parameter $b^{(t)}_{\mathrm{db}}$ remains constant across all $t$ under the squared loss (cf.~Proposition~\ref{prop:1bit_cs_bias_var}-(2)), the associated CI again has a length proportional to the generalization error.

	\section{Numerical experiments}\label{sec:numerics}
	In this section, we conduct numerical experiments to evaluate the performance of our proposed debiased gradient descent inference methods in Section \ref{section:iterative_inf} across a variety of  statistical models in Sections \ref{section:example_linear} and \ref{section:example_1bit} with both convex and non-convex losses.

	\subsection{Common numerical settings}
	
	We set the sample size as $m = 1200$ and the signal dimension as $n = 1000$. The ground truth signal $\mu_\ast \in \mathbb{R}^n$ is generated with i.i.d. entries sampled from $\abs{\mathcal{N}(0,5)}$. The scaled random design matrix $\sqrt{n}A$ has i.i.d. entries following $\mathcal{N}(0,1)$ (orange), $t$ distribution with $10$ degrees of freedom  (blue), Bernoulli$(1/2)$ (purple). The colors in parentheses correspond to those used in the figures.  Proper normalization is applied to the latter two cases so that the variance is $1$. The gradient descent inference algorithm is run for $50$ iterations with random Gaussian initialization $\mu^{(0)}\sim \mathcal{N}(0,I_n)$, and Monte Carlo repetition $B = 1000$.
	
	\subsection{Linear regression model}  We examine the performance of the gradient descent inference algorithm in linear regression with squared loss and the following robust loss, known as the pseudo-Huber function:
	\begin{align*}
	\mathsf{L}_\delta(x,y)\equiv \delta^2\big(\{1+(x-y)^2/\delta^2\}^{1/2}-1\big),\quad \forall x, y \in \R.
	\end{align*}
	Clearly $\mathsf{L}_\delta$ satisfies the regularity conditions in Theorem \ref{thm:linear} for any $\delta>0$. For definiteness, here we use $\delta=1$ as in the numerical experiment in \cite[Section 4]{tan2024estimating}. For squared loss, the noise vector $\xi \in \mathbb{R}^m$ has i.i.d. entries from $\mathcal{N}(0,1)$. As the pseudo-Huber loss function is designed to  accommodate heavy-tailed errors, here we choose the noises $\{\xi_i\}$ as i.i.d. $t$-distributed random variables with only 2 degrees of freedom. The step size is fixed at $\eta = 0.3$. 
	
		\begin{figure}[t]
		\begin{minipage}[t]{0.3\textwidth}
			\includegraphics[width=\textwidth]{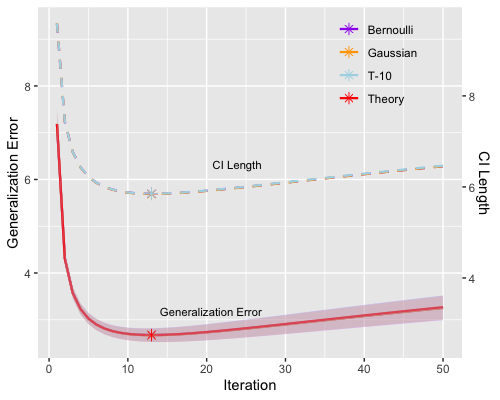}
		\end{minipage}
		\begin{minipage}[t]{0.3\textwidth}
			\includegraphics[width=\textwidth]{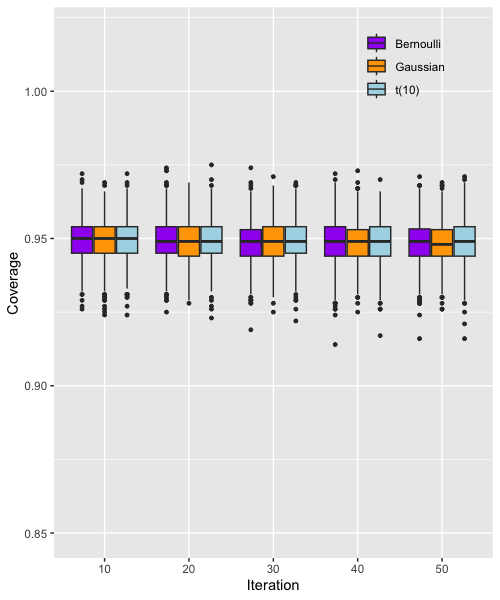}
		\end{minipage}
		\begin{minipage}[t]{0.3\textwidth}
			\includegraphics[width=\textwidth]{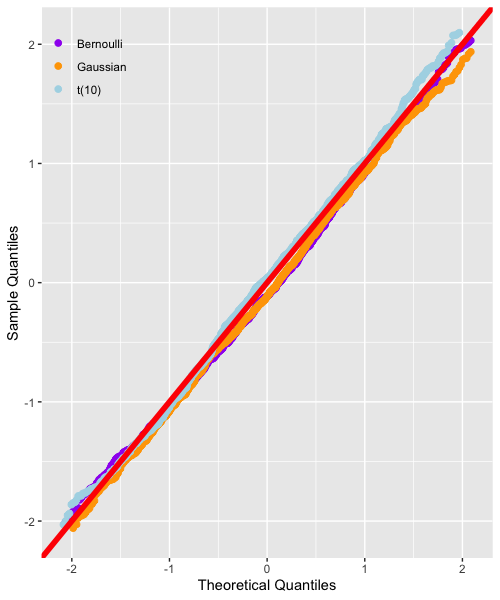}
		\end{minipage}
		\begin{minipage}[t]{0.3\textwidth}
			\includegraphics[width=\textwidth]{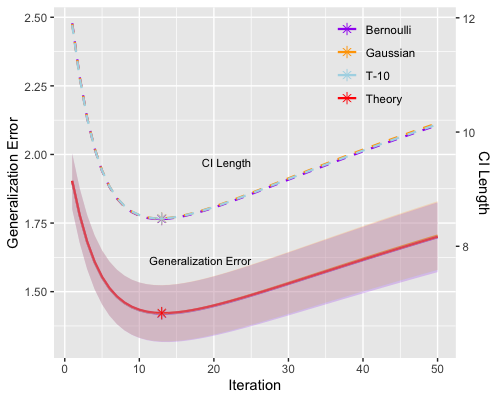}
		\end{minipage}
		\begin{minipage}[t]{0.3\textwidth}
			\includegraphics[width=\textwidth]{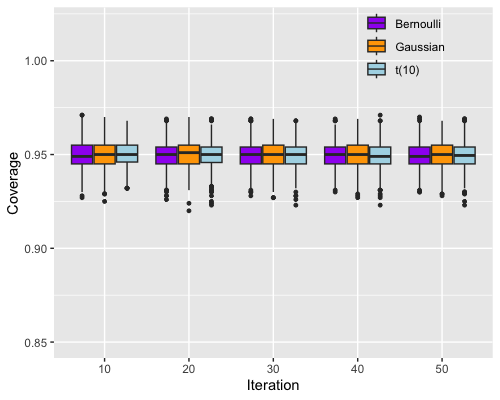}
		\end{minipage}
		\begin{minipage}[t]{0.3\textwidth}
			\includegraphics[width=\textwidth]{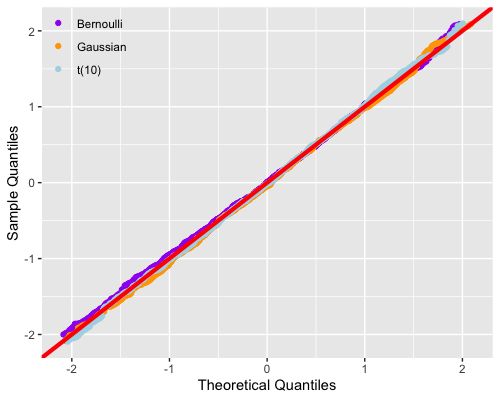}
		\end{minipage}
		\caption{Linear regression. \emph{Top row}: Squared loss. \emph{Bottom row}: Pseudo-Huber loss. }
		\label{fig:1}
	\end{figure}
	
	We report in Figure \ref{fig:1} the simulation results under both squared loss and pseudo-Huber loss for linear regression without regularization:
	\begin{itemize}
		\item 	The left panel of Figure~\ref{fig:1} compares the estimated generalization error $\hat{\mathscr{E}}_{\mathsf{L}}^{(t)}$ with the theoretical generalization error $\mathscr{E}_{\mathsf{L}}^{(t)}$, together with the CI length at each iteration.
		For squared loss (top left), the iteration that minimizes the generalization error coincides with the iteration yielding the shortest confidence interval; this is consistent with our theory in Proposition \ref{prop:linear_bias_var}-(1).

		\item  Using $\hat{\sigma}^{(t)}_{\mathrm{db}}$ computed from (\ref{def:db_gd_var_est}), the middle panel displays the coverage of $95\%$ CIs for all coordinates. For all design distributions, the proposed CIs in (\ref{def:linear_CI}) achieve approximate nominal coverage. 
		\item The right panel validates the approximate normality of $\hat{\mu}^{(t)}_{\mathrm{db}}$. Since the bias and the variance are identical across coordinates, we report only the QQ-plot of $(\hat{\mu}^{(t)}_{\mathrm{db};1} - \mu_{\ast;1}) / \hat{\sigma}^{(t)}_{\mathrm{db}}$ at the last iteration of our simulation. The empirical quantiles align closely with the theoretical standard Gaussian quantiles, supporting the conclusion that  $
		\hat{\mu}^{(t)}_{\mathrm{db}} \stackrel{d}{\approx} \mu_\ast+ \sigma^{(t)}_{\mathrm{db}}\cdot  \mathsf{Z}$.
	\end{itemize}
	We note that the proposal in \cite{tan2024estimating} does not cover the approximate normality of the debiased gradient descent and its associated CI's for the pseudo-Huber loss, as shown in the middle and right panels of Figure \ref{fig:1}. Additionally, while our theory does not strictly apply to the $t(10)$ distribution due to its heavy tails, the simulation results shown in these plots suggest that the moment condition in (A2) could potentially be further relaxed for our theory to hold.

	\subsection{Single-index regression model}  We next evaluate the performance of the gradient descent inference algorithm under the single-index regression model with squared loss. We consider two choices for the link function: the sigmoid function $\varphi_\ast(x) = 1/(1 + e^{-x})$, and a smooth nonlinear function $\varphi_\ast(x) = x + \sin(x)$. For both link functions, it can be verified that $\mathsf{L}_\ast$ satisfies the regularity conditions in Theorem \ref{thm:linear}. In this simulation, we set the step size to $\eta = 1.5$, and the noise vector $\xi \in \mathbb{R}^m$ has i.i.d. entries drawn from $\mathcal{N}(0, 0.1)$. Simulations use known signal strength $\sigma_{\mu_\ast}$ for illustration.
	
		\begin{figure}[t]
		\begin{minipage}[t]{0.3\textwidth}
			\includegraphics[width=\textwidth]{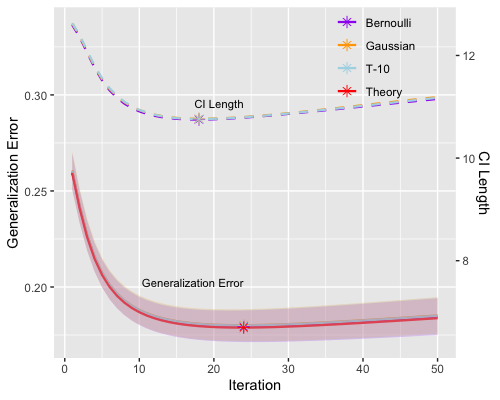}
		\end{minipage}
		\begin{minipage}[t]{0.3\textwidth}
			\includegraphics[width=\textwidth]{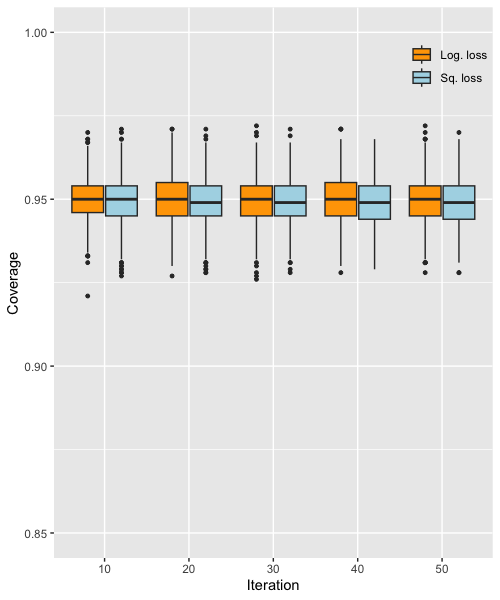}
		\end{minipage}
		\begin{minipage}[t]{0.3\textwidth}
			\includegraphics[width=\textwidth]{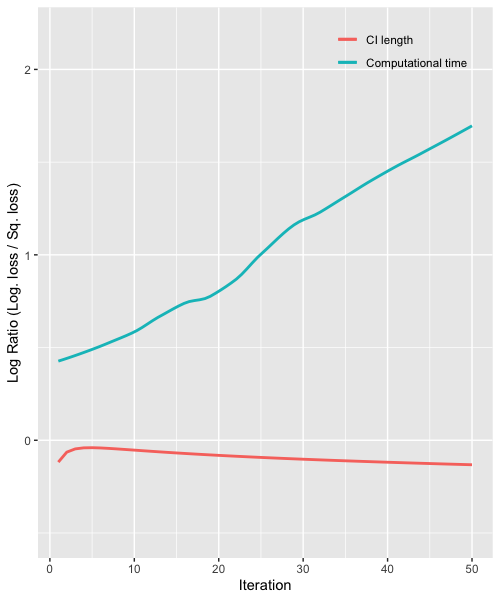}
		\end{minipage}
		\begin{minipage}[t]{0.3\textwidth}
			\includegraphics[width=\textwidth]{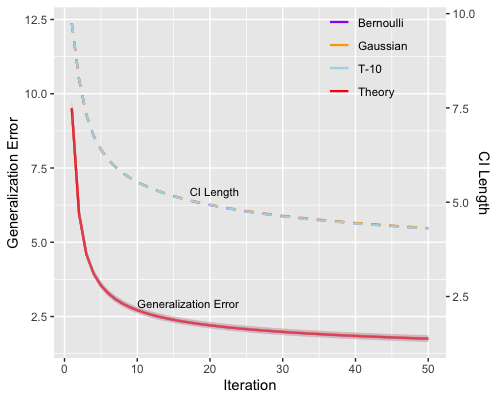}
		\end{minipage}
		\begin{minipage}[t]{0.3\textwidth}
			\includegraphics[width=\textwidth]{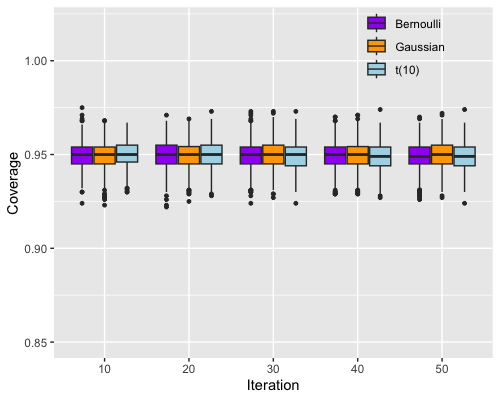}
		\end{minipage}
		\begin{minipage}[t]{0.3\textwidth}
			\includegraphics[width=\textwidth]{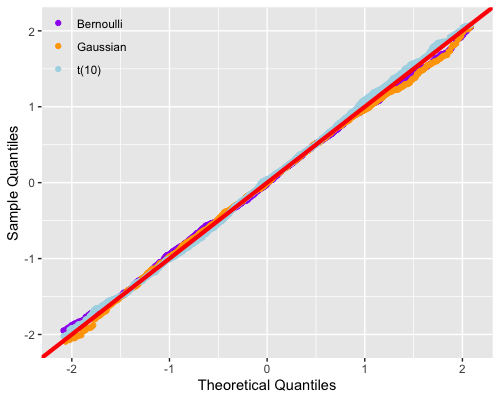}
		\end{minipage}
		\caption{Single-index regression model with squared loss. \emph{Top row}: sigmoid link $\varphi_\ast(x) = 1/(1 + e^{-x})$. \emph{Bottom row}: nonlinear link $\varphi_\ast(x) = x + \sin(x)$.}
		\label{fig:2}
	\end{figure}
	
	Figure~\ref{fig:2} presents the results for both link functions:
	\begin{itemize}
		\item The left panel plots the generalization error and CI length across iterations. For the sigmoid link, the curves reach their minima at different iterations, indicating a mismatch between the generalization-optimal and inference-optimal stopping points.
		\item The middle and right panels show that the proposed CIs achieve close-to-nominal coverage and that the standardized debiased estimator exhibits approximate normality similar to the results observed in the linear regression plots.
	\end{itemize}

	\subsection{Logistic regression model} 	
		\begin{figure}[t]
		\begin{minipage}[t]{0.3\textwidth}
			\includegraphics[width=\textwidth]{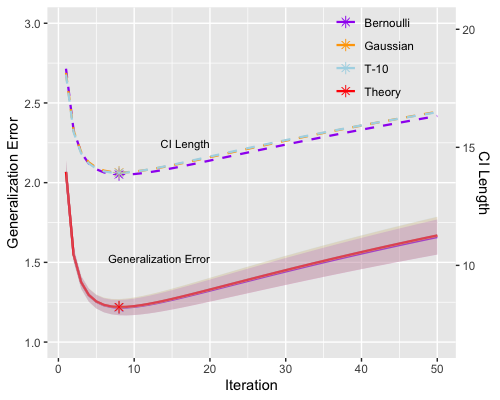}
		\end{minipage}
		\begin{minipage}[t]{0.3\textwidth}
			\includegraphics[width=\textwidth]{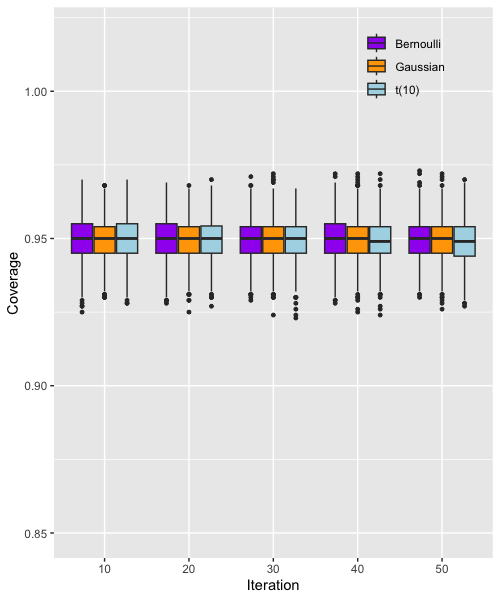}
		\end{minipage}
		\begin{minipage}[t]{0.3\textwidth}
			\includegraphics[width=\textwidth]{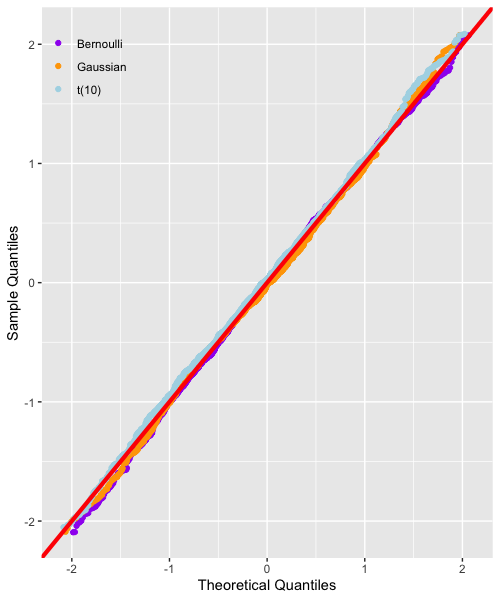}
		\end{minipage}
		\begin{minipage}[t]{0.3\textwidth}
			\includegraphics[width=\textwidth]{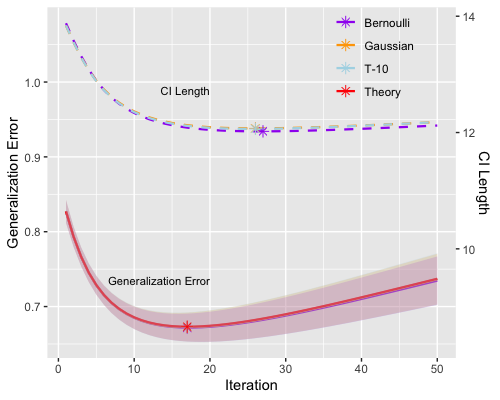}
		\end{minipage}
		\begin{minipage}[t]{0.3\textwidth}
			\includegraphics[width=\textwidth]{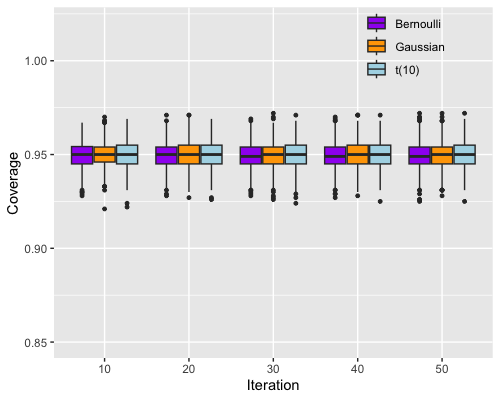}
		\end{minipage}
		\begin{minipage}[t]{0.3\textwidth}
			\includegraphics[width=\textwidth]{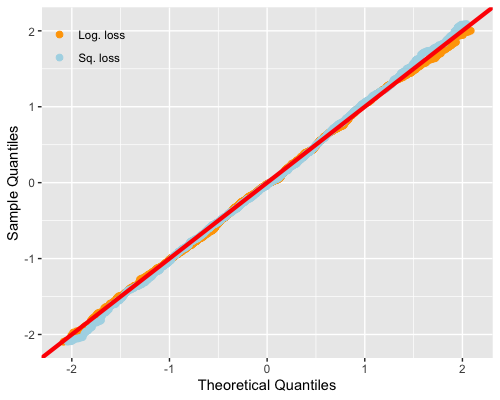}
		\end{minipage}
		\caption{Logistic regression. \emph{Top row}: Squared loss. \emph{Bottom row}: Logistic loss.}
		\label{fig:3}
	\end{figure}
	Finally, we evaluate the performance of the gradient descent inference algorithm in logistic regression model with both squared loss and logistic loss. The step size is set to $\eta = 0.2$. We report in Figure \ref{fig:3} some simulation results under both loss functions without regularization:
	\begin{itemize}
		\item In the left panel, we plot the estimated generalization error and CI length over iterations. For squared loss, the minimizing points for generalization error and CI length are closely aligned, as predicted by Proposition~\ref{prop:1bit_cs_bias_var}-(2). In contrast, for logistic loss, the CI length reaches its minimum after the generalization error does. Note that this is opposite to the sigmoid link case in Figure \ref{fig:2}, and suggests that the relative position of the two minimizing points can vary across loss functions and model structures.
		\item The middle and right panels show that the proposed CIs maintain nominal coverage and that the standardized estimator is approximately normal. Additionally, the right panel in the first row shows
		\begin{align*}
		\log_{10} \frac{@ \text{ under logistic loss}}{@ \text{ under squared loss}},\quad @ \in \{\text{CI length},  \text{Computational time}\}.
		\end{align*}
		This plot highlights two interesting observations: (i) the CI lengths under both losses are comparable, and (ii) the squared loss offers significant computational advantages over the logistic loss. Specifically, in Algorithm \ref{def:alg_tau_rho}, $\hat{\bm{\tau}}^{[t]}$ can be computed in a single step under the squared loss, whereas all $\{\hat{\bm{\tau}}^{[t]}_k\}_{k \in [m]}$ need be computed under the logistic loss. This advantage becomes increasingly pronounced with more iterations; for instance, by iteration 50, computation under the logistic loss is approximately 50 times slower in our simulation.
	\end{itemize}
	
	Our simulations are conducted with known signal strength $\sigma_{\mu_\ast}$ to illustrate the numerical features inherent to our inference methods. As mentioned earlier, estimating $\sigma_{\mu_\ast}$ is a separate problem, and can be tackled by existing methods such as \texttt{ProbeFrontier} \cite{sur2019modern} and \texttt{SLOE} \cite{yadlowsky2021sloe}. Furthermore, while the right panel in Figure \ref{fig:3} display results under Gaussian designs $A$, the findings remain nearly identical for non-Gaussian designs, such as Bernoulli or $t(10)$ variables. For brevity, we omit these additional plots. 
	
	In Appendix~\ref{section:additional_simulation}, we present additional simulation results for the above settings with $\ell_1$ regularization (i.e., $\mathsf{f} \equiv \lambda \abs{\cdot}_1$ in (\ref{def:grad_descent_1bit_cs})), and further evaluate the numerical performance of our proposed inference method in the one-bit compressed sensing model in Example~\ref{model:1bit}, when the errors $\{\xi_i\}$ are i.i.d. standard Gaussian as in \cite{huang2018robust}. In the latter setting, a closed-form estimator for $\sigma_{\mu_\ast}$ can be easily constructed (see Eqn.~(\ref{def:1bit_signal_strength_est_gaussan}) for details). These simulation results, reported in Figures~\ref{fig:4}-\ref{fig:7}, exhibit qualitatively similar patterns to those observed in Figures~\ref{fig:1}-\ref{fig:3}.

\section{Proofs for Section \ref{section:gd_dynamics}}\label{section:proof_gd_dynamics}

\subsection{An apriori estimate}
The following apriori estimates are important for the proofs of many results in the sequel. Moreover, the calculations in the proof of these estimates provide the precise formulae that lead to Algorithm \ref{def:alg_tau_rho}.

\begin{lemma}\label{lem:rho_tau_bound}
	Suppose (A1), (A3), (A4') and (A5') hold. The following hold for some $c_t=c_t(t)>1$:
\begin{enumerate}
	\item $ \pnorm{\bm{\tau}^{[t]}}{\op}+ \pnorm{\bm{\rho}^{[t]}}{\op}\leq (K\Lambda)^{c_t}$.
	\item $\pnorm{\Sigma_{\mathfrak{Z}}^{[t]}}{\op}+ \pnorm{\Sigma_{\mathfrak{W}}^{[t]}}{\op}+ \pnorm{\bm{\delta}^{[t]}}{}\leq (K\Lambda L_\mu )^{c_t}$.
	\item $\max\limits_{k \in [m], r \in [1:t]} \big(\abs{\Upsilon_{r;k}(z^{([0:r])})}+\abs{\Theta_{r;k}(z^{([0:r])})}\big)\leq (K\Lambda)^{c_t}\cdot \big(1+\pnorm{z^{([0:t])} }{}\big)$.
	\item $\max\limits_{\ell \in [n], r \in [1:t]} \abs{\Omega_{r;\ell}(w^{([1:r])})}\leq (K\Lambda L_\mu)^{c_t}\cdot \big(1+\pnorm{w^{([1:t])}}{}\big)$.
	\item $\max\limits_{\substack{k \in [m], \ell \in [n], r,s \in [1:t]}} \big(\abs{\partial_{(s)}\Upsilon_{r;k}(z^{([0:r])})}+\abs{\partial_{(s)}\Omega_{r;\ell}(w^{([1:r])})}\big)\leq (K\Lambda)^{c_t}.$
\end{enumerate}

\end{lemma}
\begin{proof}
The proof is divided into several steps. For notational simplicity, we write $\partial_{(s)}\Upsilon_{t;k}(z^{([0:t])})\equiv \partial_{z^{(s)}}\Upsilon_{t;k}(z^{([0:t])})$, and similarly $\partial_{(s)}\Omega_{t;\ell}(w^{([1:t])})\equiv \partial_{w^{(s)}}\Omega_{t;\ell}(w^{([1:t])})$.
	
\noindent (\textbf{Step 1}). We prove the estimate in (1) in this step. First, for any $k\in [m]$ and $1\leq s\leq t$, using (S1),
\begin{align*}
\partial_{(s)}\Upsilon_{t;k}(z^{([0:t])})&\equiv - \eta_{t-1}  \partial_{11} \mathsf{L}_{t-1}\big(\Theta_{t;k}(z^{([0:t])}) ,\mathcal{F}(z^{(0)},\xi_k)\big)\\
&\qquad \times \bigg(\bm{1}_{t=s}+  \sum_{r \in [1:t-1]}\rho_{t-1,r} \partial_{(s)} \Upsilon_{r;k}(z^{([0:r])})\bigg).
\end{align*}
In the matrix form, with $\bm{\Upsilon}_{k}^{';[t]}(z^{([0:t])}) \equiv \big(\partial_{(s)}\Upsilon_{r;k}(z^{([0:r])})\big)_{r,s \in [1:t]}$ and $\bm{L}^{[t]}_{k}(z^{([0:t])})\equiv \mathrm{diag}\big(\big\{-\eta_{s-1}  \partial_{11} \mathsf{L}_{s-1}\big(\Theta_{s;k}(z^{([0:s])}) ,\mathcal{F}(z^{(0)},\xi_k)\big)\big\}_{s \in [1:t]}\big)$, 
\begin{align*}
\bm{\Upsilon}_{k}^{';[t]}(z^{([0:t])}) =  \bm{L}^{[t]}_{k}(z^{([0:t])}) + \bm{L}^{[t]}_{k}(z^{([0:t])}) \mathfrak{O}_{t}(\bm{\rho}^{[t-1]}) \bm{\Upsilon}_{k}^{';[t]}(z^{([0:t])}).
\end{align*}
Solving for $\bm{\Upsilon}_{k}^{';[t]}$ yields that
\begin{align}\label{ineq:rho_tau_bound_Upsilon}
\bm{\Upsilon}_{k}^{';[t]} (z^{([0:t])}) = \big[I_t- \bm{L}^{[t]}_{k}(z^{([0:t])}) \mathfrak{O}_{t}(\bm{\rho}^{[t-1]})\big]^{-1} \bm{L}^{[t]}_{k}(z^{([0:t])}).
\end{align}
As $\bm{L}^{[t]}_{k}(z^{([0:t])}) \mathfrak{O}_{t}(\bm{\rho}^{[t-1]})$ is a lower triangular matrix with $0$ diagonal elements, $\big(\bm{L}^{[t]}_{k}(z^{([0:t])}) \mathfrak{O}_{t}(\bm{\rho}^{[t-1]})\big)^t=0_{t\times t}$, and therefore using $\pnorm{\bm{L}^{[t]}_{k}(z^{([0:t])})}{\op}\leq \Lambda^2$,
\begin{align}\label{ineq:rho_tau_bound_Upsilon_1}
\pnorm{\bm{\Upsilon}_{k}^{';[t]} (z^{([0:t])}) }{\op}&\leq \bigg(1+\sum_{r \in [1:t]} \Lambda^{2r} \pnorm{ \bm{\rho}^{[t-1]}}{\op}^r\bigg)\cdot \Lambda^2\leq \Lambda^{c_0 t}\cdot   \pnorm{ \bm{\rho}^{[t-1]}}{\op}^t.
\end{align}
Using definition of $\{\tau_{r,s}\}$, we then arrive at
\begin{align}\label{ineq:rho_tau_bound_1}
\pnorm{\bm{\tau}^{[t]}}{\op}\leq (K\Lambda)^{c_0 t}\cdot   \pnorm{ \bm{\rho}^{[t-1]}}{\op}^t.
\end{align}
Next, for any $\ell \in [n]$ and $1\leq s\leq t$, using (S3), 
\begin{align*}
\partial_{(s)}\Omega_{t;\ell}(w^{([1:t])})&=\mathsf{P}_{t;\ell}'\big(\Delta_{t;\ell} (w^{([1:t])}) \big)\\
&\quad \times \bigg(\bm{1}_{t=s}+ \sum_{r \in [1:t]} (\tau_{t,r}+\bm{1}_{t=r})\cdot  \partial_{(s)}\Omega_{r-1;\ell}(w^{([1:r-1])}) \bigg).
\end{align*}
In the matrix form, with $\bm{\Omega}_{\ell}^{';[t]}(w^{([1:t])})\equiv \big(\partial_{(s)}\Omega_{r;\ell}(w^{([1:r])})\big)_{r,s \in [1:t]}$ and $\bm{P}^{[t]}_{\ell}(w^{([1:t])})\equiv \mathrm{diag}\big(\big\{\mathsf{P}_{s;\ell}'(\Delta_{s;\ell} (w^{([1:t])}) )\big\}_{s \in [1:t]}\big)$, 
\begin{align}\label{ineq:rho_tau_bound_Omega}
\bm{\Omega}_{\ell}^{';[t]}(w^{([1:t])}) = \bm{P}^{[t]}_{\ell}(w^{([1:t])})\big[I_t+(\bm{\tau}^{[t]}+I_t)\mathfrak{O}_t\big(\bm{\Omega}_{\ell}^{';[t-1]}(w^{([1:t-1])}) \big)\big].
\end{align}
Consequently, 
\begin{align*}
\pnorm{\bm{\Omega}_{\ell}^{';[t]}(w^{([1:t])})}{\op}\leq \Lambda\cdot \Big(1+(1+ \pnorm{\bm{\tau}^{[t]}}{\op})\cdot  \pnorm{\bm{\Omega}_{\ell}^{';[t-1]}(w^{([1:t-1])})}{\op}\Big).
\end{align*}
Iterating the bound and using the trivial initial condition $\pnorm{\bm{\Omega}_{\ell}^{';(1)}(w^{(1)})}{\op}\leq \Lambda$, 
\begin{align}\label{ineq:rho_tau_bound_Omega_1}
\pnorm{\bm{\Omega}_{\ell}^{';[t]}(w^{([1:t])})}{\op}\leq \big(\Lambda (1+ \pnorm{\bm{\tau}^{[t]}}{\op})\big)^{c_t}.
\end{align}
Using the definition of $\{\rho_{r,s}\}$, we then have
\begin{align}\label{ineq:rho_tau_bound_2}
\pnorm{ \bm{\rho}^{[t]}}{\op}\leq \big(\Lambda (1+ \pnorm{\bm{\tau}^{[t]}}{\op})\big)^{c_t}.
\end{align}
Combining (\ref{ineq:rho_tau_bound_1}) and (\ref{ineq:rho_tau_bound_2}), it follows that 
\begin{align*}
\pnorm{\bm{\tau}^{[t]}}{\op}\leq (K\Lambda)^{c_t }\cdot   (1+ \pnorm{\bm{\tau}^{[t-1]}}{\op})^{c_t}.
\end{align*}
Iterating the bound and using the initial condition $\pnorm{\bm{\tau}^{[1]}}{\op}\leq \Lambda^2$ to conclude the bound for $\pnorm{\bm{\tau}^{[t]}}{\op}$. The bound for $\pnorm{\bm{\rho}^{[t]}}{\op}$ then follows from (\ref{ineq:rho_tau_bound_2}).

 \noindent (\textbf{Step 2}). In this step we note the following recursive estimates:
	\begin{enumerate}
		\item[(a)] A direct induction argument for (S1) shows that 
		\begin{align*}
		\max_{k \in [m]} \max_{r \in [1:t]} \abs{\Upsilon_{r;k}(z^{([0:r])})}\leq (K\Lambda)^{c_t}\cdot \big(1+\pnorm{z^{([0:t])} }{}\big).
		\end{align*}
		\item[(b)]  A direct induction argument for (S3) shows that 
		\begin{align*}
		\max_{\ell \in [n]} \max_{r \in [1:t]} \abs{\Omega_{r;\ell}(w^{([1:r])})}&\leq (K\Lambda L_\mu)^{c_t}\cdot \big(1+\pnorm{w^{([1:t])} }{}+ \pnorm{\bm{\delta}^{[t]}}{} \big).
		\end{align*}
		\item[(c)]  Using (S2), we have
		\begin{align*}
		\pnorm{\Sigma_{\mathfrak{Z}}^{[t]}}{\op}&\leq (K\Lambda L_\mu)^{c_t}\cdot \big(1+\pnorm{\Sigma_{\mathfrak{W}}^{[t-1]}}{\op}+ \pnorm{\bm{\delta}^{[t-1]}}{} \big),\\
		\pnorm{\Sigma_{\mathfrak{W}}^{[t]}}{\op} &\leq (K\Lambda)^{c_t}\cdot \big(1+\pnorm{\Sigma_{\mathfrak{Z}}^{[t]}}{\op} \big).
		\end{align*}
	\end{enumerate}

\noindent (\textbf{Step 3}). In order to use the recursive estimates in Step 2, in this step we prove the estimate for $\pnorm{\bm{\delta}^{[t]}}{}$. For $k \in [m]$, let $\mathsf{L}_{t-1;k}^{\mathcal{F}}(u_1,u_2)\equiv \mathsf{L}_{t-1;k}(u_1,\mathcal{F}(u_2,\xi_k))$. Then by assumption, the mapping $(u_1,u_2)\mapsto \partial_1 \mathsf{L}_{t-1;k}^{\mathcal{F}}(u_1,u_2)$ is $\Lambda$-Lipschitz on $\R^2$. By using (S1), we then have
\begin{align*}
\abs{\partial_{(0)}\Upsilon_{t;k}(z^{([0:t])})}&= \biggabs{ - \eta_{t-1}  \partial_{11} \mathsf{L}_{t-1}^{\mathcal{F}}\big(\Theta_{t;k}(z^{([0:t])}) ,z^{(0)})\big) \cdot \bigg( \sum_{r \in [1:t-1]}\rho_{t-1,r} \partial_{(0)} \Upsilon_{r;k}(z^{([0:r])})\bigg)\\
	&\qquad -\eta_{t-1}\partial_{12} \mathsf{L}_{t-1}^{\mathcal{F}}\big(\Theta_{t;k}(z^{([0:t])}) ,z^{(0)})\big)}\\
&\leq t \Lambda^2\cdot \pnorm{\bm{\rho}^{[t-1]} }{\op}\cdot \max_{r \in [1:t-1]} \abs{\partial_{(0)} \Upsilon_{r;k}(z^{([0:r])}) }.
\end{align*}
Invoking the proven estimate in (1) and iterating the above bound with the trivial initial condition $\abs{\partial_{(0)}\Upsilon_{1;k}(z^{([0:1])})}\leq \Lambda^2$ to conclude that $
\abs{\partial_{(0)}\Upsilon_{t;k}(z^{([0:t])})}\leq (K\Lambda)^{c_t}$, and therefore by definition of $\delta_t$, we conclude that
\begin{align}\label{ineq:rho_tau_bound_3}
\pnorm{\bm{\delta}^{[t]}}{}\leq (K\Lambda)^{c_t}.
\end{align}

\noindent (\textbf{Step 4}). Now we shall use the estimate for $\pnorm{\bm{\delta}^{[t]}}{}$ in  Step 3 to run the recursive estimates in Step 2. Combining the first line of (c) and (\ref{ineq:rho_tau_bound_3}), we have
	\begin{align*}
	\pnorm{\Sigma_{\mathfrak{Z}}^{[t]}}{\op}&\leq (K\Lambda L_\mu)^{c_t}\cdot \big(1+\pnorm{\Sigma_{\mathfrak{W}}^{[t-1]}}{\op} \big).
	\end{align*}
	Combined with the second line of (c), we obtain
	\begin{align*}
	\pnorm{\Sigma_{\mathfrak{Z}}^{[t]}}{\op}&\leq (K\Lambda L_\mu)^{c_t}\cdot \Big(1+\max_{r\in [1:t-1]}\pnorm{\Sigma_{\mathfrak{Z}}^{[r]}}{\op}\Big).
	\end{align*}
	Coupled with the initial condition $\pnorm{\Sigma_{\mathfrak{Z}}^{[1]}}{\op}\leq L_\mu^2$, we arrive at the estimate
	\begin{align*}
	\pnorm{\Sigma_{\mathfrak{Z}}^{[t]}}{\op}+ \pnorm{\Sigma_{\mathfrak{W}}^{[t]}}{\op}+ \pnorm{\bm{\delta}^{[t]}}{}\leq (K\Lambda L_\mu)^{c_t}.
	\end{align*}
	The proof of (1)-(4) is complete by collecting the estimates. The estimate in (5) follows by combining (\ref{ineq:rho_tau_bound_Upsilon_1}) and (\ref{ineq:rho_tau_bound_Omega_1}) along with the estimate in (1).
\end{proof}

\subsection{Proof of Theorem \ref{thm:gd_se}}

The proof of Theorem \ref{thm:gd_se} relies on the general state evolution theory developed in \cite{han2025entrywise}. For the convenience of the reader, we review some of its basics in Appendix \ref{section:GFOM_se}.

\begin{proof}[Proof of Theorem \ref{thm:gd_se}]
The proof is divided into several steps.

\noindent (\textbf{Step 1}). Let us now rewrite the proximal gradient descent algorithm (\ref{def:grad_descent}) into the canonical form in which the state evolution theory in \cite{han2025entrywise} can be applied. Consider initialization $u^{(-1)}=0_m$, $v^{(-1)}=\mu_\ast$ at iteration $t=-1$. For $t=0$, let  $u^{(0)}\equiv A v^{(-1)}=A\mu_\ast$, $v^{(0)}\equiv \mu^{(0)}$. For $t\geq 1$, 
\begin{align*}
u^{(t)}&\equiv A \mathsf{R}_{t-1}(v^{(t-1)})\in \R^m,\\ v^{(t)}&\equiv \mathsf{R}_{t-1}(v^{(t-1)})-\eta_{t-1}\cdot  A^\top \partial_1 \mathsf{L}_{t-1}\big(u^{(t)},\mathcal{F}(u^{(0)},\xi)\big)\in \R^n.
\end{align*}
Here $\mathsf{R}_{t-1}\equiv \mathsf{P}_{t-1}\cdot \bm{1}_{t\geq 2}+\mathrm{id}\cdot \bm{1}_{t=1}$. The proximal gradient descent is identified as $\mu^{(t)}=\mathsf{R}_{t}(v^{(t)})$ for $t\geq 0$.  Consequently, for $t\geq 0$,
\begin{align*}
\mathsf{F}_t^{\langle 1 \rangle}(v^{([-1:t-1])}) &= \mathsf{R}_{t-1}(v^{(t-1)})\bm{1}_{t\geq 1}+ \mu_\ast\bm{1}_{t=0},\\
\mathsf{F}_t^{\langle 2 \rangle}(v^{([-1:t-1])}) & = \mathsf{R}_{t-1}(v^{(t-1)})\bm{1}_{t\geq 1}+\mu^{(0)}\bm{1}_{t=0},\\
\mathsf{G}_{t}^{\langle 1\rangle}(u^{([-1:t-1])}) &= 0_{m},\\
\mathsf{G}_{t}^{\langle 2\rangle}(u^{([-1:t])}) &= -\eta_{t-1}\cdot \partial_1 \mathsf{L}_{t-1}\big(u^{(t)},\mathcal{F}(u^{(0)},\xi)\big)\bm{1}_{t\geq 1}.
\end{align*}
Using Definition \ref{def:GFOM_se_asym}, we have the following state evolution. We initialize with $\Phi_{-1}=\mathrm{id}(\R^m),\Xi_{-1}\equiv \mathrm{id}(\R^n)$, $\mathfrak{U}^{(-1)}=0_m$ and $\mathfrak{V}^{(-1)}=\mu_\ast$. For $t=0$:
\begin{itemize}
	\item $\Phi_0: \R^{m\times [-1:0]}\to \R^{m\times [-1:0]}$ is defined as $\Phi_{0}(\mathfrak{u}^{([-1:0])})\equiv [\mathfrak{u}^{(-1)}\,|\, \mathfrak{u}^{(0)}]$.
	\item The Gaussian law of $\mathfrak{U}^{(0)}\in \R$ is determined via $\var(\mathfrak{U}^{(0)})=\pnorm{\mu_\ast}{}^2/n$.
	\item $\Xi_0: \R^{n\times [-1:0]}\to \R^{n\times [-1:0]}$ is defined as $\Xi_{0}(\mathfrak{v}^{([-1:0])})\equiv [\mathfrak{v}^{(-1)}\,|\, \mathfrak{v}^{(0)}+\mu^{(0)}] $.
	\item $\mathfrak{V}^{(0)}\in \R$ is degenerate (identically 0).
\end{itemize}
For $t\geq 1$, we have the following state evolution:
\begin{enumerate}
	\item[(O1)] Let $\Phi_{t}:\R^{m\times [-1:t]}\to \R^{m\times [-1:t]}$ be defined as follows: for $w \in [-1:t-1]$, $\big[\Phi_{t}(\mathfrak{u}^{([-1:t])})\big]_{\cdot,w}\equiv \big[\Phi_{w}(\mathfrak{u}^{([-1:w])})\big]_{\cdot,w}$, and for $w=t$,
	\begin{align*}
	\big[\Phi_{t}(\mathfrak{u}^{([-1:t])})\big]_{\cdot,t} \equiv \mathfrak{u}^{(t)}-\sum_{s \in [1:t-1]}  \eta_{s-1}\cdot  \mathfrak{f}_{s}^{(t-1)}\cdot  \partial_1\mathsf{L}_{s-1}\Big(\big[\Phi_{s}(\mathfrak{u}^{([-1:t])})\big]_{\cdot,s}, \mathcal{F}\big(\mathfrak{u}^{(0)},\xi\big) \Big).
	\end{align*}
	Here the correction coefficients $\{\mathfrak{f}_{s}^{(t-1) } \}_{s \in [1:t-1]}\subset \R$ (defined for $t\geq 2$) are determined by
	\begin{align*}
	\mathfrak{f}_{s}^{(t-1) } \equiv   \E^{(0)} \partial_{\mathfrak{V}^{(s)}}  \mathsf{P}_{t-1;\pi_n}\Big([\Xi_{t-1;\pi_n} (\mathfrak{V}^{(-1)}_{\pi_n},\mathfrak{V}^{([0:t-1])})]_{\cdot,t-1}\Big).
	\end{align*}
	\item[(O2)] Let the Gaussian law of $\mathfrak{U}^{(t)}$ be determined via the following correlation specification: for $s \in [0:t]$,
	\begin{align*}
	\mathrm{Cov}\big(\mathfrak{U}^{(t)}, \mathfrak{U}^{(s)} \big)\equiv   \E^{(0)} \prod_{\ast \in \{s,t\}}  \mathsf{F}_{\ast;\pi_n}^{\langle 1\rangle}\Big( \Xi_{\ast-1;\pi_n} (\mathfrak{V}^{(-1)}_{\pi_n},\mathfrak{V}^{([0:\ast-1])})\Big).
	\end{align*}
	\item[(O3)] Let $\Xi_{t}:\mathbb{R}^{n\times [-1:t]}\to \mathbb{R}^{n\times [-1:t]}$ be defined as follows: for $w \in [-1:t-1]$, $\big[\Xi_{t}(\mathfrak{v}^{([-1:t])})\big]_{\cdot,w}\equiv \big[\Xi_{w}(\mathfrak{v}^{([-1:w])})\big]_{\cdot,w}$, and for $w=t$,
	\begin{align*}
	\big[\Xi_{t}(\mathfrak{v}^{([-1:t])})\big]_{\cdot,t} &\equiv \mathfrak{v}^{(t)}+\sum_{s \in [1:t]} \big(\mathfrak{g}_{s}^{(t)}+\bm{1}_{s=t}\big)\cdot \mathsf{R}_{s-1}\big([\Xi_{s-1}(\mathfrak{v}^{([-1:s-1])})]_{\cdot,s-1} \big) + \mathfrak{g}_{0}^{(t)}\cdot \mu_\ast.
	\end{align*}
	Here the coefficients $\{\mathfrak{g}_{s}^{(t)}\}_{s \in [0:t]}\subset \R$ are determined via
	\begin{align*}
	\mathfrak{g}_{s}^{(t)}
	& \equiv -\phi \eta_{t-1}\cdot \E^{(0)} \partial_{\mathfrak{U}^{(s)}} \partial_1 \mathsf{L}_{t-1;\pi_m}\Big([\Phi_{t;\pi_m}  (\mathfrak{U}^{(-1)}_{\pi_m},\mathfrak{U}^{([0:t])})]_{\cdot,t},\mathcal{F}\big(\mathfrak{U}^{(0)},\xi_{\pi_m}\big)\Big).
	\end{align*}
	\item[(O4)] Let the Gaussian law of $\mathfrak{V}^{(t)}$ be determined via the following correlation specification: for $s \in [1:t]$,
	\begin{align*}
	\mathrm{Cov}(\mathfrak{V}^{(t)},\mathfrak{V}^{(s)})\equiv  \phi\cdot  \E^{(0)} \prod_{\ast \in \{s,t\}} \eta_{\ast-1} \partial_1 \mathsf{L}_{\ast-1;\pi_m}\Big([\Phi_{\ast;\pi_m}  (\mathfrak{U}^{(-1)}_{\pi_m},\mathfrak{U}^{([0:\ast])})]_{\cdot,\ast},\mathcal{F}\big(\mathfrak{U}^{(0)},\xi_{\pi_m}\big)\Big).
	\end{align*}
\end{enumerate}

\noindent (\textbf{Step 2}). We now make a few identifications to convert (O1)-(O4) to the state evolution in (S1)-(S3). 

First, we identify $\mathfrak{U}^{([0:t])}$ as $\mathfrak{Z}^{([0:t])}$ and $\mathfrak{V}^{([1:t])}$ as $\mathfrak{W}^{([1:t])}$. Variable $\mathfrak{U}^{(-1)}$ can be dropped for free, and variables $\mathfrak{V}^{([-1:0])}$ are contained in the recursively defined mappings as detailed below. With a formal variable $\mathsf{R}_0(\Delta_0)\equiv \mu^{(0)}\in \R^n$, for $t\geq 1$, let $\Delta_t: \R^{n\times [1:t]}\to \R^n$ be defined recursively via the following relation:
\begin{align*}
\Delta_t\big(\mathfrak{w}^{([1:t])}\big)\equiv \mathfrak{w}^{(t)}+\sum_{s \in [1:t]} \big(\mathfrak{g}_s^{(t)}+\bm{1}_{s=t}\big)\cdot  \mathsf{R}_{s-1}\Big(\Delta_{s-1}\big(\mathfrak{w}^{([1:s-1])}\big)\Big)+ \mathfrak{g}_0^{(t)}\cdot \mu_\ast.
\end{align*}
Moreover, let $\Theta_t, \Upsilon_t: \mathbb{R}^{m\times [0:t]}\to \R^m$ be defined recursively: for $t\geq 1$, 
\begin{itemize}
	\item $\Theta_t(\mathfrak{z}^{([0:t])})\equiv \mathfrak{z}^{(t)}- \sum_{s \in [1:t-1]}\eta_{s-1} \cdot \mathfrak{f}_{s}^{(t-1) }\cdot  \partial_1 \mathsf{L}_{s-1}\big(\Theta_{s}(\mathfrak{z}^{([0:s])}),\mathcal{F}(\mathfrak{z}^{(0)},\xi)\big)$,
	\item $\Upsilon_t(\mathfrak{z}^{([0:t])}) = - \eta_{t-1}  \partial_1 \mathsf{L}_{t-1}\big(\Theta_t(\mathfrak{z}^{([0:t])}),\mathcal{F}(\mathfrak{z}^{(0)},\xi)\big)$.
\end{itemize}
	We may translate the recursive definition for the Gaussian laws of $\mathfrak{U}^{([0:t])}$ and $\mathfrak{V}^{([1:t])}$ to those of $\mathfrak{Z}^{([0:t])}\in \R^{[0:t]}$ and $\mathfrak{W}^{([1:t])}\in \R^{[1:t]}$ as follows: initialized with $\cov(\mathfrak{Z}^{(0)},\mathfrak{Z}^{(0)})=\pnorm{\mu_\ast}{}^2/n$, with formal variables $\mathsf{P}_{-1}(\Delta_{-1})=\mu_\ast$ and $\mathsf{P}_{0}(\Delta_{0})=\mu^{(0)}$, for $0\leq s\leq t$, 
	\begin{align*}
	\cov(\mathfrak{Z}^{(t)},\mathfrak{Z}^{(s)})&=
	\E^{(0)} \prod\limits_{\ast \in \{s-1,t-1\}} \mathsf{P}_{*;\pi_n}\big([\Xi_{*;\pi_n}(\mathfrak{V}_{\pi_n}^{(-1)},\mathfrak{V}^{([0:*])})]_{\cdot,*}\big)\\
	& = \E^{(0)} \prod\limits_{\ast \in \{s-1,t-1\}} \mathsf{P}_{*;\pi_n} \big(\Delta_{*;\pi_n} (\mathfrak{W}^{([1:*])})\big),
	\end{align*}
	and for $1\leq s\leq t$,
	\begin{align*}
	\cov(\mathfrak{W}^{(t)},\mathfrak{W}^{(s)})&=\phi\cdot  \E^{(0)} \prod_{\ast \in \{s,t\}} \eta_{\ast-1} \partial_1 \mathsf{L}_{\ast-1;\pi_m}\Big([\Phi_{\ast;\pi_m}  (\mathfrak{U}^{(-1)}_{\pi_m},\mathfrak{U}^{([0:\ast])})]_{\cdot,\ast},\mathcal{F}\big(\mathfrak{U}^{(0)},\xi_{\pi_m}\big)\Big)\\
	& = \phi\cdot \E^{(0)} \prod_{\ast \in \{s,t\}} \Upsilon_{*;\pi_m}(\mathfrak{Z}^{([0:*])}).
	\end{align*}
	Then we have
	\begin{align}\label{ineq:gd_se_chg_not}
	\big(\Upsilon_{t;k}(\mathfrak{Z}^{([0:t])}),\Theta_{t;k}(\mathfrak{Z}^{([0:t])})\big)_{k \in [m]}&\equald \big( (\mathsf{G}_{t}^{\langle 2\rangle}\circ\Phi_{t})_k(\mathfrak{U}^{([-1:t])}),
	\big[\Phi_{t;k}(\mathfrak{U}^{([-1:t])})\big]_{\cdot,t}\big)_{k \in [m]},\nonumber\\
	\big(\Delta_{t;\ell}(\mathfrak{W}^{([1:t])}) \big)_{\ell \in [n]}&\equald \big([\Xi_{t;\ell} (\mathfrak{V}_{\ell}^{(-1)},\mathfrak{V}^{([0:t])})]_{\cdot,t} \big)_{\ell \in [n]}.
	\end{align}
	Furthermore, for $1\leq s\leq t$, with $\Omega_t\equiv \mathsf{P}_t\circ \Delta_t$,
	\begin{align*}
	\mathfrak{f}_{s}^{(t) }& = \E^{(0)} \partial_{ \mathfrak{V}^{(s)}} \mathsf{P}_{t;\pi_n}\big(\big[\Xi_{t;\pi_n}(\mathfrak{V}_{\pi_n}^{(-1)},\mathfrak{V}^{([0:t])})\big]_{\cdot,t}\big)\\
	& = \E^{(0)} \partial_{ \mathfrak{W}^{(s)}} \mathsf{P}_{t;\pi_n}\big(\Delta_{t;\pi_n}(\mathfrak{W}^{([1:t])})\big) =\rho_{t,s},\\
	\mathfrak{g}_s^{(t)}& = \phi\cdot \E^{(0)} \partial_{\mathfrak{U}^{(s)}} (\mathsf{G}_{t}^{\langle 2\rangle}\circ\Phi_{t})_{\pi_m}(\mathfrak{U}^{([-1:t])})\\
	&= \phi\cdot \E^{(0)}\partial_{\mathfrak{Z}^{(s)}} \Upsilon_{t;\pi_m}(\mathfrak{Z}^{([0:t])}) = \tau_{t,s},\\
	\mathfrak{g}_0^{(t)}& =\phi\cdot \E^{(0)} \partial_{\mathfrak{U}^{(0)}} (\mathsf{G}_{t}^{\langle 2\rangle}\circ\Phi_{t})_{\pi_m}(\mathfrak{U}^{([-1:t])})\\
	&= \phi\cdot \E^{(0)}\partial_{\mathfrak{Z}^{(0)}} \Upsilon_{t;\pi_m}(\mathfrak{Z}^{([0:t])}) = \delta_t.
	\end{align*}
	This concludes the desired state evolution in (S1)-(S3), by identifying the formal variables $\Omega_* = \mathsf{P}_*(\Delta_*)$ for $*=-1,0$.
	
	\noindent (\textbf{Step 3}). Next we prove the claim for $\E^{(0)} \Psi\big((A \mu^{(t-1)})_k, {Z}_k^{(t-1)}, (A \mu_\ast)_k\big)$. Note that for $k \in [m]$, $(A\mu^{(t-1)})_k=u_{k}^{(t)}$, and
	\begin{align*}
	Z_k^{(t-1)} & = \iprod{A_k}{\mu^{(t-1)}}+\sum_{s \in [1:t-1]} \eta_{s-1} \rho_{t-1,s}\cdot  \partial_1\mathsf{L}_{s-1;k}\big(\iprod{A_k}{\mu^{(s-1)}},Y_k\big)\\
	& =u_{k}^{(t)}+ \sum_{s \in [1:t-1]} \eta_{s-1}\rho_{t-1,s}\cdot \partial_1\mathsf{L}_{s-1;k}\big(u_{k}^{(s)},\mathcal{F}(u_{k}^{(0)},\xi_k)\big)\\
	& \equiv F_Z\big(u_k^{([-1:t])}\big).
	\end{align*}
	By (O1), $F_Z\big(\Phi_{t;k}(\mathfrak{U}_k^{(-1)},\mathfrak{U}^{([0:t])})\big)=\mathfrak{U}^{(t)}\equald \mathfrak{Z}^{(t)}$. Consequently, with 
	\begin{align}\label{ineq:gd_se_0}
	\Psi_Z\big(u_k^{([-1:t])}\big)\equiv \Psi\big(u_{k}^{(t)},F_Z(u_k^{([-1:t])}),u_{k}^{(0)}\big),
	\end{align}
	by Theorem \ref{thm:GFOM_se_asym}, modulo verification of the condition (\ref{cond:Psi_asym}) for $\Psi_Z$, we have 
	\begin{align}\label{ineq:gd_se_1}
	\Psi\big((A \mu^{(t-1)})_k, {Z}_k^{(t-1)}, (A \mu_\ast)_k\big)& =\Psi\big(u_{k}^{(t)},F_Z(u_k^{([-1:t])}),u_{k}^{(0)}\big) \nonumber \\
	&= \Psi_Z\big(u_k^{([-1:t])}\big)\stackrel{d}{\approx}  \Psi_Z\big(\Phi_{t;k}(\mathfrak{U}_k^{(-1)},\mathfrak{U}^{([0:t])})\big)\nonumber\\
	&\equald \Psi\big(\Theta_{t;k}(\mathfrak{Z}^{([0:t])}), \mathfrak{Z}^{(t)}, \mathfrak{Z}^{(0)}\big).
	\end{align}
	To verify the condition (\ref{cond:Psi_asym}) for $\Psi_Z$ and quantify the error in the above display, it suffices to invoke Lemma \ref{lem:rho_tau_bound} for a bound on $\{\rho_{t,s}\}_{s \in [1:t]}$: with $\Psi_Z$ defined in (\ref{ineq:gd_se_0}) and the estimate in Lemma \ref{lem:rho_tau_bound}, some simple algebra shows that condition (\ref{cond:Psi_asym}) is satisfied for $\Psi_Z$ with $\Lambda_{\Psi_Z}\equiv \Lambda_{\Psi} (K\Lambda)^{c_t}$. The claim for $\E^{(0)} \Psi\big((A \mu^{(t-1)})_k, {Z}_k^{(t-1)}, (A \mu_\ast)_k\big)$ now follows from (\ref{ineq:gd_se_1}).
	
	\noindent (\textbf{Step 4}). Finally we prove the claim for $ \E^{(0)} \Psi\big(\mu^{(t)}_\ell, {W}_\ell^{(t)},\mu_{\ast,\ell}\big)$. Recall for $\ell \in [n]$, $\mu^{(t)}_\ell=\mathsf{R}_{t}(v_{\ell}^{(t)})$, and moreover by (O3),
	\begin{align*}
	W_\ell^{(t)}&=-\eta_{t-1}\cdot  \bigiprod{Ae_\ell}{ \partial_1 \mathsf{L}_{t-1} (A \mu^{(t-1)},Y)}-\sum_{s \in [1:t]} \tau_{t,s}\cdot \mu^{(s-1)}_\ell - \delta_t\cdot \mu_{\ast,\ell}\\
	& =v^{(t)}_{\ell}-\mathsf{R}_{t-1;\ell}\big(v^{(t-1)}_{\ell}\big)- \sum_{s \in [1:t]} \tau_{t,s}\cdot \mathsf{R}_{s-1;\ell}\big(v^{(s-1)}_{\ell}\big)-\delta_t\cdot v^{(-1)}_{\ell}\\
	&\equiv F_W\big(v_\ell^{([-1:t])}\big).
	\end{align*}
	By (O3), $F_W\big(\Xi_{t;\ell}(\mathfrak{V}_\ell^{(-1)},\mathfrak{V}^{([0:t])})\big)=\mathfrak{V}^{(t)}\equald \mathfrak{W}^{(t)}$.  So similar to (\ref{ineq:gd_se_0}), with
	\begin{align}\label{ineq:gd_se_W_1}
	\Psi_W\big(v_\ell^{([-1:t])}\big)\equiv \Psi\big(\mathsf{R}_{t;\ell}(v_{\ell}^{(t)}),F_W(v_\ell^{([-1:t])}),v_{\ell}^{(-1)}\big),
	\end{align}
	by Theorem \ref{thm:GFOM_se_asym}, modulo verification of the condition (\ref{cond:Psi_asym}) for $\Psi_W$, we have 
	\begin{align}\label{ineq:gd_se_W_2}
	&\Psi\big(\mu^{(t)}_\ell, {W}_\ell^{(t)},\mu_{\ast,\ell}\big) =\Psi\big(\mathsf{R}_{t;\ell}(v_{\ell}^{(t)}),F_W(v_\ell^{([-1:t])}),v_{\ell}^{(-1)}\big) =\Psi_W\big(v_\ell^{([-1:t])}\big)\nonumber\\
	&\stackrel{d}{\approx}  \Psi_W\big(\Xi_{t;\ell}(\mathfrak{V}_\ell^{(-1)},\mathfrak{V}^{([0:t])})\big)\equald \Psi\big(\mathsf{R}_{t;\ell}(\Delta_{t;\ell}(\mathfrak{W}^{[1:t]})),\mathfrak{W}^{(t)},\mu_{\ast,\ell}\big).
	\end{align}
	Here, $\stackrel{d}{\approx}$ is used in the sense of the statement of Theorem~\ref{thm:GFOM_se_asym}, with the associated error bounds stated therein. From here, in view of Lemma \ref{lem:rho_tau_bound}, the condition (\ref{cond:Psi_asym}) is satisfied for $\Psi_W$ with the choice $\Lambda_{\Psi_W}\equiv \Lambda_{\Psi}(1+\delta_t) ( K\Lambda)^{c_t}$. Using the assumption (A5), we have $\delta_t\leq \Lambda$. The claim for the entrywise distributional characterization follows from (\ref{ineq:gd_se_W_2}). 
	
	For the averaged distributional characterization, it suffices to note that the estimates in Step 3 of the proof of Theorem \ref{thm:gd_se} and $\delta_t\leq \Lambda$ remain valid under (A4')-(A5'), so we may apply Theorem \ref{thm:GFOM_se_asym_avg} to conclude. 
\end{proof}

\subsection{Proof of Theorem \ref{thm:db_gd_oracle}}

\begin{lemma}\label{lem:tau_diag}
	The following formula holds: for any $s \in [1:t]$,
	\begin{align*}
	\tau_{s,s} = -\eta_{s-1} \cdot \E^{(0)} \partial_{11} \mathsf{L}_{s-1;\pi_m}\big(\Theta_{s;\pi_m}(\mathfrak{Z}^{([0:s])}),\mathcal{F}(\mathfrak{Z}^{(0)},\xi_{\pi_m})\big).
	\end{align*}
\end{lemma}
\begin{proof}
	Using the inversion formula (\ref{ineq:rho_tau_bound_Upsilon}) and the notation therein, for any $s \in [t]$,
	\begin{align*}
	\partial_{(s)} \Upsilon_{s;k}(z^{([0:s])})&= e_s^\top \big[I_s- \bm{L}^{[s]}_{k}(z^{([0:s])}) \mathfrak{O}_{s}(\bm{\rho}^{[s-1]})\big]^{-1} \bm{L}^{[s]}_{k}(z^{([0:s])}) e_s\\
	&=  -\eta_{s-1}\partial_{11} \mathsf{L}_{s-1;k}\big(\Theta_{s;k}(z^{([0:s])}),\mathcal{F}(z^{(0)},\xi_k)\big),
	\end{align*}
	where in the second identity we used the fact that $e_s^\top \big[I_s- \bm{L}^{[s]}_{k}(z^{([0:s])}) \mathfrak{O}_{s}(\bm{\rho}^{[s-1]})\big]^{-1}e_s=1$, as the matrix in the middle is lower triangular with $0$ diagonal elements. The claim follows by the definition of $\tau_{s,s}$. 
\end{proof}

\begin{proof}[Proof of Theorem \ref{thm:db_gd_oracle}]
	By definition of $\mu^{(t)}_{\textrm{db}}$ in (\ref{def:debias_gd_oracle}), it suffices to control 
	\begin{align*}
	\E^{(0)}\psi\big(b^{(t)}_{\mathrm{db}}\cdot \mu_{\ast,\ell}- e_\ell^\top \bm{W}^{[t]} \bm{\omega}^{[t],\top} e_t\big).
	\end{align*}
	As $e_\ell^\top \bm{W}^{[t]}\in \R^{1\times t}$ only involves the elements in the $\ell$-th row of $\bm{W}^{[t]}$, by letting
	\begin{align*}
	\psi_0(w_{[1:t]})\equiv \psi\big(b^{(t)}_{\mathrm{db}}\cdot \mu_{\ast,\ell}- w_{[1:t]}^\top \bm{\omega}^{[t],\top} e_t\big),\quad \forall w_{[1:t]}\in \R^{[1:t]},
	\end{align*}
	we have $\psi_0(\{W_\ell^{(s)}\}_{s \in [1:t]})\equiv \psi\big(b^{(t)}_{\mathrm{db}}\cdot \mu_{\ast,\ell}- e_\ell^\top \bm{W}^{[t]} \bm{\omega}^{[t],\top} e_t\big)$. Note that for any multi-index $\alpha$ with $\abs{\alpha}\leq 3$, for some constant $c_t=c_t(t,\mathfrak{p})>1$,
	\begin{align*}
	\bigabs{\partial_\alpha \psi_0(w_{[1:t]})}&\leq \Lambda_\psi \Big(1+\bigabs{b^{(t)}_{\mathrm{db}}\cdot \mu_{\ast,\ell}- w_{[1:t]}^\top \bm{\omega}^{[t],\top} e_t}\Big)^{\mathfrak{p}}\cdot \pnorm{\bm{\omega}_{t\cdot}^{[t]}}{}^3\\
	&\leq \Lambda_\psi \big(\Lambda L_\mu\cdot (1+\pnorm{\bm{\omega}^{[t]}_{t\cdot} }{})\big)^{c_t}\cdot \big(1+\pnorm{w_{[1:t]}}{}\big)^{\mathfrak{p}}.
	\end{align*}
	Here the second line follows as $\abs{b^{(t)}_{\mathrm{db}}}\leq \pnorm{\bm{\omega}^{[t]}_{t\cdot} }{}\cdot \pnorm{\bm{\delta}^{[t]}}{} \leq  \Lambda^{c_t} \pnorm{\bm{\omega}^{[t]}_{t\cdot} }{}$. By using Lemma \ref{lem:lower_tri_mat_inv} coupled with Lemmas \ref{lem:rho_tau_bound} and \ref{lem:tau_diag}, we have
	\begin{align*}
	\pnorm{\bm{\omega}^{[t]}_{t\cdot} }{}\leq \bigg(\frac{t\cdot \pnorm{\bm{\tau}^{[t]}}{\op} }{\min_{s \in [t]} \abs{\bm{\tau}^{[t]} _{ss}}}\bigg)^t \leq \big(K\Lambda\cdot  \tau_\ast^{(t),-1}\big)^{c_t}.
	\end{align*}
	Now we may apply Theorem \ref{thm:gd_se} to conclude the general bounds.
\end{proof}

\subsection{Proof of Theorem \ref{thm:gen_error_oracle}}

\begin{lemma}\label{lem:gen_error_replace_pop}
	Suppose $\{\mathsf{H}_{\mathcal{F};k}\}\subset C^3(\R^2)$ have mixed derivatives of order 3 all bounded by $\Lambda$. Then for any $q>1$, there exists some $c_t=c_t(t,q)>1$ such that
	\begin{align*}
	\E^{(0)}\bigabs{\mathscr{E}_{\mathsf{H}}^{(t)}(A,Y)-\E\mathsf{H}_{\mathcal{F}}\big(\mathfrak{Z}^{(t+1)},\mathfrak{Z}^{(0)}\big)}^q\leq (K\Lambda L_\mu)^{c_t}\cdot n^{-1/c_t}.
	\end{align*}
\end{lemma}
\begin{proof}
	With $\mathsf{Z}_n\sim \mathcal{N}(0,I_n/n)$, let
	\begin{align}\label{ineq:gen_error_replace_pop_1}
	\mathscr{E}_{\mathsf{H};\mathsf{Z}_n}^{(t)}(A,Y)&\equiv \E\big[\mathsf{H}\big(\iprod{\mathsf{Z}_n}{\mu^{(t)}},\mathcal{F}(\iprod{\mathsf{Z}_n}{\mu_\ast},\xi_{\pi_m})\big)|(A,Y)\big]\nonumber\\
	&= \E \big[\mathsf{H}_{\mathcal{F}}\big(\iprod{\mathsf{Z}_n}{\mu^{(t)}}, \iprod{\mathsf{Z}_n}{\mu_\ast}\big)|(A,Y)\big].
	\end{align}
	
	\noindent (\textbf{Step 1}). We shall prove in this step that for some universal constant $c_0>0$,
	\begin{align}\label{ineq:gen_error_replace_pop_step1}
	\bigabs{\mathscr{E}_{\mathsf{H};\mathsf{Z}_n}^{(t)}(A,Y)-\mathscr{E}_{\mathsf{H}}^{(t)}(A,Y)}\leq\frac{c_0\Lambda}{\sqrt{n}}\cdot \big(\pnorm{\mu^{(t)}}{\infty}^3+\pnorm{\mu_\ast}{\infty}^3\big).
	\end{align}
	Let the function $G:\R^n\to \R$ be defined by 
	\begin{align*}
	G(z)&\equiv \E_{\pi_m}\big[\mathsf{H}\big(\iprod{z}{\mu^{(t)}},\mathcal{F}(\iprod{z}{\mu_\ast},\xi_{\pi_m})\big)|(A,Y)\big]=\mathsf{H}_{\mathcal{F}}\big(\iprod{z}{\mu^{(t)}},\iprod{z}{\mu_\ast}\big).
	\end{align*}
	It is easy to compute that $\pnorm{\partial_i^3 G}{\infty}\leq c_0\Lambda\cdot (\abs{\mu^{(t)}_i}^3+\abs{\mu_{\ast,i}}^3)$, so by Lindeberg's universality principle (cf. Lemma \ref{lem:lindeberg}), we have 
	\begin{align*}
	\bigabs{\mathscr{E}_{\mathsf{H};\mathsf{Z}_n}^{(t)}(A,Y)-\mathscr{E}_{\mathsf{H}}^{(t)}(A,Y)}& = \abs{\E G(\mathsf{Z}_n)-\E G(A_{\mathrm{new}})}\leq \frac{c_0\Lambda}{\sqrt{n}}\cdot \big(\pnorm{\mu^{(t)}}{\infty}^3+\pnorm{\mu_\ast}{\infty}^3\big),
	\end{align*}
	proving the claim (\ref{ineq:gen_error_replace_pop_step1}).
	
	\noindent (\textbf{Step 2}). In this step we prove that for any $q>1$, there exists some $c_t=c_t(t,q)>1$ such that
	\begin{align}\label{ineq:gen_error_replace_pop_step2}
	\E^{(0)}\bigabs{\mathscr{E}_{\mathsf{H};\mathsf{Z}_n}^{(t)}(A,Y)-\E\mathsf{H}_{\mathcal{F}}\big(\mathfrak{Z}^{(t+1)},\mathfrak{Z}^{(0)}\big)}^q\leq (K\Lambda L_\mu)^{c_t}\cdot n^{-1/c_t}.
	\end{align}
	To this end, let 
	\begin{align*}
	\Sigma^{(t+1)}\equiv \frac{1}{n}
	\begin{pmatrix}
	\pnorm{\mu^{(t)}}{}^2 & \iprod{ \mu^{(t)}}{\mu_\ast}\\
	\iprod{ \mu^{(t)}}{\mu_\ast} & \pnorm{\mu_\ast}{}^2
	\end{pmatrix}, \, {\Sigma}_0^{(t+1)}\equiv 
	\begin{pmatrix}
	\var(\mathfrak{Z}^{(t+1)}) & \cov(\mathfrak{Z}^{(t+1)},\mathfrak{Z}^{(0)})\\
	\cov(\mathfrak{Z}^{(t+1)},\mathfrak{Z}^{(0)}) & \var(\mathfrak{Z}^{(0)}) 
	\end{pmatrix}.
	\end{align*}
	Then we have
	\begin{align*}
	&\bigabs{\mathscr{E}_{\mathsf{H};\mathsf{Z}_n}^{(t)}(A,Y)-\E\mathsf{H}_{\mathcal{F}}\big(\mathfrak{Z}^{(t+1)},\mathfrak{Z}^{(0)}\big)}\\
	&= \bigabs{ \E \big[\mathsf{H}_{\mathcal{F}}\big(\Sigma^{(t+1),1/2}\mathcal{N}(0,I_2)\big)|(A,Y)\big] - \E \mathsf{H}_{\mathcal{F}}\big({\Sigma}_0^{(t+1),1/2}\mathcal{N}(0,I_2)\big) }\\
	&\leq c_0\Lambda\cdot \pnorm{\Sigma^{(t+1),1/2}-\Sigma_0^{(t+1),1/2}}{\op} \leq c_0\Lambda\cdot \pnorm{\Sigma^{(t+1)}-\Sigma_0^{(t+1)}}{\op}^{1/2}.
	\end{align*}
	Now (\ref{ineq:gen_error_replace_pop_step2}) follows by Theorem \ref{thm:gd_se}-(2) applied to the right hand side of the above display, upon noting the definition of covariance for $\mathfrak{Z}^{(\cdot)}$ in (S2).
	
	The claim now follows combining (\ref{ineq:gen_error_replace_pop_step1}) in Step 1, (\ref{ineq:gen_error_replace_pop_step2}) in Step 2, and the delocalization estimate for $\mu^{(t)}$ obtained in \cite[Proposition 6.2]{han2025entrywise}.
\end{proof}

\begin{proof}[Proof of Theorem \ref{thm:gen_error_oracle}]
As $\overline{\mathscr{E}}_{\mathsf{H}}^{(t)}=m^{-1}\sum_{k \in [m]} \mathsf{H}_{\mathcal{F};k}\big(Z^{(t)}_k, (A\mu_\ast)_k\big)$, an application of Theorem \ref{thm:gd_se}-(2) yields that for some $c_t=c_t(t,q)>1$,
\begin{align*}
\E^{(0)}\abs{\overline{\mathscr{E}}_{\mathsf{H}}^{(t)} - \E\mathsf{H}_{\mathcal{F}}\big(\mathfrak{Z}^{(t+1)},\mathfrak{Z}^{(0)}\big)}^q\leq (K\Lambda L_\mu)^{c_t}\cdot n^{-1/c_t}.
\end{align*}
The claim now follows from Lemma \ref{lem:gen_error_replace_pop}.
\end{proof}

\section{Proofs for Section \ref{section:iterative_inf}}

\subsection{Proof of Theorem \ref{thm:consist_tau_rho}}

We first prove a preliminary estimate.

\begin{lemma}\label{lem:hat_rho_tau_bound}
	Suppose (A1), (A3), (A4') and (A5') hold. Then there exists some constant $c_t=c_t(t)>1$ such that
	\begin{align*}
	\pnorm{\hat{\bm{\tau}}^{[t]}}{\op}\vee \pnorm{\hat{\bm{\rho}}^{[t]}}{\op}\leq (K\Lambda)^{c_t}.
	\end{align*}
\end{lemma}
\begin{proof}
	By Algorithm \ref{def:alg_tau_rho}, as $\hat{\bm{L}}^{[t]}_k \mathfrak{O}_t(\hat{\bm{\rho}}^{[t-1]})$ is a lower triangular matrix with diagonal elements $0$, we have $\big(\hat{\bm{L}}^{[t]}_k \mathfrak{O}_t(\hat{\bm{\rho}}^{[t-1]})\big)^t=0_{t\times t}$, and therefore for $k \in [m]$,
	\begin{align*}
	1+\pnorm{\hat{\bm{\tau}}^{[t]}_k}{\op} &\leq 1+\phi\cdot  \bigpnorm{\big[I_t - \hat{\bm{L}}^{[t]}_k \mathfrak{O}_t(\hat{\bm{\rho}}^{[t-1]})\big]^{-1}}{\op} \pnorm{\hat{\bm{L}}^{[t]}_k}{\op}\\
	&\leq 1+\phi\cdot \sum_{r=0}^{t-1} \bigpnorm{\hat{\bm{L}}^{[t]}_k \mathfrak{O}_t(\hat{\bm{\rho}}^{[t-1]})}{\op}^r\cdot \pnorm{\hat{\bm{L}}^{[t]}_k}{\op}\leq  (K\Lambda)^{c_0 t}\cdot \big(1+\pnorm{\hat{\bm{\rho}}^{[t-1]}}{\op}\big)^t.
	\end{align*}
	Taking average over $k \in [m]$, we have
	\begin{align}\label{ineq:hat_rho_tau_bound_1}
	1+\pnorm{\hat{\bm{\tau}}^{[t]}}{\op} &\leq (K\Lambda)^{c_0 t}\cdot \big(1+\pnorm{\hat{\bm{\rho}}^{[t-1]}}{\op}\big)^t.
	\end{align}
	Here $c_0>0$ is a universal constant whose numeric value may change from line to line. On the other hand, using Algorithm \ref{def:alg_tau_rho} and the above display
	\begin{align*}
	1+\pnorm{\hat{\bm{\rho}}_\ell^{[t]}}{\op}&\leq 1+ \pnorm{\hat{\bm{P}}_\ell^{[t]}}{\op}\cdot \Big[1+\big(\pnorm{\hat{\bm{\tau}}^{[t]}}{\op}+1\big) \cdot \pnorm{\hat{\bm{\rho}}_\ell^{[t-1]}}{\op}\Big]\\
	&\leq (K\Lambda)^{c_0 t}\cdot \big(1+\pnorm{\hat{\bm{\rho}}^{[t-1]}}{\op}\big)^{c_0 t}.
	\end{align*}
	Taking average and iterating the above bound, we obtain $
	\pnorm{\hat{\bm{\rho}}^{[t]}}{\op}\leq (K\Lambda)^{ c_t}$. 
	The claim for $\pnorm{\hat{\bm{\tau}}^{[t]}}{\op}$ follows from the above display in combination with (\ref{ineq:hat_rho_tau_bound_1}).
\end{proof}

For notational convenience, let 
\begin{align*}
\epsilon_{\rho;t}\equiv \pnorm{\hat{\bm{\rho}}^{[t]}-\bm{\rho}^{[t]} }{\op},\quad \epsilon_{\tau;t}\equiv \pnorm{\hat{\bm{\tau}}^{[t]}-\bm{\tau}^{[t]} }{\op}.
\end{align*}

\begin{lemma}\label{lem:epi_tau}
	Under the same assumptions as in Theorem \ref{thm:consist_tau_rho}, for any $q>1$, there exists some constant $c_t=c_t(t,q)>0$ such that 
	\begin{align*}
	\E^{(0)} \epsilon_{\tau;t}^q&\leq \big(K\Lambda\big)^{c_t}\cdot \E^{(0)} \epsilon_{\rho;t-1}^q +  (K\Lambda L_\mu)^{ c_t} \cdot n^{-1/c_t}.
	\end{align*}
\end{lemma}

\begin{proof}
	Consider the auxiliary sequence defined by 
	\begin{align}\label{ineq:cons_tau_rho_1}
	\overline{\bm{\tau}}^{[t]}_k \equiv \phi\cdot \big[I_t - \hat{\bm{L}}^{[t]}_k  \mathfrak{O}_t(\bm{\rho}^{[t-1]})\big]^{-1} \hat{\bm{L}}^{[t]}_k  \in \R^{t\times t},\quad k \in [m].
	\end{align}
	Here recall $\hat{\bm{L}}^{[t]}_k =\mathrm{diag}\big(\{-\eta_{s-1} \iprod{e_k}{\partial_{11}\mathsf{L}_{s-1}(A\mu^{(s-1)},Y)}\}_{s \in [1:t]}\big)\in \R^{t\times t}$ defined in Algorithm \ref{def:alg_tau_rho}.
	
	\noindent (\textbf{Step 1}). In this step, we will prove that for any $q>1$, there exists some $c_t=c_t(t,q)>0$ such that
	\begin{align}\label{ineq:cons_tau_rho_step1}
	\E^{(0)}\pnorm{\E_{\pi_m} \overline{\bm{\tau}}^{[t]}_{\pi_m}-\bm{\tau}^{[t]} }{\op}^q\leq \big(K\Lambda L_\mu\big)^{ c_t} \cdot n^{-1/c_t}.
	\end{align}
	Consider the map $H_{t;k}:\R^{[0:t]}\to \R^{t\times t}$:
	\begin{align}\label{ineq:cons_tau_rho_step1_1}
	H_{t;k}(u_{[0:t]}) &\equiv \phi\cdot \Big[I_t- M_{t;k}(u_{[0:t]}) \mathfrak{O}_t(\bm{\rho}^{[t-1]})\Big]^{-1}  M_{t;k}(u_{[0:t]}),
	\end{align}
	where
	\begin{align*}
	M_{t;k}(u_{[0:t]})\equiv \mathrm{diag}\Big(\Big\{-\eta_{s-1}\partial_{11} \mathsf{L}_{s-1;k}(u_s,\mathcal{F}(u_0,\xi_k))  \Big\}_{s \in [1:t]}\Big).
	\end{align*}
	By (\ref{ineq:cons_tau_rho_1}),
	\begin{align}\label{ineq:cons_tau_rho_step1_2}
	\E_{\pi_m} \overline{\bm{\tau}}^{[t]}_{\pi_m} = \E_{\pi_m} H_{t;\pi_m}\big((A\mu_\ast)_{\pi_m},\big\{(A\mu^{(r-1)})_{\pi_m}\big\}_{r \in [1:t]}\big).
	\end{align}
	On the other hand, by (\ref{ineq:rho_tau_bound_Upsilon}), 
	\begin{align}\label{ineq:cons_tau_rho_step1_3}
	\bm{\tau}^{[t]}& =  \E^{(0)} H_{t;\pi_m}\big(\mathfrak{Z}^{(0)},\{\Theta_{r;\pi_m}(\mathfrak{Z}^{([0:r])})\}_{r \in [1:t]}\big).
	\end{align}
	Combining (\ref{ineq:cons_tau_rho_step1_2})-(\ref{ineq:cons_tau_rho_step1_3}), in view of Theorem \ref{thm:gd_se}-(2), it remains to provide a bound on the Lipschitz constant $\Lambda_{H_t}$ of the maps $\{(H_{t;k})_{r,s}: \R^{[0:t]\to \R}\}_{k \in [m], r,s \in [t]}$. To this end, as $M_{t;k}(u_{[0:t]}) \mathfrak{O}_t(\bm{\rho}^{[t-1]})$ is lower triangular with diagonal elements all equal to $0$, $\big(M_{t;k}(u_{[0:t]}) \mathfrak{O}_t(\bm{\rho}^{[t-1]})\big)^r=0_{t\times t}$ for all $r\geq t$, so using Lemma \ref{lem:rho_tau_bound},
	\begin{align*}
	\bigpnorm{\big(I_t- M_{t;k}(u_{[0:t]}) \mathfrak{O}_t(\bm{\rho}^{[t-1]}) \big)^{-1}}{\op}&\leq \sum_{r=0}^{t-1} \bigpnorm{M_{t;k}(u_{[0:t]}) \mathfrak{O}_t(\bm{\rho}^{[t-1]}) }{\op}^r\leq \big(K\Lambda\big)^{c_t}.
	\end{align*}
	We may now proceed to control 
	\begin{align*}
	\pnorm{H_{t;k}(u_{[0:t]})-H_{t;k}(u_{[0:t]}')}{\op}&\leq \big( K\Lambda\big)^{c_t}\cdot \pnorm{ M_{t;k}(u_{[0:t]})-M_{t;k}(u_{[0:t]}') }{\op}\\
	&\leq \big(K\Lambda\big)^{c_t}\cdot \pnorm{u_{[0:t]}-u_{[0:t]}'}{}.
	\end{align*}
	Consequently, we may take $\Lambda_{H_t}= \big( K\Lambda\big)^{c_t}$ to conclude (\ref{ineq:cons_tau_rho_step1}).
	
	\noindent (\textbf{Step 2}). In this step, we prove that 
	\begin{align}\label{ineq:cons_tau_rho_step2}
	\pnorm{\E_{\pi_m} \overline{\bm{\tau}}^{[t]}_{\pi_m}-\hat{\bm{\tau}}^{[t]} }{\op}\leq \big(K\Lambda\big)^{c_t}\cdot \epsilon_{\rho;t-1}.
	\end{align}
	Comparing (\ref{ineq:cons_tau_rho_1}) and Algorithm \ref{def:alg_tau_rho}, we have for any $k \in [m]$,
	\begin{align*}
	\pnorm{\overline{\bm{\tau}}^{[t]}_k-\hat{\bm{\tau}}^{[t]}_k}{\op}\leq \big(K\Lambda\big)^{c_t}\cdot \epsilon_{\rho;t-1}.
	\end{align*}
	Taking average over $k \in [m]$ we conclude (\ref{ineq:cons_tau_rho_step2}) by adjusting constants.
	
	\noindent (\textbf{Step 3}). Finally, we combine (\ref{ineq:cons_tau_rho_step1}) in Step 1 and (\ref{ineq:cons_tau_rho_step2}) in Step 2 to conclude the desired estimate.
\end{proof}

\begin{lemma}\label{lem:epi_rho}
	Under the same assumptions as in Theorem \ref{thm:consist_tau_rho}, for any $q>1$, there exists some constant $c_t=c_t(t,q)>0$ such that 
	\begin{align*}
	\E^{(0)} \epsilon_{\rho;t}^q&\leq (K\Lambda)^{c_t}\cdot \E^{(0)} \epsilon_{\tau;t}^q + (K\Lambda L_\mu )^{ c_t} \cdot n^{-1/c_t}.
	\end{align*}
\end{lemma}

\begin{proof}
	Consider the auxiliary sequence defined by 
	\begin{align}\label{ineq:cons_rho_tau_1}
	\overline{\bm{\rho}}_\ell^{[t]}\equiv \hat{\bm{P}}^{[t]}_\ell\big[I_t+(\bm{\tau}^{[t]}+I_t)\mathfrak{O}_t(\overline{\bm{\rho}}_\ell^{[t-1]})\big],\quad \ell \in [n].
	\end{align}
	Here recall $\hat{\bm{P}}^{[t]}_\ell=\mathrm{diag}\big(\big\{\mathsf{P}_{s;\ell}'\big(\bigiprod{e_\ell}{\mu^{(s-1)}-\eta_{s-1} A^\top \partial_1 \mathsf{L}_{s-1}(A \mu^{(s-1)},Y) }\big)\big\}_{s \in [t]}\big)\in \R^{t\times t}$ defined in Algorithm \ref{def:alg_tau_rho}.

	\noindent (\textbf{Step 1}). In this step, we will prove that for any $q>1$, there exists some $c_t=c_t(t,q)>0$ such that 
	\begin{align}\label{ineq:cons_rho_tau_step1}
	\E^{(0)}\pnorm{\E_{\pi_n} \overline{\bm{\rho}}^{[t]}_{\pi_n}-\bm{\rho}^{[t]} }{\op}^q\leq   (K\Lambda L_\mu)^{ c_t} \cdot n^{-1/c_t}.
	\end{align}
	For any $\ell \in [n]$, let the mapping $G_{t;\ell}:\R^{[1:t]}\to \R^{t\times t}$ be defined recursively via
	\begin{align}\label{ineq:cons_rho_tau_step1_1}
	G_{t;\ell}(v_{[1:t]})\equiv N_{t;\ell}(v_{[1:t]}) \big[I_t + (\bm{\tau}^{[t]}+I_t)\mathfrak{O}_t\big(G_{t-1;\ell}(v_{[1:t-1]}) \big)\big],
	\end{align}
	where 
	\begin{align*}
	N_{t;\ell}(v_{[1:t]})\equiv \mathrm{diag}\big(\big\{(\mathsf{P}_{s;\ell})'(v_s)\big\}_{s \in [1:t]}\big).
	\end{align*}
	For notational convenience, we also let 
	\begin{align*}
	\hat{\Delta}_t\equiv \mu^{(t-1)}-\eta_{t-1}\cdot A^\top \partial_1 \mathsf{L}_{t-1}\big(A\mu^{(t-1)},Y\big).
	\end{align*}
	Then by comparing (\ref{ineq:cons_rho_tau_1}) and (\ref{ineq:cons_rho_tau_step1_1}), we have 
	\begin{align}\label{ineq:cons_rho_tau_step1_2}
	\E_{\pi_n} \overline{\bm{\rho}}_{\pi_n}^{[t]} = \E_{\pi_n} G_{t;\pi_n}\big(\big\{\hat{\Delta}_{r;\pi_n}\big\}_{r \in [1:t]}\big).
	\end{align}
	Recall $\bm{\Omega}_{\ell}^{';[t]}(w_{[1:t]})\in \R^{t\times t}$ in (\ref{ineq:rho_tau_bound_Omega}). Compared with (\ref{ineq:cons_rho_tau_step1_1}), we have
	\begin{align*}
	\bm{\Omega}_{\ell}^{';[t]}(w_{[1:t]}) = G_{t;\ell}\Big(\big\{\Delta_{r;\ell}(w_{[1:r]})\big\}_{r\in [1:t]}\Big).
	\end{align*}
	So by the definition of $\rho_{\cdot,\cdot}$,
	\begin{align}\label{ineq:cons_rho_tau_step1_3}
	\bm{\rho}^{[t]}&=\E^{(0)}  \bm{\Omega}_{\pi_n}^{';[t]}(\mathfrak{W}^{([1:t])}) = \E^{(0)} G_{t;\pi_n}\Big(\big\{\Delta_{r;\pi_n}(\mathfrak{W}^{([1:r])})\big\}_{r\in [1:t]}\Big).
	\end{align}
	On the other hand, using the same notation as in Step 1 in the proof of Theorem \ref{thm:gd_se}, by noting that $\hat{\Delta}_t=v^{(t)}$ and the distributional relation (\ref{ineq:gd_se_chg_not}), Theorem \ref{thm:GFOM_se_asym_avg} shows that for a sequence of $\Lambda_\psi$-pseudo-Lipschitz functions $\{\psi_\ell:\R^{t} \to \R\}_{\ell \in [n]}$ of order $\mathfrak{p}$ and any $q>0$, by enlarging $c_t=c_t(t,\mathfrak{p},q)>1$ if necessary,
	\begin{align}\label{ineq:cons_rho_tau_step1_4}
	& \E^{(0)}\biggabs{ \frac{1}{n}\sum_{\ell \in [n]} \Big(\psi_{\ell}\big(\{\hat{\Delta}_{r;\ell}\}_{r \in [1:t]}\big)- \E^{(0)} \psi_{\ell}\big(\big\{\Delta_{r,\ell}(\mathfrak{W}^{[1:r]})\big\}_{r \in [1:t]}\big)\Big) }^{q}\nonumber\\
	& \leq \big(K\Lambda \Lambda_\psi L_\mu\big)^{c_t} \cdot n^{-1/c_t}\equiv \err(\Lambda_\psi).
	\end{align}
	Consequently, using (\ref{ineq:cons_rho_tau_step1_2})-(\ref{ineq:cons_rho_tau_step1_4}), with $\Lambda_{G_t}$ denoting the maximal Lipschitz constant of $\big\{G_{t;\ell}: (\R^{[1:t]},\pnorm{\cdot}{})\to (\R^{t\times t},\pnorm{\cdot}{\op})\}_{\ell \in [n]}$, we have 
	\begin{align}\label{ineq:cons_rho_tau_step1_5}
	\E^{(0)}\pnorm{\E_{\pi_n} \overline{\bm{\rho}}_{\pi_n}^{[t]}-\bm{\rho}^{[t]}  }{\op}^q  \lesssim_q \err(\Lambda_{G_t}).
	\end{align}
	It therefore remains to provide a control for $\Lambda_{G_t}$. Using (\ref{ineq:cons_rho_tau_step1_1}) and Lemma \ref{lem:rho_tau_bound}, 
	\begin{align*}
	\bigpnorm{G_{t;\ell}(v_{[1:t]})}{\op}&\leq \pnorm{N_{t;\ell}(v_{[1:t]})}{\op}\cdot \big[ 1 +  \big(\pnorm{\bm{\tau}^{[t]}}{\op}+1\big)\cdot \pnorm{ G_{t-1;\ell}(v_{[1:t-1]}) }{\op}\big]\\
	&\leq \Lambda +  \big(K\Lambda\big)^{c_t}\cdot \bigpnorm{ G_{t-1;\ell}(v_{[1:t-1]}) }{\op}.
	\end{align*}
	Iterating the bound, 
	\begin{align*}
	\sup_{v_{[1:t]} \in \R^{[1:t]} }\pnorm{G_{t;\ell}(v_{[1:t]})}{\op}\leq (K\Lambda)^{c_t}.
	\end{align*}
	Next, for $v_{[1:t]}, v_{[1:t]}'\in \R^{[1:t]}$, using the above estimate,
	\begin{align*}
	&\bigpnorm{G_{t;\ell}(v_{[1:t]})-G_{t;\ell}(v_{[1:t]}')}{\op}\leq \bigpnorm{N_{t;\ell}(v_{[1:t]}) - N_{t;\ell}(v_{[1:t]}')}{\op}\cdot (K\Lambda)^{c_t}\\
	&\qquad + \pnorm{N_{t;\ell}(v_{[1:t]}')}{\op}\cdot (K\Lambda)^{c_t}\cdot  \bigpnorm{G_{t-1;\ell}(v_{[1:t-1]})-G_{t-1;\ell}(v_{[1:t-1]}')}{\op}\\
	&\leq (K\Lambda)^{c_t}\cdot \pnorm{v_{[1:t]}-v_{[1:t]}'}{}+ (K\Lambda)^{c_t}\cdot  \bigpnorm{G_{t-1;\ell}(v_{[1:t-1]})-G_{t-1;\ell}(v_{[1:t-1]}')}{\op}.
	\end{align*}
	Iterating the bound, we obtain 
	\begin{align*}
	\bigpnorm{G_{t;\ell}(v_{[1:t]})-G_{t;\ell}(v_{[1:t]}')}{\op}\leq  (K\Lambda)^{c_t}\cdot \pnorm{v_{[1:t]}-v_{[1:t]}'}{}.
	\end{align*}
	So we may take $\Lambda_{G_t}=(K\Lambda)^{c_t}$, 	
	\begin{align}\label{ineq:cons_rho_tau_step1_I0_2}
	\err(\Lambda_{G_t})\leq (K\Lambda L_\mu)^{c_t} \cdot n^{-1/c_1}.
	\end{align}
	The claim (\ref{ineq:cons_rho_tau_step1})  follows now by combining (\ref{ineq:cons_rho_tau_step1_4}), (\ref{ineq:cons_rho_tau_step1_5}) and the above display (\ref{ineq:cons_rho_tau_step1_I0_2}).

	\noindent (\textbf{Step 2}). In this step, we prove that 
	\begin{align}\label{ineq:cons_rho_tau_step2}
	\pnorm{\E_{\pi_n} \overline{\bm{\rho}}^{[t]}_{\pi_n}-\hat{\bm{\rho}}^{[t]} }{\op}\leq (K\Lambda)^{c_t}\cdot \epsilon_{\tau;t}.
	\end{align}
	Comparing (\ref{ineq:cons_rho_tau_1}) and Algorithm \ref{def:alg_tau_rho}, we have
	\begin{align*}
	\pnorm{\overline{\bm{\rho}}^{[t]}_{\ell}-\hat{\bm{\rho}}^{[t]}_{\ell}}{\op}&\leq \pnorm{ \hat{\bm{\tau}}^{[t]}-\bm{\tau}^{[t]}}{\op}\cdot  \pnorm{\hat{\bm{P}}^{[t-1]}_\ell}{\op}\cdot \pnorm{\hat{\bm{\rho}}_\ell^{[t-1]}}{\op}\\
	&\quad + \big(1+ \pnorm{\bm{\tau}^{[t]}}{\op}\big)\cdot \pnorm{\hat{\bm{P}}^{[t-1]}_\ell}{\op}\cdot \pnorm{\overline{\bm{\rho}}^{[t-1]}_{\ell}-\hat{\bm{\rho}}^{[t-1]}_{\ell}}{\op}.
	\end{align*}
	Using Lemmas \ref{lem:rho_tau_bound} and \ref{lem:hat_rho_tau_bound},
	\begin{align*}
	\pnorm{\overline{\bm{\rho}}^{[t]}_{\ell}-\hat{\bm{\rho}}^{[t]}_{\ell}}{\op}\leq (K\Lambda)^{c_t}\cdot \epsilon_{\tau;t}+(K\Lambda)^{c_t}\cdot \pnorm{\overline{\bm{\rho}}^{[t-1]}_{\ell}-\hat{\bm{\rho}}^{[t-1]}_{\ell}}{\op}.
	\end{align*}
	Iterating the above bound proves the claim (\ref{ineq:cons_rho_tau_step2}) upon averaging over $\ell \in [n]$.
	
	\noindent (\textbf{Step 3}). Finally we combine (\ref{ineq:cons_rho_tau_step1}) in Step 1 and (\ref{ineq:cons_rho_tau_step2}) in Step 2 to conclude.
\end{proof}

\begin{proof}[Proof of Theorem \ref{thm:consist_tau_rho}]
	Combining Lemmas \ref{lem:epi_tau} and \ref{lem:epi_rho}, the following estimate holds for both $*\in \{\rho,\tau\}$:
	\begin{align*}
	\E^{(0)} \epsilon_{*;t}^q&\leq (K\Lambda)^{c_t}\cdot \E^{(0)} \epsilon_{*;t-1}^q +  (K\Lambda L_\mu)^{ c_t} \cdot n^{-1/c_t}.
	\end{align*} 
	We may conclude the desired estimate by iterating the above estimate, and using the initial condition $\E^{(0)} \epsilon_{\tau;t}^q\leq \big(K\Lambda L_\mu\big)^{ c_t} \cdot n^{-1/c_t}$ (which follows by an application of Theorem \ref{thm:gd_se}-(2)).
\end{proof}

\subsection{Proof of Theorem \ref{thm:gen_error}}

\begin{lemma}\label{lem:gen_err_replace_Z}
	Suppose the assumptions in Theorem \ref{thm:consist_tau_rho} hold, and additionally $\{\mathsf{H}_{\mathcal{F};k}\}_{k \in [m]}$ are $\Lambda$-pseudo-Lipschitz functions of order 2.  Then for any $q>1$ there exists some $c_t=c_t(t,q)>1$ such that
	\begin{align}\label{ineq:gen_err_step1}
	\E^{(0)}\abs{\overline{\mathscr{E}}_{\mathsf{H}}^{(t)}-\hat{\mathscr{E}}_{\mathsf{H}}^{(t) }}^q\leq (K\Lambda L_\mu)^{c_t}\cdot n^{-1/c_t}.
	\end{align} 
\end{lemma}

\begin{proof}
	Note that by the pseudo-Lipschitz property of $\mathsf{H}$,
	\begin{align}\label{ineq:gen_err_step1_3}
	\abs{\overline{\mathscr{E}}_{\mathsf{H}}^{(t)}-\hat{\mathscr{E}}_{\mathsf{H}}^{(t) }}&\leq \frac{c_0\Lambda }{m}\sum_{k \in [m]} \abs{ \hat{Z}^{(t)}_k-Z^{(t)}_k }\cdot \big(1+\abs{Z^{(t)}_k}+\abs{\hat{Z}^{(t)}_k}+\abs{\iprod{A_k}{\mu_\ast}}\big)\nonumber\\
	&\leq c_0\Lambda\cdot \frac{\pnorm{\hat{Z}^{(t)}-Z^{(t)}}{}}{\sqrt{m}}\cdot \bigg(1+ \frac{ \pnorm{Z^{(t)}}{} }{\sqrt{m} }+ \frac{ \pnorm{\hat{Z}^{(t)}}{} }{\sqrt{m} } + \frac{\pnorm{A}{\op} \pnorm{\mu_\ast}{}}{\sqrt{m}} \bigg).
	\end{align}
	
	\noindent \emph{Bound for $\pnorm{\hat{Z}^{(t)}-Z^{(t)}}{}$}. First, note that for any $s \in [t]$,
	\begin{align}\label{ineq:gen_err_step1_0}
	\pnorm{\partial_1\mathsf{L}_{s-1}\big(A\mu^{(s-1)},Y\big)}{}\leq c_0 \Lambda (1+\pnorm{A}{\op})\cdot \big(\sqrt{m}+\pnorm{\mu_\ast}{}+\pnorm{\mu^{(s-1)}}{}\big).
	\end{align}
	Comparing the definitions of (\ref{def:Z_W}) for $Z^{(t)}$ and (\ref{def:Z_hat}) for $\hat{Z}^{(t)}$, we then have 
	\begin{align}\label{ineq:gen_err_step1_1}
	&\pnorm{\hat{Z}^{(t)}-Z^{(t)}}{ }\leq \Lambda t\cdot \pnorm{\hat{\bm{\rho}}^{[t]}-\bm{\rho}^{[t]} }{\op}\cdot \max_{s \in [1:t]} \pnorm{\partial_1\mathsf{L}_{s-1}\big(A\mu^{(s-1)},Y\big)}{}\nonumber\\
	&\leq c_0\cdot \Lambda^2 t\cdot \big(1+\pnorm{A}{\op}\big)\cdot  \pnorm{\hat{\bm{\rho}}^{[t]}-\bm{\rho}^{[t]} }{\op}\cdot \Big(\sqrt{m}+\pnorm{\mu_\ast}{}+\max_{s \in [1:t]}\pnorm{\mu^{(s-1)}}{}\Big).
	\end{align}
	On the other hand, using (\ref{def:grad_descent}), we have
	\begin{align*}
	\pnorm{\mu^{(t)}}{}&\leq \pnorm{\mathsf{P}_t(0)}{}+\Lambda\cdot \pnorm{\mu^{(t-1)}-\eta_{t-1}\cdot A^\top \partial_1 \mathsf{L}_{t-1} (A\mu^{(t-1)},Y)}{}\\
	&\leq \sqrt{n} \Lambda+ \Lambda \pnorm{\mu^{(t-1)}}{}+ \Lambda^2\cdot \pnorm{A}{\op}\cdot \pnorm{\partial_1 \mathsf{L}_{t-1} (A\mu^{(t-1)},Y)}{}\\
	&\leq  \big(K\Lambda (1+\pnorm{A}{\op})\big)^{c_0} \cdot \big(\sqrt{m}+\pnorm{\mu_\ast}{}+\pnorm{\mu^{(t-1)}}{}\big).
	\end{align*}
	Iterating the bound we obtain
	\begin{align}\label{ineq:gen_err_step1_2}
	\pnorm{\mu^{(t)}}{}/\sqrt{n}&\leq \big(K\Lambda L_\mu (1+\pnorm{A}{\op})\big)^{c_0 t}.
	\end{align}
	Combined with (\ref{ineq:gen_err_step1_1}), we have
	\begin{align}\label{ineq:gen_err_step1_4}
	\pnorm{\hat{Z}^{(t)}-Z^{(t)}}{}/\sqrt{m}\leq \big(K\Lambda L_\mu (1+\pnorm{A}{\op})\big)^{c_0 t}\cdot \pnorm{\hat{\bm{\rho}}^{[t]}-\bm{\rho}^{[t]} }{\op}.
	\end{align}
	
	\noindent \emph{Bounds for $\pnorm{{Z}^{(t)}}{}$ and $\pnorm{\hat{Z}^{(t)}}{}$}. Using the definition of $Z^{(t)}$ in (\ref{def:Z_W}), along with the estimates (\ref{ineq:gen_err_step1_0}), (\ref{ineq:gen_err_step1_2}) and Lemma \ref{lem:rho_tau_bound}, we have
	\begin{align}\label{ineq:gen_err_step1_5}
	\pnorm{Z^{(t)}}{}/\sqrt{m}&\leq \pnorm{A}{\op} \cdot \pnorm{\mu^{(t)}}{}/\sqrt{m} + t \Lambda \pnorm{\bm{\rho}^{[t]} }{\op}\cdot \max_{s \in [1:t]} \pnorm{\partial_1\mathsf{L}_{s-1}\big(A\mu^{(s-1)},Y\big)}{}/\sqrt{m}\nonumber\\
	&\leq \big(K\Lambda L_\mu\big)^{c_t}\cdot (1+\pnorm{A}{\op})^{c_0 t}.
	\end{align}
	On the other hand, using the definition of $\hat{Z}^{(t)}$ in (\ref{def:Z_hat}), and now using Lemma \ref{lem:hat_rho_tau_bound}, the above estimate remains valid.
	
	The claimed estimate now follows by combining (\ref{ineq:gen_err_step1_3}) with the estimates in (\ref{ineq:gen_err_step1_4})-(\ref{ineq:gen_err_step1_5}), and then using Theorem \ref{thm:consist_tau_rho}.
\end{proof}

\begin{proof}[Proof of Theorem \ref{thm:gen_error}]
	The claim follows from Theorem \ref{thm:gen_error_oracle} and Lemma \ref{lem:gen_err_replace_Z} above.
\end{proof}

\section{Proofs for Section \ref{section:example_linear}}\label{section:proof_linear}

\subsection{Proof of Lemma \ref{lem:linear_delta}}
	Note that in the single-index regression model,
	\begin{align*}
	\Upsilon_t(\mathfrak{z}^{([0:t])})& = - \eta \cdot  \mathsf{L}_\ast'\bigg(\mathfrak{z}^{(t)}-\varphi_\ast(\mathfrak{z}^{(0)}) + \sum_{s \in [1:t-1]}\rho_{t-1,s}\Upsilon_s(\mathfrak{z}^{([0:s])})-\xi\bigg).
	\end{align*}
	This means $\Upsilon_t(\mathfrak{z}^{([0:t])})$ depend on $\mathfrak{z}^{([0:t])}$ only through $\{\mathfrak{z}^{(s)}-\varphi_\ast(\mathfrak{z}^{(0)})\}_{s \in [1:t]}$. In other words, let $\mathsf{F}_t: \R^{m\times [1:t]}\to \R^m$ be defined recursively via
	\begin{align*}
	\mathsf{F}_t(u^{[1:t]})\equiv - \eta \cdot  \mathsf{L}_\ast'\bigg(u^{(t)} + \sum_{s \in [1:t-1]}\rho_{t-1,s}\mathsf{F}_s(u^{([1:s])})-\xi\bigg).
	\end{align*}
	Then $
	\Upsilon_t(\mathfrak{z}^{([0:t])})=\mathsf{F}_t\big(\mathfrak{z}^{(1)}-\varphi_\ast(\mathfrak{z}^{(0)}),\ldots,\mathfrak{z}^{(t)}-\varphi_\ast(\mathfrak{z}^{(0)})\big)$, and therefore by chain rule,
	\begin{align*}
	\partial_{ \mathfrak{z}^{(0)}}\Upsilon_t(\mathfrak{z}^{([0:t])}) &= -\sum_{s \in [1:t]} \partial_s \mathsf{F}_t \big(\mathfrak{z}^{(1)}-\varphi_\ast(\mathfrak{z}^{(0)}),\ldots,\mathfrak{z}^{(t)}-\varphi_\ast(\mathfrak{z}^{(0)})\big)\cdot \varphi_\ast'(\mathfrak{z}^{(0)})\\
	&=-\sum_{s \in [1:t]} \partial_{ \mathfrak{z}^{(s)}}\Upsilon_t(\mathfrak{z}^{([0:t])})\cdot \varphi_\ast'(\mathfrak{z}^{(0)}).
	\end{align*}
	Taking expectation over the Gaussian laws of $\mathfrak{Z}^{([0:t])}$ and using the definition of $\delta_t$ and $\{\tau_{t,s}\}$ in (S3), we have
	\begin{align*}
	\delta_t &= \phi\cdot \E^{(0)}\partial_{\mathfrak{Z}^{(0)}} \Upsilon_{t;\pi_m}(\mathfrak{Z}^{([0:t])})\cdot \varphi_\ast'(\mathfrak{Z}^{(0)})\\
	&= -  \sum_{s \in [1:t]} \phi\cdot  \E^{(0)}\partial_{\mathfrak{Z}^{(s)}} \Upsilon_{t;\pi_m}(\mathfrak{Z}^{([0:t])})\cdot \varphi_\ast'(\mathfrak{Z}^{(0)}).
	\end{align*}
	For linear model, as $\varphi_\ast'\equiv 1$, the last display equals to $-\sum_{s \in [1:t]} \tau_{t,s}$. This completes the proof. \qed

\subsection{Proof of Proposition \ref{prop:linear_bias_var}}

For the bias, by Lemma \ref{lem:linear_delta},
\begin{align*}
b^{(t)}_{\mathrm{db}} = -e_t^\top \big(\bm{\tau}^{[t]}\big)^{-1}\bm{\delta}^{[t]}& = -\sum_{s \in [1:t]} \big(\bm{\tau}^{[t]}\big)^{-1}_{t,s} \delta_s = \sum_{s \in [1:t]} \big(\bm{\tau}^{[t]}\big)^{-1}_{t,s} \sum_{r \in [1:s]} (\bm{\tau}^{[t]})_{s,r} \\
& = \sum_{r \in [1:t]} \sum_{s \in [r:t]} \big(\bm{\tau}^{[t]}\big)^{-1}_{t,s}(\bm{\tau}^{[t]})_{s,r} = \sum_{r \in [1:t]}  \bm{1}_{t=r} = 1,
\end{align*}
proving the first claim.

For the variance, under the squared loss we have a further simplification:
\begin{align}\label{ineq:linear_bias_var_2}
\Upsilon_t(\mathfrak{z}^{([0:t])})& = - \eta \cdot  \bigg(\mathfrak{z}^{(t)}-\mathfrak{z}^{(0)} -\xi + \sum_{r \in [1:t-1]}\rho_{t-1,r}\Upsilon_r(\mathfrak{z}^{([0:r])})\bigg).
\end{align}
Consequently, for any $k \in [m]$, by writing $\bm{\Upsilon}_k^{(t)}(z_{[0:t]})\equiv \big(  \Upsilon_{r;k}(z_{[0:r]})\big)_{r \in [1:t]} \in \R^t$, we may represent the above display in the matrix form:
\begin{align*}
\bm{\Upsilon}_k^{(t)}(z_{[0:t]}) = -\eta \big(z_{r}-z_{0}-\xi_k\big)_{r \in [1:t]} -\eta\cdot \mathfrak{O}_t(\bm{\rho}^{[t-1]})\bm{\Upsilon}_k^{(t)}(z_{[0:t]}).
\end{align*}
Solving for $\bm{\Upsilon}_k^{(t)}(z_{[0:t]})$ we obtain
\begin{align}\label{ineq:linear_bias_var_1}
\bm{\Upsilon}_k^{(t)}(z_{[0:t]}) = -\eta\cdot \big(I+\eta \mathfrak{O}_t(\bm{\rho}^{[t-1]})\big)^{-1}  \big(z_r-z_{0}-\xi_k\big)_{r \in [1:t]}.
\end{align}
This means with 
\begin{align*}
\Sigma_{\mathscr{E}}^{[t]}\equiv \E^{(0)}(\mathfrak{Z}^{(r)}-\mathfrak{Z}^{(0)})_{r \in [1:t]} (\mathfrak{Z}^{(r)}-\mathfrak{Z}^{(0)})_{r \in [1:t]}^\top + \sigma_m^2\cdot \bm{1}_t \bm{1}_t^\top,
\end{align*}
we have
\begin{align}\label{ineq:linear_bias_var_3}
\Sigma_{\mathfrak{W}}^{[t]}&=\phi\cdot \E^{(0)} \bm{\Upsilon}_{\pi_m}^{(t)}(\mathfrak{Z}^{([0:t])}) \bm{\Upsilon}_{\pi_m}^{(t),\top}(\mathfrak{Z}^{([0:t])})\nonumber\\
& = \phi\eta^2\cdot \big(I+\eta \mathfrak{O}_t(\bm{\rho}^{[t-1]})\big)^{-1} \Sigma_{\mathscr{E}}^{[t]} \big(I+\eta \mathfrak{O}_t(\bm{\rho}^{[t-1]})\big)^{-\top}.
\end{align}
By taking derivative on both side of (\ref{ineq:linear_bias_var_1}),  $\big(  \partial_{(s)}\Upsilon_{r;k}(z_{[0:r]})\big)_{s,r \in [1:t]}=-\eta \big(I+\eta \mathfrak{O}_t(\bm{\rho}^{[t-1]})\big)^{-1}$. By definition of $\bm{\tau}^{[t]}$, we then have 
\begin{align}\label{ineq:linear_bias_var_4}
\bm{\omega}^{[t]}=(\bm{\tau}^{[t]})^{-1}= -(\phi\eta)^{-1} \cdot \big(I+\eta \mathfrak{O}_t(\bm{\rho}^{[t-1]})\big).
\end{align}
Combining (\ref{ineq:linear_bias_var_3}) and (\ref{ineq:linear_bias_var_4}), we may compute
\begin{align*}
(\sigma^{(t)}_{\mathrm{db}})^2&\equiv e_t^\top \bm{\omega}^{[t]} \Sigma_{\mathfrak{W}}^{[t]} \bm{\omega}^{[t],\top} e_t = \phi^{-1}\cdot e_t^\top \Sigma_{\mathscr{E}}^{[t]} e_t = \phi^{-1}\big\{\E^{(0)}\big(\mathfrak{Z}^{(t)}-\mathfrak{Z}^{(0)}\big)^2+\sigma_m^2\big\},
\end{align*} 
proving (1). The claim in (2) follows immediately.\qed

\section{Proofs for Section \ref{section:example_1bit}}\label{section:proof_logistic}

\subsection{The smoothed problem}

Let $\varphi \in C^\infty(\R)$ be such that $\varphi$ is non-decreasing, taking values in $[-1,1]$ and $\varphi|_{(-\infty,-1]}\equiv -1$ and $\varphi|_{[1,\infty)}\equiv 1$. For any $\epsilon>0$, let $\varphi_\epsilon(\cdot)\equiv \varphi(\cdot/\epsilon)$. For notational convenience, we write $\varphi_0(\cdot)\equiv \sign(\cdot)$. We also define the following `smoothed' quantities:
\begin{itemize}
	\item Let $\mathcal{F}_\epsilon(z,\xi)\equiv \varphi_\epsilon (z+\xi)$.
	\item Let $\texttt{SE}_\epsilon^{(t)}\equiv \big(\{\Upsilon_{\epsilon;t}\},\{\Omega_{\epsilon;t}\},\Sigma_{\epsilon;\mathfrak{Z}}^{[t]},\Sigma_{\epsilon;\mathfrak{W}}^{[t]},\bm{\tau}_\epsilon^{[t]},\bm{\rho}_\epsilon^{[t]},\bm{\delta}_\epsilon^{[t]}\big)$ be smoothed state evolution parameters obtained by replacing $\mathcal{F}$ with $\mathcal{F}_\epsilon$.
	\item Let $\mathfrak{Z}_\epsilon^{([0:t])}\sim \mathcal{N}(0,\Sigma_{\epsilon;\mathfrak{Z}}^{[t]})$ and $\mathfrak{W}_\epsilon^{([1:t])}\sim \mathcal{N}(0,\Sigma_{\epsilon;\mathfrak{W}}^{[t]})$ be the smoothed versions of $\mathfrak{Z}^{([0:t])},\mathfrak{W}^{([1:t])}$.
	\item Let $\{\Theta_{\epsilon;t}\}$, $\{\Delta_{\epsilon;t}\}$ be defined as (\ref{def:Theta_fcn}), (\ref{def:Delta_fcn}) by replacing $\texttt{SE}_0^{(t)}$ with $\texttt{SE}_\epsilon^{(t)}$.
	\item Let $Y_{\epsilon;i}\equiv \varphi_\epsilon(\iprod{A_i}{\mu_\ast}+\xi_i)$, $i \in [m]$, be smoothed observations.
	\item Let $\mu_\epsilon^{(t)} = \prox_{\eta \mathsf{f}}\big(\mu^{(t-1)}_\epsilon-\eta \cdot A^\top \partial_1\mathsf{L}(A\mu^{(t-1)}_\epsilon,Y_\epsilon)\big)$ be the smoothed gradient descent iterate. 
	\item Let $Z_\epsilon^{(t)}, W_\epsilon^{(t)}$ be defined as (\ref{def:Z_W}) by replacing $(\{\mu^{(\cdot)}\},\texttt{SE}_0^{(\cdot)})$ with $(\{\mu_\epsilon^{(\cdot)}\},\texttt{SE}_\epsilon^{(\cdot)})$.
	\item Let $\hat{\bm{\tau}}_\epsilon^{[t]}$, $\hat{\bm{\rho}}_\epsilon^{[t]}$ be the output of Algorithm \ref{def:alg_tau_rho} by replacing $(\{\mu^{(\cdot)}\},Y)$ with $(\{\mu_\epsilon^{(\cdot)}\},Y_\epsilon)$.
	\item Let $\mathscr{E}_{\epsilon;\mathsf{H}}^{(t)}$ be defined as (\ref{def:gen_err}) by replacing $(\{\mu^{(\cdot)}\},Y)$ with $(\{\mu_\epsilon^{(\cdot)}\},Y_\epsilon)$.
	\item Let $\hat{Z}_\epsilon^{(t)}$ be defined as (\ref{def:Z_hat}) by replacing $(\{\mu^{(\cdot)}\},Y,\hat{\bm{\rho}}^{[t]})$ with $(\{\mu_\epsilon^{(\cdot)}\},Y_\epsilon,\hat{\bm{\rho}}_\epsilon^{[t]})$.
	\item Let $\hat{\mathscr{E}}_{\epsilon;\mathsf{H}}^{(t)}$ be defined as (\ref{def:gen_err_est}) by replacing $(Y,\hat{Z})$ with $(Y_\epsilon,\hat{Z}_\epsilon)$.
	\item Let $\mu^{(t)}_{\epsilon;\mathrm{db}}$, $b^{(t)}_{\epsilon;\mathrm{db}}$ and $\sigma^{(t)}_{\epsilon;\mathrm{db}}$ be defined as (\ref{def:debias_gd_oracle})-(\ref{def:db_gd_bias_var}), by replacing $(\{\mu^{(\cdot)}\},Y,\texttt{SE}_0^{(t)})$ with $(\{\mu_\epsilon^{(\cdot)}\},Y_\epsilon,\texttt{SE}_\epsilon^{(t)})$.
\end{itemize}
Notation with subscript $0$ will be understood as the unsmoothed version.

\subsection{Stability of smoothed state evolution}
The following apriori estimates will be useful.
\begin{lemma}\label{lem:1bit_cs_apriori_est}
	Suppose (A1), (A3), (A4') and (A5') hold. The following hold for some $c_t=c_t(t)>1$:
	\begin{enumerate}
		\item $\sup_{\epsilon\geq 0} (\pnorm{\bm{\tau}_\epsilon^{[t]}}{\op}+ \pnorm{\bm{\rho}_\epsilon^{[t]}}{\op})\leq (K\Lambda)^{c_t}$.
		\item $\sup_{\epsilon\geq 0}\big(\pnorm{\Sigma_{\epsilon;\mathfrak{Z}}^{[t]}}{\op}+ \pnorm{\Sigma_{\epsilon;\mathfrak{W}}^{[t]}}{\op}+ \pnorm{\bm{\delta}_\epsilon^{[t]}}{}\big)\leq \big(K\Lambda L_\mu (1\wedge \sigma_{\mu_\ast})^{-1}\big)^{c_t}$.
		\item $\sup\limits_{\epsilon\geq 0}\max\limits_{k \in [m], r \in [1:t]} \big(\abs{\Upsilon_{\epsilon;r;k}(z^{([0:r])})}+\abs{\Theta_{\epsilon;r;k}(z^{([0:r])})}\big)\leq (K\Lambda)^{c_t}\cdot \big(1+\pnorm{z^{([0:t])} }{}\big)$.
		\item $\sup\limits_{\epsilon\geq 0}\max\limits_{\ell \in [n], r \in [1:t]} \abs{\Omega_{\epsilon;r;\ell}(w^{([1:r])})}\leq \big(K\Lambda L_\mu(1\wedge\sigma_{\mu_\ast})^{-1}\big)^{c_t}\cdot \big(1+\pnorm{w^{([1:t])}}{}\big)$.
		\item $\sup\limits_{\epsilon\geq 0}\max\limits_{\substack{k \in [m], \ell \in [n], r,s \in [1:t]}} \big(\abs{\partial_{(s)}\Upsilon_{\epsilon;r;k}(z^{([0:r])})}+\abs{\partial_{(s)}\Omega_{\epsilon;r;\ell}(w^{([1:r])})}\big)\leq (K\Lambda)^{c_t}.$
	\end{enumerate}
\end{lemma}

\begin{proof}
For notational convenience, we assume $\sigma_{\mu_\ast}^{-1}\leq 1$. We may follow exactly the same proof as in Lemma \ref{lem:rho_tau_bound} until Step 2. A crucial modification lies in Step 3. In the current setting, in order to get a uniform-in-$\epsilon$ estimate, by using the Gaussian integration-by-parts formula (\ref{def:delta_t_alternative}) for $\delta_{t}$, 
\begin{align}\label{ineq:1bit_cs_apriori_est_1}
\pnorm{\bm{\delta}_\epsilon^{[t]}}{}\leq \sigma_{\mu_\ast}^{-1} (K\Lambda)^{c_t}\cdot \pnorm{\Sigma_{\epsilon;\mathfrak{Z}}^{[t]}}{\op}.
\end{align}
The above estimate is different from (\ref{ineq:rho_tau_bound_3}), as it compensate the large Lipschitz constants involving $\epsilon^{-c_t}$ with the factor $\sigma_{\mu_\ast}^{-1}$ via the formula (\ref{def:delta_t_alternative}).

Now combining the estimate (\ref{ineq:1bit_cs_apriori_est_1}) with the first line of (c) in Step 2 of the proof of Lemma \ref{lem:rho_tau_bound}, we have
\begin{align*}
\pnorm{\Sigma_{\epsilon;\mathfrak{Z}}^{[t]}}{\op}&\leq (K\Lambda L_\mu \sigma_{\mu_\ast}^{-1})^{c_t}\cdot \big(1+\pnorm{\Sigma_{\epsilon;\mathfrak{W}}^{[t-1]}}{\op}+ \pnorm{\Sigma_{\epsilon;\mathfrak{Z}}^{[t-1]}}{\op} \big).
\end{align*}
Iterating the bound, we have
\begin{align*}
\pnorm{\Sigma_{\epsilon;\mathfrak{Z}}^{[t]}}{\op}&\leq (K\Lambda L_\mu \sigma_{\mu_\ast}^{-1})^{c_t}\cdot \Big(1+\max_{r\in [1:t-1]}\pnorm{\Sigma_{\epsilon;\mathfrak{W}}^{[r]}}{\op}\Big).
\end{align*}
Further combined with the second line of (c) in Step 2 of the proof of Lemma \ref{lem:rho_tau_bound}, 
\begin{align*}
\pnorm{\Sigma_{\epsilon;\mathfrak{Z}}^{[t]}}{\op}&\leq (K\Lambda L_\mu \sigma_{\mu_\ast}^{-1})^{c_t}\cdot \Big(1+\max_{r\in [1:t-1]}\pnorm{\Sigma_{\epsilon;\mathfrak{Z}}^{[r]}}{\op}\Big).
\end{align*}
Coupled with initial condition $\pnorm{\Sigma_{\epsilon;\mathfrak{Z}}^{[1]}}{\op}\leq L_\mu^2$, we arrive at the estimate
\begin{align*}
\pnorm{\Sigma_{\epsilon;\mathfrak{Z}}^{[t]}}{\op}+ \pnorm{\Sigma_{\epsilon;\mathfrak{W}}^{[t]}}{\op}+ \pnorm{\bm{\delta}_\epsilon^{[t]}}{}\leq (K\Lambda L_\mu \sigma_{\mu_\ast}^{-1})^{c_t}.
\end{align*}
The modified estimate for $\pnorm{\bm{\delta}_\epsilon^{[t]}}{}$ also impacts the estimate for $\Omega_\cdot$ via (b) in Step 2 of the proof of Lemma \ref{lem:rho_tau_bound}. 
\end{proof}

We now quantify the smoothing effect for state evolution parameters. Let us define a few further notation. For a state evolution parameter, and later on, a gradient descent iterate statistics $\star$, we write $\d_\epsilon \star\equiv \star_{\epsilon}-\star_{0}$. For instance, $\d_\epsilon \Upsilon_t\equiv \Upsilon_{\epsilon;t}-\Upsilon_{0;t}=\Upsilon_{\epsilon;t}-\Upsilon_{t}$, $\d_\epsilon \mu^{(t)}\equiv \mu_\epsilon^{(t)}-\mu_0^{(t)}=\mu_\epsilon^{(t)}-\mu^{(t)}$, and similar notation applies to other quantities. 

With these further notation, we define for $t\geq 1$,
\begin{align*}
\mathcal{P}_{\texttt{SE}}^{(t)}(\epsilon)\equiv \pnorm{\d_\epsilon \bm{\tau}^{[t]}}{\op}+ \pnorm{\d_\epsilon \bm{\rho}^{[t]}}{\op}+\pnorm{\d_\epsilon \Sigma_{\mathfrak{Z}}^{[t]} }{\op}+\pnorm{\d_\epsilon \Sigma_{\mathfrak{W}}^{[t]}}{\op}+\pnorm{\d_\epsilon \bm{\delta}^{[t]} }{}.
\end{align*}

	\begin{lemma}\label{lem:1bit_cs_se_smooth_diff}
		Suppose (A1), (A3) and (A4*), and (\ref{ineq:111112bLound}) hold. Then the following hold for some $c_t=c_t(t)>1$:
		\begin{enumerate}
			\item For any $\epsilon>0$, $\mathcal{P}_{\texttt{SE}}^{(t)}(\epsilon)\leq \big(K\Lambda L_\mu(1\wedge\sigma_{\mu_\ast})^{-1}\big)^{c_t}\cdot \epsilon^{1/c_t}$.
			\item For any $\epsilon>0$ and $k \in [m]$,
			\begin{align*}
			&\max_{r \in [1:t]} \big(\abs{\d_\epsilon \Upsilon_{r;k}(z^{([0:r])})}+\abs{\d_\epsilon \Theta_{r;k}(z^{([0:r])})}\big)\nonumber\\
			&\leq \big(K\Lambda L_\mu(1\wedge\sigma_{\mu_\ast})^{-1}\big)^{c_t}\cdot \Big[\big(1+\pnorm{z^{([0:t])} }{}\big)\cdot \epsilon^{1/c_t}+\abs{\d_\epsilon \varphi (z^{(0)}+\xi_k)}\Big].
			\end{align*}
			\item For any $\epsilon>0$,
			\begin{align*}
			\max_{r \in [1:t]}\max_{\ell \in [n]}\abs{\d_\epsilon \Omega_{r;\ell}(w^{([1:r])})}\leq \big(K\Lambda L_\mu(1\wedge\sigma_{\mu_\ast})^{-1}\big)^{c_t}\cdot \big(1+\pnorm{w^{([1:t])} }{}\big)\cdot \epsilon^{1/c_t}.
			\end{align*}
		\end{enumerate}
	\end{lemma}

	\begin{proof}
		For notational convenience, we assume $\sigma_{\mu_\ast}\leq 1$, and for formal consistency, we let $\mathcal{P}_{\texttt{SE}}^{(0)}(\epsilon)\equiv 0$. Fix $t\geq 1$.

		\noindent (1). Using (S1), for any $k \in [m]$,
		\begin{align*}
		&\abs{\d_\epsilon \Upsilon_{t;k} (z^{([0:t])}) } = \abs{ \Upsilon_{\epsilon;t;k}(z^{([0:t])} ) - \Upsilon_{0;t;k}(z^{([0:t])} ) } \leq \Lambda^2 \cdot  \Big(   \abs{\d_\epsilon \varphi(z^{(0)}+\xi_k )  } \\
		&\qquad + t\cdot  \max_{r \in [1:t-1]} \abs{\Upsilon_{0;r;k}(z^{([0:r])} )} \cdot  \pnorm{\d_\epsilon \bm{\rho}^{[t-1]}}{\op} +  t\cdot  \pnorm{\bm{\rho}_{\epsilon}^{[t-1]}}{\op} \cdot \max_{r \in [1:t-1]} \abs{\d_\epsilon \Upsilon_{r;k}}  \Big).
		\end{align*}
		Using the apriori estimates in Lemma \ref{lem:1bit_cs_apriori_est}, we then arrive at
		\begin{align*}
		\abs{\d_\epsilon \Upsilon_{t;k} (z^{([0:t])}) } &\leq (K\Lambda)^{c_t} \cdot  \Big(\abs{\d_\epsilon \varphi(z^{(0)}+\xi_k )  } \\
		&\qquad + (1 + \pnorm{z^{[0:t]} }{ } ) \cdot  \pnorm{\d_\epsilon \bm{\rho}^{[t-1]}}{\op}  + \max_{r \in [1:t-1]} \abs{\d_\epsilon \Upsilon_{r;k}} \Big).
		\end{align*}
		Iterating the bound and  using the initial condition $\abs{\d_\epsilon \Upsilon_{1;k} (z^{([0:1])}) } \leq \Lambda^2\cdot  \abs{\d_\epsilon \varphi(z^{(0)}+\xi_k ) } $, we have 
		\begin{align}\label{ineq:1bit_cs_se_smooth_diff_Upsilon}
		\abs{\d_\epsilon \Upsilon_{t;k} (z^{([0:t])}) } &\leq (K\Lambda)^{c_t}\cdot (1 + \pnorm{z^{([0:t])} }{ } ) \cdot \big[\pnorm{\d_\epsilon \bm{\rho}^{[t-1]}}{\op} + \abs{\d_\epsilon \varphi(z^{(0)}+\xi_k )}\big].
		\end{align}
		
		\noindent (2). Using (S3) and the apriori estimates in Lemma \ref{lem:1bit_cs_apriori_est}, for $\ell \in [n]$,
		\begin{align*}
		&\abs{\d_\epsilon \Omega_{t;\ell} (w^{([1:t])}) } = \abs{\Omega_{\epsilon;t;\ell} (w^{([1:t])}  ) - \Omega_{0;t;\ell} (w^{([1:t])}  )  } \\
		&\leq (K\Lambda L_\mu\sigma_{\mu_\ast}^{-1})^{c_t}\cdot \Big( (1 + \pnorm{w^{([1:t])} }{ } ) \cdot \pnorm{\d_\epsilon \bm{\tau}^{[t]} }{ }+  \pnorm{\d_\epsilon\bm{\delta}^{[t]}}{} +\max_{r \in [0:t-1]} \abs{\d_\epsilon \Omega_{r;\ell}}   \Big).
		\end{align*}
		Iterating the bound using the initial condition $\abs{\d_\epsilon \Omega_{0;\ell}} =0$, we obtain
		\begin{align}\label{ineq:1bit_cs_se_smooth_diff_Omega}
		\max_{\ell \in [n]}\abs{\d_\epsilon \Omega_{t;\ell} (w^{([1:t])}) } &\leq (K\Lambda L_\mu\sigma_{\mu_\ast}^{-1})^{c_t}\cdot(1 + \pnorm{w^{([1:t])} }{ } )  \cdot \big[ \pnorm{\d_\epsilon \bm{\tau}^{[t]} }{ } + \pnorm{\d_\epsilon\bm{\delta}^{[t]}}{} \big].
		\end{align}

		\noindent (3). Recall $\Sigma_{\epsilon;\mathfrak{Z}}^{[t]}=\E^{(0)} \mathfrak{Z}_\epsilon^{([0:t])} \mathfrak{Z}_\epsilon^{([0:t]),\top}$, and $\Sigma_{\epsilon;\mathfrak{W}}^{[t]}=\E^{(0)} \mathfrak{W}_\epsilon^{([1:t])} \mathfrak{W}_\epsilon^{([1:t]),\top}$. Using (S2), the apriori estimates in Lemma \ref{lem:1bit_cs_apriori_est} and (\ref{ineq:1bit_cs_se_smooth_diff_Omega}),
		\begin{align}\label{ineq:1bit_cs_se_smooth_diff_cov_Z}
		\pnorm{\d_\epsilon \Sigma_{\mathfrak{Z}}^{[t]}}{\op}&\leq  (t+1)\cdot \max_{-1\leq r,s \leq t-1} \E^{(0)} \biggabs{ \prod_{\ast \in \{r,s\}} \Omega_{\ast;\pi_n}(\mathfrak{W}^{([1:\ast])}) - \prod_{\ast \in \{r,s\}}\Omega_{\epsilon;\ast;\pi_n}(\mathfrak{W}_\epsilon^{([1:\ast])}) }\nonumber\\
		&\leq (K\Lambda L_\mu\sigma_{\mu_\ast}^{-1})^{c_t}\cdot \max_{1\leq r\leq t-1} \E^{(0),1/2} \big( \Omega_{r;\pi_n}(\mathfrak{W}^{([1:r])})-\Omega_{\epsilon;r;\pi_n}(\mathfrak{W}_\epsilon^{([1:r])}) \big)^2\nonumber\\
		&\leq (K\Lambda L_\mu\sigma_{\mu_\ast}^{-1})^{c_t}\cdot \big(\mathcal{P}_{\texttt{SE}}^{(t-1)}(\epsilon)+\pnorm{\d_\epsilon \Sigma_{\mathfrak{W}}^{[t-1]}}{\op}^{1/2}\big).
		\end{align}
		Similarly we may estimate, using the apriori estimates in Lemma \ref{lem:1bit_cs_apriori_est} and (\ref{ineq:1bit_cs_se_smooth_diff_Upsilon}),
		\begin{align*}
		\pnorm{\d_\epsilon \Sigma_{\mathfrak{W}}^{[t]}}{\op}
		&\leq (K\Lambda L_\mu\sigma_{\mu_\ast}^{-1})^{c_t} \cdot \big( \pnorm{\d_\epsilon \bm{\rho}^{[t-1]}}{\op}+ \E^{(0),1/2} \abs{\d_\epsilon \varphi (\mathfrak{Z}^{(0)}+\xi_{\pi_m})}^2 +\pnorm{\d_\epsilon \Sigma_{\mathfrak{Z}}^{[t]}}{\op}^{1/2}\big).
		\end{align*}	
		Using that for $k \in [m]$,
		\begin{align}\label{ineq:1bit_cs_se_smooth_diff_anti_conc}
		\E^{(0)} \abs{\d_\epsilon \varphi (\mathfrak{Z}^{(0)}+\xi_k)}^2\leq 4 \E^{(0)}\bm{1}\big(\abs{\mathfrak{Z}^{(0)}+\xi_k}\leq \epsilon\big)\leq 8\epsilon/\sigma_{\mu_\ast},
		\end{align}
		we have
		\begin{align}\label{ineq:1bit_cs_se_smooth_diff_cov_W}
		\pnorm{\d_\epsilon \Sigma_{\mathfrak{W}}^{[t]}}{\op}\leq (K\Lambda L_\mu\sigma_{\mu_\ast}^{-1})^{c_t}\cdot \big(\epsilon^{1/2}+\mathcal{P}_{\texttt{SE}}^{(t-1)}(\epsilon) +\pnorm{\d_\epsilon \Sigma_{\mathfrak{Z}}^{[t]}}{\op}^{1/2}\big).
		\end{align}
		Iterating across (\ref{ineq:1bit_cs_se_smooth_diff_cov_Z}) and (\ref{ineq:1bit_cs_se_smooth_diff_cov_W}), and using the apriori estimates in Lemma \ref{lem:1bit_cs_apriori_est}, we arrive at 
		\begin{align}\label{ineq:1bit_cs_se_smooth_diff_cov_Z_vee_W}
		\pnorm{\d_\epsilon \Sigma_{\mathfrak{Z}}^{[t]}}{\op}\vee \pnorm{\d_\epsilon \Sigma_{\mathfrak{W}}^{[t]}}{\op}\leq (K\Lambda L_\mu\sigma_{\mu_\ast}^{-1})^{c_t}\cdot \big(\epsilon+\mathcal{P}_{\texttt{SE}}^{(t-1)}(\epsilon)\big)^{1/c_t}.
		\end{align}

		\noindent (4). Similar to (\ref{ineq:rho_tau_bound_Upsilon}), for $k \in [m]$, with $\bm{\Upsilon}_{\epsilon;k}^{';[t]}\equiv \big(\partial_{(s)}\Upsilon_{\epsilon;r;k}(z^{([0:r])})\big)_{r,s \in [1:t]}$, we have with $\bm{L}_{\epsilon;k}^{[t]}(z^{([0:t])})\equiv \mathrm{diag}\big(\big\{-\eta  \cdot \partial_{11} \mathsf{L}_{s-1}\big(\Theta_{\epsilon;s;k}(z^{([0:s])}) ,\varphi_\epsilon(z^{(0)},\xi_k)\big)\big\}_{s \in [1:t]}\big)$, 
		\begin{align*}
		\bm{\Upsilon}_{\epsilon;k}^{';[t]}(z^{([0:t])}) =  \bm{L}^{[t]}_{\epsilon;k}(z^{([0:t])}) + \bm{L}^{[t]}_{\epsilon;k}(z^{([0:t])}) \mathfrak{O}_{t}(\bm{\rho}_\epsilon^{[t-1]}) \bm{\Upsilon}_{\epsilon;k}^{';[t]}(z^{([0:t])}).
		\end{align*}
		Solving for $\bm{\Upsilon}_{\epsilon;k}^{';[t]}$ yields that
		\begin{align}\label{ineq:1bit_cs_se_smooth_diff_rho_tau_1}
		\bm{\Upsilon}_{\epsilon;k}^{';[t]} (z^{([0:t])}) = \big[I_t- \bm{L}^{[t]}_{\epsilon;k}(z^{([0:t])}) \mathfrak{O}_{t}(\bm{\rho}_\epsilon^{[t-1]})\big]^{-1} \bm{L}^{[t]}_{\epsilon;k}(z^{([0:t])}).
		\end{align}
		Using that the matrix $\bm{L}^{[t]}_{\epsilon;k}(z^{([0:t])}) \mathfrak{O}_{t}(\bm{\rho}_\epsilon^{[t-1]})$ is lower triangular, and therefore $\big(\bm{L}^{[t]}_{\epsilon;k}(z^{([0:t])}) \mathfrak{O}_{t}(\bm{\rho}_\epsilon^{[t-1]})\big)^t=0_{t\times t}$, by the apriori estimates in Lemma \ref{lem:1bit_cs_apriori_est}, we have  
		\begin{align*}
		\bigpnorm{\big[I_t- \bm{L}^{[t]}_{\epsilon;k}(z^{([0:t])}) \mathfrak{O}_{t}(\bm{\rho}_\epsilon^{[t-1]})\big]^{-1}}{\op}\leq 1+\sum_{r \in [1:t]} \pnorm{  \bm{L}^{[t]}_{\epsilon;k}(z^{([0:t])}) \mathfrak{O}_{t}(\bm{\rho}_\epsilon^{[t-1]}) }{}^r \leq (K\Lambda)^{c_t}.
		\end{align*}
		Therefore, uniformly in $\epsilon>0$,
		\begin{align}\label{ineq:1bit_cs_se_smooth_diff_rho_tau_2}
		\pnorm{\d_\epsilon \bm{\tau}^{[t]}}{\op} 
		&\overset{(\ref{ineq:1bit_cs_se_smooth_diff_anti_conc})}{\leq}   (K\Lambda L_\mu \sigma_{\mu_\ast}^{-1})^{c_t} \cdot \big( \pnorm{\d_\epsilon \Sigma_{\mathfrak{Z}}^{[t]}}{\op}^{1/2}+   \pnorm{\d_\epsilon \bm{\rho}^{[t-1]}}{\op} + \epsilon^{1/2}  \big)\nonumber\\
		&\overset{(\ref{ineq:1bit_cs_se_smooth_diff_cov_Z_vee_W})}{\leq}  (K\Lambda L_\mu \sigma_{\mu_\ast}^{-1})^{c_t} \cdot \big(\epsilon^{1/2} +  \mathcal{P}_{\texttt{SE}}^{(t-1)}(\epsilon)\big)^{1/c_t}.
		\end{align}
		Next, similar to (\ref{ineq:rho_tau_bound_Omega}), for any $\ell \in [n]$, with $\bm{\Omega}_{\epsilon;\ell}^{';[t]}\equiv \big(\partial_{(s)}\Omega_{\epsilon;r;\ell}(w^{([1:r])})\big)_{r,s \in [1:t]}$, we have with $\bm{P}^{[t]}_{\epsilon;\ell}(w^{([1:t])})\equiv \mathrm{diag}\big(\big\{\mathsf{P}^{'}_{s;\ell}(\Delta_{\epsilon;s;\ell} (w^{([1:t])}) )\big\}_{s \in [1:t]}\big)$, 
		\begin{align}\label{ineq:1bit_cs_se_smooth_diff_rho_tau_3}
		\bm{\Omega}_{\epsilon;\ell}^{';[t]}(w^{([1:t])}) = \bm{P}^{[t]}_{\epsilon;\ell}(w^{([1:t])})\big[I_t+(\bm{\tau}_\epsilon^{[t]}+I_t)\mathfrak{O}_t\big(\bm{\Omega}_{\epsilon;\ell}^{';[t-1]}(w^{([1:t-1])}) \big)\big].
		\end{align}
		From the above display, it is easy to deduce with the apriori estimates in Lemma \ref{lem:1bit_cs_apriori_est} that uniformly in $\epsilon>0$,
		\begin{align*}
		\pnorm{\bm{\Omega}_{\epsilon;\ell}^{';[t]}(w^{([1:t])}) }{\op}\leq (K\Lambda)^{c_t}.
		\end{align*}
		This means, with the apriori estimates in Lemma \ref{lem:1bit_cs_apriori_est},
		\begin{align}\label{ineq:1bit_cs_se_smooth_diff_rho_tau_4}
		\pnorm{\d_\epsilon \bm{\Omega}_{\ell}^{';[t]}(w^{([1:t])})}{\op}&\leq (K\Lambda)^{c_t}\cdot \Big(\pnorm{\d_\epsilon \bm{P}^{[t]}_{\ell}(w^{([1:t])})}{\op}\nonumber\\
		&\qquad +\pnorm{\d_\epsilon \bm{\tau}^{[t]}}{\op}+\pnorm{\d_\epsilon \bm{\Omega}_{\ell}^{';[t-1]}(w^{([1:t-1])})}{\op}\Big).
		\end{align}
		In order to control $\pnorm{\d_\epsilon \bm{P}^{[t]}_{\ell}(w^{([1:t])})}{\op}$ in the above display, it suffices to control $\max_{\ell \in [n]} \abs{ \d_\epsilon \Delta_{t;\ell} (w^{([1:t])}) }$. To this end, using the definition (\ref{def:Delta_fcn}) and the apriori estimates in Lemma \ref{lem:1bit_cs_apriori_est}, for any $\ell \in [n]$,
		\begin{align*}
		\abs{\d_\epsilon \Delta_{t;\ell}} &= \abs{\Delta_{\epsilon;t;\ell}(w^{([1:t])} )  - \Delta_{0;t;\ell}(w^{([1:t])} )} \\
		& \leq (K\Lambda L_\mu\sigma_{\mu_\ast}^{-1})^{c_t}\cdot \big[ \pnorm{\d_\epsilon \bm{\tau}^{[t]}}{\op} \cdot \big(1+\pnorm{w^{([1:t-1])} }{}\big)+ \max_{r\in[0:t-1]}  	\abs{\d_\epsilon \Delta_{r;\ell}} + \pnorm{\d_\epsilon \bm{\delta}^{[t]} }{} \big].
		\end{align*}
		Iterating the estimate and using the initial condition $\abs{\d_\epsilon \Delta_{0;\ell} }= 0$, we have
		\begin{align*}
		&\pnorm{\d_\epsilon \bm{P}^{[t]}_{\ell}(w^{([1:t])})}{\op}\leq \Lambda\cdot \max_{\ell \in [n]}\abs{\d_\epsilon \Delta_{t;\ell} }\\
		&\leq (K\Lambda L_\mu\sigma_{\mu_\ast}^{-1})^{c_t}\cdot  \Big[\pnorm{\d_\epsilon \bm{\tau}^{[t]}}{\op}\cdot \big(1+\pnorm{w^{([1:t-1])} }{}\big) +\pnorm{\d_\epsilon \bm{\delta}^{[t]} }{}\Big].
		\end{align*}
		Combined with (\ref{ineq:1bit_cs_se_smooth_diff_rho_tau_4}),
		\begin{align*}
		\pnorm{\d_\epsilon \bm{\Omega}_{\ell}^{';[t]}(w^{([1:t])})}{\op}&\leq  (K\Lambda L_\mu\sigma_{\mu_\ast}^{-1})^{c_t}\cdot \Big[\pnorm{\d_\epsilon \bm{\Omega}_{\ell}^{';[t-1]}(w^{([1:t-1])})}{\op}\nonumber\\
		&\qquad +\pnorm{\d_\epsilon \bm{\tau}^{[t]}}{\op}\cdot \big(1+\pnorm{w^{([1:t-1])} }{}\big)+ \pnorm{\d_\epsilon \bm{\delta}^{[t]} }{}\Big].
		\end{align*}
		Iterating the above estimate and using the initial condition $\pnorm{\d_\epsilon \bm{\Omega}_{\ell}^{';[1]}(w^{(1)})}{\op}\leq (KL_\mu) \cdot [\pnorm{\d_\epsilon \bm{\tau}^{[t]}}{\op}+ \pnorm{\d_\epsilon \bm{\delta}^{[t]} }{}]$, we arrive at
		\begin{align*}
		\pnorm{\d_\epsilon \bm{\Omega}_{\epsilon;\ell}^{';[t]}(w^{([1:t])})}{\op}&\leq  (K\Lambda L_\mu\sigma_{\mu_\ast}^{-1})^{c_t}\cdot \Big[ \pnorm{\d_\epsilon \bm{\tau}^{[t]}}{\op}\cdot \big(1+\pnorm{w^{([1:t-1])} }{}\big)+ \pnorm{\d_\epsilon \bm{\delta}^{[t]} }{}\Big].
		\end{align*}
		Consequently, using again the apriori estimates in Lemma \ref{lem:1bit_cs_apriori_est}, 
		\begin{align}\label{ineq:1bit_cs_se_smooth_diff_rho_tau_5}
		\pnorm{\d_\epsilon \bm{\rho}^{[t]}}{\op}&\leq  (K\Lambda L_\mu\sigma_{\mu_\ast}^{-1})^{c_t}\cdot \big[\pnorm{\d_\epsilon \bm{\tau}^{[t]}}{\op}+ \pnorm{\d_\epsilon \bm{\delta}^{[t]} }{} + \pnorm{\d_\epsilon \Sigma^{[t]}_{\mathfrak{W}} }{\op}^{1/2}\big].
		\end{align}
		Combining (\ref{ineq:1bit_cs_se_smooth_diff_cov_Z_vee_W}), (\ref{ineq:1bit_cs_se_smooth_diff_rho_tau_2}) and (\ref{ineq:1bit_cs_se_smooth_diff_rho_tau_5}),
		\begin{align}\label{ineq:1bit_cs_se_smooth_diff_rho_tau_6}
		\pnorm{\d_\epsilon \bm{\rho}^{[t]}}{\op}
		&\leq (K\Lambda L_\mu\sigma_{\mu_\ast}^{-1})^{c_t}\cdot \Big[\big(\epsilon+\mathcal{P}_{\texttt{SE}}^{(t-1)}(\epsilon)\big)^{1/c_t} + \pnorm{\d_\epsilon \bm{\delta}^{[t]} }{}\Big].
		\end{align}
		\noindent (5). Using the Gaussian integration-by-parts representation (\ref{def:delta_t_alternative}), by using the apriori estimates in Lemma \ref{lem:1bit_cs_apriori_est}, with some calculations,
		\begin{align*}
		\pnorm{\d_\epsilon \bm{\delta}^{[t]} }{}&\leq (K\Lambda L_\mu\sigma_{\mu_\ast}^{-1})^{c_t}\cdot \Big( \E^{(0),1/2} \abs{\d_\epsilon \Upsilon_{r;\pi_n} (\mathfrak{Z}^{([0:t])})  }^2+\pnorm{\d_\epsilon \Sigma_{\mathfrak{Z}}^{[t]}}{\op}^{1/2}+\pnorm{\d_\epsilon \bm{\tau}^{[t]}}{\op}\Big).
		\end{align*}
		Plugging the estimates  (\ref{ineq:1bit_cs_se_smooth_diff_Upsilon}) for $\abs{\d_\epsilon \Upsilon_{r;\pi_n}}^2$, (\ref{ineq:1bit_cs_se_smooth_diff_cov_Z_vee_W}) for $\pnorm{\d_\epsilon \Sigma_{\mathfrak{Z}}^{[t]}}{\op}^{1/2}$ and (\ref{ineq:1bit_cs_se_smooth_diff_rho_tau_2}) for $\pnorm{\d_\epsilon \bm{\tau}^{[t]}}{\op}$ into the above display,
		\begin{align}
		\pnorm{\d_\epsilon \bm{\delta}^{[t]} }{}&\leq (K\Lambda L_\mu\sigma_{\mu_\ast}^{-1})^{c_t}\cdot \big(\epsilon+ \mathcal{P}_{\texttt{SE}}^{(t-1)}(\epsilon)\big)^{1/c_t}.
		\end{align}
		Plugging the above estimate into (\ref{ineq:1bit_cs_se_smooth_diff_rho_tau_6}), $\pnorm{\d_\epsilon \bm{\rho}^{[t]}}{\op}$ can also be bounded by the right hand side of the above display. In view of the proven estimate (\ref{ineq:1bit_cs_se_smooth_diff_rho_tau_2}), for $t\geq 1$,
		\begin{align*}
		\mathcal{P}_{\texttt{SE}}^{(t)}(\epsilon)\leq (K\Lambda L_\mu\sigma_{\mu_\ast}^{-1})^{c_t}\cdot \big(\epsilon+ \mathcal{P}_{\texttt{SE}}^{(t-1)}(\epsilon)\big)^{1/c_t}.
		\end{align*}
		Iterating the bound and using the trivial initial condition $	\mathcal{P}_{\texttt{SE}}^{(0)}(\epsilon)\equiv 0$, we arrive at the desired estimate
		\begin{align*}
		\mathcal{P}_{\texttt{SE}}^{(t)}(\epsilon)\leq (K\Lambda L_\mu\sigma_{\mu_\ast}^{-1})^{c_t}\cdot \epsilon^{1/c_t}.
		\end{align*}
		The estimates for functions follow from (\ref{ineq:1bit_cs_se_smooth_diff_Upsilon}) and (\ref{ineq:1bit_cs_se_smooth_diff_Omega}).
	\end{proof}

\subsection{Stability of smoothed gradient descent iterates}

We now compare the smoothed data and the original data statistics. Let
\begin{align*}
\hat{\mathcal{P}}_{\texttt{DT}}^{(t)}(\epsilon) &\equiv \frac{\pnorm{\d_\epsilon Y}{}}{\sqrt{n}}+ \pnorm{\d_\epsilon \hat{\bm{\tau}}^{[t]}}{\op}+ \pnorm{\d_\epsilon \hat{\bm{\rho}}^{[t]}}{\op} + \max_{0\leq s\leq t}\Big(\abs{\d_\epsilon \mathscr{E}_{\mathsf{H}}^{(s)}}+\abs{\d_\epsilon \hat{\mathscr{E}}_{\mathsf{H} }^{(s)}}\Big)\\
&\quad + \max_{0\leq s\leq t} \frac{1}{\sqrt{n}} \Big( \pnorm{\d_\epsilon \mu^{(s)}}{} + \pnorm{\d_\epsilon Z^{(s)}}{}+\pnorm{\d_\epsilon W^{(s)}}{}+ \pnorm{\d_\epsilon \hat{Z}^{(s)}}{}\Big).
\end{align*}

We note the following useful apriori estimate.

	\begin{lemma}\label{lem:1bit_cs_apriori_est_data}
		Suppose (A1), (A3) and (A4'), and (\ref{ineq:111112bLound}) hold. Then the following hold for some $c_t=c_t(t)>1$:
		\begin{enumerate}
			\item $\sup_{\epsilon\geq 0}\big(\pnorm{\mu_\epsilon^{(t)}}{}/\sqrt{n}\big)\leq \big(K\Lambda (1+\pnorm{A}{\op})\big)^{c_t}$.
			\item $\sup_{\epsilon\geq 0}\big(\pnorm{\hat{\bm{\tau}}_{\epsilon}^{[t]}}{\op}+\pnorm{\hat{\bm{\rho}}_{\epsilon}^{[t]}}{\op}\big)\leq (K\Lambda)^{c_t}$.
			\item $\sup_{\epsilon\geq 0}\big(\pnorm{Z_\epsilon^{(t)}}{}+\pnorm{\hat{Z}_\epsilon^{(t)}}{}\big)/\sqrt{n}\leq \big(K\Lambda (1+\pnorm{A}{\op})\big)^{c_t}$.
			\item $\sup_{\epsilon\geq 0} \pnorm{\mu^{(t)}_{\epsilon;\mathrm{db}}}{}/\sqrt{n}\leq  \big(K\Lambda (1\wedge \tau_\ast^{(t)})^{-1} (1+\pnorm{A}{\op})\big)^{c_t}$.
		\end{enumerate}
	\end{lemma}
	\begin{proof}
		\noindent (1). By definition of gradient descent (\ref{def:grad_descent}), for any $\epsilon\geq0$,
		\begin{align*}
		\pnorm{\mu_\epsilon^{(t)}}{}&\leq \Lambda \pnorm{\mu_\epsilon^{(t-1)}}{}+\Lambda \pnorm{A}{\op}\big(\pnorm{A}{\op} \pnorm{\mu_\epsilon^{(t-1)}}{}+\pnorm{Y_\epsilon}{} \big)\\
		&\leq \Lambda \big(1+\pnorm{A}{\op}^2\big)\cdot \pnorm{\mu_\epsilon^{(t-1)}}{}+2\Lambda \pnorm{A}{\op} \sqrt{m}\leq \cdots \leq \big(\Lambda(1+\pnorm{A}{\op})\big)^{c_t} \sqrt{m}.
		\end{align*}	
	
		\noindent (2). As for any $k \in [m]$, 
		\begin{align*}
		\hat{\bm{L}}^{[t]}_{\epsilon;k} &\equiv \mathrm{diag}\Big(\Big\{-\eta \iprod{e_k}{\partial_{11}\mathsf{L}_{s-1}(A\mu_\epsilon^{(s-1)},Y_\epsilon)}\Big\}_{s \in [1:t]}\Big),
		\end{align*}
		for some universal $c_0 > 0$ whose value may change from line to line,
		\begin{align*}
		\pnorm{\hat{\bm{\tau}}_{\epsilon;k}^{[t]}  }{\op} \leq (K\Lambda)^{c_0} \cdot  \bigg(1 +  \sum_{r \in [1:t]} \pnorm{\hat{\bm{\rho}}_\epsilon^{[t-1]}}{\op}^r \bigg) \leq (K\Lambda)^{c_0}\cdot  \big(1+ \pnorm{\hat{\bm{\rho}}_\epsilon^{[t-1]}}{\op}\big)^t.
		\end{align*}
		By taking average, it holds for any $\epsilon\geq 0$ that
		\begin{align}\label{ineq:1bit_cs_apriori_est_data_1}
		\pnorm{\hat{\bm{\tau}}_\epsilon^{[t]} }{\op}\leq  (K\Lambda)^{c_0}\cdot \big(1+ \pnorm{\hat{\bm{\rho}}_\epsilon^{[t-1]}}{\op}\big)^t.
		\end{align}
		On the other hand, recall that for any $\ell \in [n]$, $
		\hat{\bm{P}}^{[t]}_{\epsilon;\ell}= \mathrm{diag}\big(\big\{\prox_{\eta \mathsf{f}}'\big(\iprod{e_\ell}{\mu_\epsilon^{(s-1)}-\eta A^\top \partial_1 \mathsf{L}_{s-1} (A \mu_\epsilon^{(s-1)},Y_\epsilon) }\big)\big\}_{s \in [1:t]}\big)$, 
		so $\pnorm{\hat{\bm{P}}^{[t]}_{\epsilon;\ell}}{\op}\leq \Lambda$. This means 
		\begin{align*}
		\pnorm{\hat{\bm{\rho}}_{\epsilon;\ell}^{[t]}}{\op}\leq  \Lambda \cdot \big( 1+(\pnorm{\hat{\bm{\tau}}_\epsilon^{[t]}}{\op}+1)\cdot \pnorm{\hat{\bm{\rho}}_{\epsilon;\ell}^{[t-1]}}{\op}\big).
		\end{align*}
		Iterating the estimate with trivial initial condition and taking average we have
		\begin{align}\label{ineq:1bit_cs_apriori_est_data_2}
		\pnorm{\hat{\bm{\rho}}_{\epsilon}^{[t]}}{\op}\leq   \Lambda^{c_0t}\cdot \big(\pnorm{\hat{\bm{\tau}}_\epsilon^{[t]}}{\op}+1\big)^t.
		\end{align} 
		Combining (\ref{ineq:1bit_cs_apriori_est_data_1}) and (\ref{ineq:1bit_cs_apriori_est_data_2}), and using the trivial initial condition $\pnorm{\hat{\bm{\tau}}_\epsilon^{[1]} }{\op}\leq (K\Lambda)^{c_0}$, we arrive at the desired estimates. 
		
		\noindent (3). Using the definition for $Z_\epsilon^{(t)}$ in (\ref{def:Z_W}),
		\begin{align*}
		\pnorm{Z_\epsilon^{(t)}}{}\leq \pnorm{A}{\op} \pnorm{\mu_\epsilon^{(t)}}{}+ t \Lambda^2 (1+\pnorm{A}{\op})^2\cdot  \pnorm{\bm{\rho}_\epsilon^{[t]}}{\op}\cdot \Big(\max_{s\in [0:t-1]} \pnorm{\mu_\epsilon^{(s)}}{}+\pnorm{Y_\epsilon}{}\Big).
		\end{align*}
		Using the apriori estimate in Lemma \ref{lem:1bit_cs_apriori_est} and the estimate in (1), we have 
		\begin{align*}
		\pnorm{Z_\epsilon^{(t)}}{}/\sqrt{n}\leq \big(K\Lambda (1+\pnorm{A}{\op})\big)^{c_t}. 
		\end{align*}
		On the other hand, using the definition for $\hat{Z}_\epsilon^{(t)}$ in (\ref{def:Z_hat}), we may use estimate in (2) to conclude the same bound as above.
		
		\noindent (4). Using the definition of $\mu^{(t)}_{\epsilon;\mathrm{db}}$ in (\ref{def:debias_gd_oracle}), we have
		\begin{align*}
		\pnorm{\mu^{(t)}_{\epsilon;\mathrm{db}}}{}&\leq \pnorm{\mu^{(t)}_{\epsilon} }{}+t \Lambda^2 (1+\pnorm{A}{\op})^2\cdot \pnorm{\bm{\omega}^{[t]}}{\op}\cdot \Big(\max_{s\in [0:t-1]} \pnorm{\mu_\epsilon^{(s)}}{}+\pnorm{Y_\epsilon}{}\Big).
		\end{align*}
		We may conclude now using Lemma \ref{lem:lower_tri_mat_inv}, the apriori estimate in Lemma \ref{lem:1bit_cs_apriori_est} and the estimate in (1).
	\end{proof}

	\begin{lemma}\label{lem:1bit_cs_data_smooth_diff}
		Suppose (A1), (A3) and (A4*), and (\ref{ineq:111112bLound}) hold. Further assume the conditions on $\mathsf{H}$ in Proposition \ref{prop:gen_error_linear}. Then there exists some $c_t=c_t(t)>1$ such that 
		\begin{align*}
		\hat{\mathcal{P}}_{\texttt{DT}}^{(t)}(\epsilon) \leq \big(K\Lambda L_\mu (1\wedge\sigma_{\mu_\ast})^{-1}(1+\pnorm{A}{\op})\big)^{c_t}\cdot \bigg(\epsilon^{1/c_t}+ \frac{\pnorm{\d_\epsilon \varphi (A\mu_\ast+\xi)}{} }{\sqrt{n}}\bigg).
		\end{align*}
		Moreover, recall $\tau_\ast^{(t)}$ defined in (\ref{def:tau_ast}), 
		\begin{align*}
		\frac{\pnorm{\d_\epsilon \mu^{(s)}_{\mathrm{db}} }{}}{\sqrt{n}}\leq \bigg(\frac{K\Lambda L_\mu }{(1\wedge \sigma_{\mu_\ast})(1\wedge  \tau_\ast^{(t)}) }(1+\pnorm{A}{\op})\bigg)^{c_t}\cdot \bigg(\epsilon^{1/c_t}+ \frac{\pnorm{\d_\epsilon \varphi (A\mu_\ast+\xi)}{} }{\sqrt{n}}\bigg).
		\end{align*}
	\end{lemma}

\begin{proof}
	   We assume for notational simplicity that $\sigma_{\mu_\ast}\vee \tau_\ast^{(t)}\leq 1$. 
	   
		\noindent (1). It is easy to see that $\pnorm{\d_\epsilon Y}{} = \pnorm{\d_\epsilon \varphi (A\mu_\ast+\xi)}{}$.
		
		\noindent (2). By definition of gradient descent (\ref{def:grad_descent_1bit_cs}),	
		\begin{align}\label{ineq:1bit_cs_data_smooth_diff_gd}
		\pnorm{\d_\epsilon \mu^{(t)}}{}&\leq \pnorm{\d_\epsilon \mu^{(t-1)}}{}+\Lambda^2  \pnorm{A}{\op}\cdot \big(\pnorm{A}{\op} \pnorm{\d_\epsilon \mu^{(t-1)}}{}+ \pnorm{\d_\epsilon Y}{}\big)\nonumber\\
		& \leq  \Lambda^2 \big(1+\pnorm{A}{\op}^2\big)\cdot \pnorm{\d_\epsilon \mu^{(t-1)}}{}+ \Lambda^2 \pnorm{A}{\op}\cdot \pnorm{\d_\epsilon Y}{}\nonumber\\
		&\leq \cdots \leq \big(\Lambda(1+ \pnorm{A}{\op})\big)^{c_t}\cdot \pnorm{\d_\epsilon Y}{}.
		\end{align}

		\noindent (3). By Algorithm \ref{def:alg_tau_rho}, as for any $k \in [m]$,
		\begin{align*}
		\hat{\bm{L}}^{[t]}_{\epsilon;k} &\equiv \mathrm{diag}\Big(\Big\{-\eta \iprod{e_k}{\partial_{11}\mathsf{L}_{s-1}(A\mu_\epsilon^{(s-1)},Y_\epsilon)}\Big\}_{s \in [1:t]}\Big),
		\end{align*} 
		using the apriori estimates in Lemma \ref{lem:1bit_cs_apriori_est_data} and then taking average, we have 
		\begin{align}\label{ineq:1bit_cs_data_smooth_diff_tau_rho_hat_1}
		\pnorm{\d_\epsilon \hat{\bm{\tau}}^{[t]}}{\op} \leq \big(K\Lambda (1+\pnorm{A}{\op})\big)^{c_t}\cdot \bigg( \pnorm{\d_\epsilon \hat{\bm{\rho}}^{[t-1]}}{\op} +  \frac{\pnorm{\d_\epsilon Y}{}}{\sqrt{n}}  \bigg).
		\end{align}
		On the other hand, using Lemma \ref{lem:1bit_cs_apriori_est_data}, for $r,s \in [t]$,
		\begin{align*}
		\E_{\pi_n}\abs{(\d_\epsilon \hat{\bm{\rho}}_{\pi_n}^{[t]})_{r,s}}&\leq \E_{\pi_n} \abs{e_r^\top (\hat{\bm{P}}_{\epsilon;\pi_n}^{[t]}-\hat{\bm{P}}_{\pi_n}^{[t]})\big[I_t+(\hat{\bm{\tau}}_\epsilon^{[t]}+I_t)\mathfrak{O}_t(\hat{\bm{\rho}}_{\epsilon;\pi_n}^{[t-1]})\big] e_s}\\
		&\qquad + \Lambda \cdot \E_{\pi_n} \pnorm{ \d_\epsilon \hat{\bm{\tau}}^{[t]} \mathfrak{O}_t(\hat{\bm{\rho}}_{\epsilon;\pi_n}^{[t-1]}) }{\op} + \E_{\pi_n} \abs{e_r^\top \hat{\bm{P}}_{\pi_n}^{[t]} \hat{\bm{\tau}}^{[t]} \mathfrak{O}_t(\d_\epsilon \hat{\bm{\rho}}_{\pi_n}^{[t-1]}) e_s}\\
		&\leq (K\Lambda)^{c_t}\cdot \Big(\E_{\pi_n}^{1/2} \abs{e_r^\top (\hat{\bm{P}}_{\epsilon;\pi_n}^{[t]}-\hat{\bm{P}}_{\pi_n}^{[t]}) e_r}^2+\pnorm{\d_\epsilon \hat{\bm{\tau}}^{[t]}}{\op}+\E_{\pi_n} \pnorm{\d_\epsilon \hat{\bm{\rho}}_{\pi_n}^{[t-1]}}{\op}\Big).
		\end{align*}
		Using that
		\begin{align*}
		\E_{\pi_n}^{1/2} \abs{e_r^\top (\hat{\bm{P}}_{\epsilon;\pi_n}^{[t]}-\hat{\bm{P}}_{\pi_n}^{[t]})e_r}^2&\leq \big(K\Lambda(1+\pnorm{A}{\op})\big)^{c_t}\cdot \bigg(\max_{s \in [1:t-1]}\frac{\pnorm{\d_\epsilon\mu^{(s)}}{} }{\sqrt{n}}+ \frac{\pnorm{\d_\epsilon Y}{}}{\sqrt{n}}\bigg)\\
		& \leq \big(K\Lambda(1+\pnorm{A}{\op})\big)^{c_t}\cdot \frac{\pnorm{\d_\epsilon Y}{}}{\sqrt{n}},
		\end{align*}
		we arrive at the estimate 
		\begin{align*}
		\E_{\pi_n}\abs{(\d_\epsilon \hat{\bm{\rho}}_{\pi_n}^{[t]})_{r,s}}
		&\leq \big(K\Lambda(1+\pnorm{A}{\op})\big)^{c_t}\cdot \bigg(\frac{\pnorm{\d_\epsilon Y}{}}{\sqrt{n}}+\pnorm{\d_\epsilon \hat{\bm{\tau}}^{[t]}}{\op}+\E_{\pi_n} \pnorm{\d_\epsilon \hat{\bm{\rho}}_{\pi_n}^{[t-1]}}{\op}\bigg).
		\end{align*}
		Using the trivial estimate $\E_{\pi_n} \pnorm{\d_\epsilon \hat{\bm{\rho}}_{\pi_n}^{[t]}}{\op}\leq t^3\cdot \max_{r,s\in [t]} \E_{\pi_n}\abs{(\d_\epsilon \hat{\bm{\rho}}_{\pi_n}^{[t]})_{r,s}}$, we may then iterate the above bound until the initial condition $\E_{\pi_n}\pnorm{\d_\epsilon \hat{\bm{\rho}}_{\pi_n}^{[1]}}{\op} \leq \Lambda^2\cdot \pnorm{A}{\op} \pnorm{\d_\epsilon Y}{}/\sqrt{n} $, so that
		\begin{align}\label{ineq:1bit_cs_data_smooth_diff_tau_rho_hat_2}
		\pnorm{\d_\epsilon \hat{\bm{\rho}}^{[t]}}{\op}\leq \E_{\pi_n} \pnorm{\d_\epsilon \hat{\bm{\rho}}_{\pi_n}^{[t]}}{\op}\leq \big(K\Lambda(1+\pnorm{A}{\op})\big)^{c_t}\cdot \bigg(\frac{\pnorm{\d_\epsilon Y}{}}{\sqrt{n}}+\pnorm{\d_\epsilon \hat{\bm{\tau}}^{[t]}}{\op}\bigg).
		\end{align}
		Combining (\ref{ineq:1bit_cs_data_smooth_diff_tau_rho_hat_1}) and (\ref{ineq:1bit_cs_data_smooth_diff_tau_rho_hat_2}), we may iterate the estimate until the initial condition $\pnorm{\d_\epsilon\hat{\bm{\tau}}^{[1]}}{\op}\leq (K\Lambda) \cdot \pnorm{\d_\epsilon Y}{}/\sqrt{n} $, so that
		\begin{align}\label{ineq:1bit_cs_data_smooth_diff_tau_rho_hat_3}
		\pnorm{\d_\epsilon \hat{\bm{\tau}}^{[t]}}{\op}+ \pnorm{\d_\epsilon \hat{\bm{\rho}}^{[t]}}{\op}\leq \big(K\Lambda(1+\pnorm{A}{\op})\big)^{c_t}\cdot \frac{\pnorm{\d_\epsilon Y}{}}{\sqrt{n}}.
		\end{align}
		
		\noindent (4). Using the definition for $Z^{(s)}$ in (\ref{def:Z_W}), we have
		\begin{align*}
		\frac{\pnorm{\d_\epsilon Z^{(t)}}{}}{\sqrt{n}}&\leq \pnorm{A}{\op}\cdot \frac{\pnorm{\d_\epsilon \mu^{(t)}}{}}{\sqrt{n}}+\big(K\Lambda L_\mu (1+\pnorm{A}{\op})\big)^{c_0 t}\cdot \pnorm{\d_\epsilon\bm{\rho}^{[t]} }{\op}\\
		&\qquad + t\pnorm{ \bm{\rho}^{[t]}}{\op}\cdot \bigg(\pnorm{A}{\op}\cdot \max_{s \in [1:t-1]} \frac{\pnorm{\d_\epsilon \mu^{(s)}}{}}{\sqrt{n}}+\frac{\pnorm{\d_\epsilon Y}{}}{\sqrt{n}} \bigg).
		\end{align*}
		Using Lemma \ref{lem:rho_tau_bound} for $\pnorm{ \bm{\rho}^{[t]}}{\op}$ and (\ref{ineq:1bit_cs_data_smooth_diff_gd}) for $\pnorm{\d_\epsilon \mu^{(\cdot)}}{}$, 
		\begin{align}\label{ineq:1bit_cs_data_smooth_diff_Z}
		\frac{\pnorm{\d_\epsilon Z^{(t)}}{}}{n}\leq \big(K\Lambda L_\mu \sigma_{\mu_\ast}^{-1} (1+\pnorm{A}{\op})\big)^{c_t}\cdot \bigg(\mathcal{P}_{\texttt{SE}}^{(t)}(\epsilon)+ \frac{\pnorm{\d_\epsilon Y}{}}{\sqrt{n}}\bigg).
		\end{align}
		The estimate for $\pnorm{\d_\epsilon \hat{Z}^{(t)}}{}$ is the same as above, upon using Lemma \ref{lem:1bit_cs_apriori_est_data} for $\pnorm{ \hat{\bm{\rho}}^{[t]}}{\op}$ and (\ref{ineq:1bit_cs_data_smooth_diff_tau_rho_hat_3}) for $\pnorm{\d_\epsilon \hat{\bm{\rho}}^{[t]} }{\op}$.
		
		Using the definition for $W^{(t)}$ in (\ref{def:Z_W}) and similar calculations as above, we have
		\begin{align}\label{ineq:1bit_cs_data_smooth_diff_W}
		\frac{\pnorm{\d_\epsilon W^{(t)}}{}}{\sqrt{n}}
		&  \leq \big(K\Lambda L_\mu\sigma_{\mu_\ast}^{-1} (1+\pnorm{A}{\op})\big)^{c_t}\cdot \bigg(\mathcal{P}_{\texttt{SE}}^{(t)}(\epsilon)+ \frac{\pnorm{\d_\epsilon Y}{}}{\sqrt{n}}\bigg).
		\end{align}
	\noindent (5). For notation convenience, we shall omit $\mathsf{H}$ in the subscript of $\mathscr{E}$. Then using the definition in (\ref{def:gen_err}) and the apriori estimates in Lemma \ref{lem:1bit_cs_apriori_est_data},
	\begin{align*}
	\abs{\d_\epsilon \mathscr{E}^{(t)}}&\leq \Lambda^{c_0}\cdot  \E^{1/2} \big[\big(\iprod{A_{\mathrm{new}}}{\d_\epsilon \mu^{(t)}}-\d_\epsilon\varphi( \iprod{A_{\mathrm{new}}}{\mu_\ast}+\xi_{\pi_m})\big)^2|(A,Y)\big]\\
	&\qquad \times \Big(1+ \E^{1/2} \big[\big(\iprod{A_{\mathrm{new}}}{ \mu^{(t)}+\mu^{(t)}_\epsilon}\big)^4|(A,Y)\big]\Big)\\
	&\leq \big(K\Lambda (1+\pnorm{A}{\op})\big)^{c_t}\cdot \bigg(\frac{\pnorm{\d_\epsilon \mu^{(t)}}{}}{\sqrt{n}}+\sup_{x \in \R}\Prob^{1/2}\big(\abs{\sigma_{\mu_\ast}\mathsf{Z}-x}\leq \epsilon\big) \bigg).
	\end{align*}
	Using (\ref{ineq:1bit_cs_data_smooth_diff_gd}), we now arrive at the estimate
	\begin{align}\label{ineq:1bit_cs_data_smooth_diff_gen_err}
	\abs{\d_\epsilon \mathscr{E}^{(t)}}\leq \big(K\Lambda \sigma_{\mu_\ast}^{-1} (1+\pnorm{A}{\op})\big)^{c_t}\cdot \bigg(\epsilon^{1/2} + \frac{\pnorm{\d_\epsilon Y}{}}{\sqrt{n}}\bigg).
	\end{align}
	Next, using the definition in (\ref{def:gen_err_est}) and the apriori estimates in Lemma \ref{lem:1bit_cs_apriori_est_data},
	\begin{align*}
	\abs{\d_\epsilon \hat{\mathscr{E}}^{(t)}}&\leq \Lambda^{c_0}\cdot \E_{\pi_m} \abs{\d_\epsilon \hat{Z}_{\pi_m}^{(t)}-\d_\epsilon Y_{\pi_m} }\cdot \big(\abs{\hat{Z}_{\pi_m}^{(t)} }+\abs{\hat{Z}_{\epsilon;\pi_m}^{(t)} }+ \abs{Y_{\pi_m}}+\abs{Y_{\epsilon;\pi_m}}\big)^2\\
	&\leq \big(K\Lambda (1+\pnorm{A}{\op})\big)^{c_t}\cdot \bigg(\frac{\pnorm{\d_\epsilon \hat{Z}^{(t)}}{}}{\sqrt{n}}+ \frac{\pnorm{\d_\epsilon Y}{}}{\sqrt{n}}\bigg).
	\end{align*}
	In view of the argument below (\ref{ineq:1bit_cs_data_smooth_diff_Z}), we have 
	\begin{align}\label{ineq:1bit_cs_data_smooth_diff_gen_err_est}
	\abs{\d_\epsilon \hat{\mathscr{E}}^{(t)}}
	\leq \big(K\Lambda (1+\pnorm{A}{\op})\big)^{c_t}\cdot \bigg(\mathcal{P}_{\texttt{SE}}^{(t)}(\epsilon)+ \frac{\pnorm{\d_\epsilon Y}{}}{\sqrt{n}}\bigg).
	\end{align}
	
	\noindent (6). Using the definition in (\ref{def:debias_gd_oracle}),
	\begin{align*}
	\frac{\pnorm{\d_\epsilon \mu^{(s)}_{\mathrm{db}}}{}}{\sqrt{n}}\leq \big(K\Lambda (1+\pnorm{A}{\op})\big)^{c_t}\cdot \bigg(\pnorm{\d_\epsilon \bm{\omega}^{[t]}}{}+\frac{\pnorm{\d_\epsilon Y}{}}{\sqrt{n}}\bigg).
	\end{align*}
	Using the estimate Lemma \ref{lem:1bit_cs_se_smooth_diff} and Lemma \ref{lem:lower_tri_mat_inv}, we have for $\pnorm{\d_\epsilon \bm{\tau}^{[t]}}{}\leq \tau_\ast^{(t)}/2$, 
	\begin{align*}
	\pnorm{\d_\epsilon \bm{\omega}^{[t]}}{}\leq (\tau_\ast^{(t)})^{-c_t} \pnorm{\d_\epsilon \bm{\tau}^{[t]}}{}\leq \big(K\Lambda L_\mu/(\sigma_{\mu_\ast}\tau_\ast^{(t)})\big)^{c_t}\cdot \epsilon^{1/c_t}.
	\end{align*}
	The bound is trivial for $\pnorm{\d_\epsilon \bm{\tau}^{[t]}}{}> \tau_\ast^{(t)}/2$.
	
	The proof is now complete by collecting all estimates and using Lemma \ref{lem:1bit_cs_se_smooth_diff} to replace $\mathcal{P}_{\texttt{SE}}^{(t)}(\epsilon)$ with an explicit estimate.
\end{proof}

\subsection{Proof of Theorem \ref{thm:1bit_cs}}

	For notational convenience, we assume $\sigma_{\mu_\ast}\leq 1$.
	
	\noindent (1). We only prove the averaged distributional characterizations for $\big(\{(A \mu^{(s-1)}), {Z}^{(s-1)}\}, (A \mu_\ast)\big)$; the proof for the other one $\big(\{\mu^{(s)},W^{(s)}\},\mu_\ast\big)$ is completely analogous. Note that  Theorem \ref{thm:gd_se}-(2) applies to the smoothed gradient descent iterates: for any $\epsilon>0$,
	\begin{align}\label{ineq:1bit_cs_se_1}
	I_0(\epsilon)&\equiv \E^{(0)} \biggabs{\frac{1}{m}\sum_{k \in [m]}\Big( \psi_{k}\big(\big\{(A \mu_\epsilon^{(s-1)})_{k}, {Z}_{\epsilon;k}^{(s-1)}\big\}, (A \mu_\ast)_{k}\big)\nonumber\\
		&\qquad\qquad\qquad - \E^{(0)}\psi_{k}\big(\big\{\Theta_{\epsilon;s;k}(\mathfrak{Z}_\epsilon^{([0:s])}), \mathfrak{Z}_\epsilon^{(s)}\big\}, \mathfrak{Z}_\epsilon^{(0)}\big) \Big) }^{q}\nonumber\\
	& \leq \big(\epsilon^{-1}\cdot K\Lambda \Lambda_\psi L_\mu \big)^{c_t} \cdot n^{-1/c_t}.
	\end{align}
	Now we shall control the errors incurred by smoothing. First, using Lemmas \ref{lem:1bit_cs_apriori_est} and \ref{lem:1bit_cs_se_smooth_diff}, the smoothed state evolution parameters are stable:
	\begin{align}\label{ineq:1bit_cs_se_2}
	I_1(\epsilon)&\equiv  \biggabs{\frac{1}{m}\sum_{k \in [m]}\Big( \E^{(0)}\psi_{k}\big(\big\{\Theta_{\epsilon;s;k}(\mathfrak{Z}_\epsilon^{([0:s])}), \mathfrak{Z}_\epsilon^{(s)}\big\}, \mathfrak{Z}_\epsilon^{(0)}\big)- \E^{(0)}\psi_{k}\big(\big\{\Theta_{s;k}(\mathfrak{Z}^{([0:s])}), \mathfrak{Z}^{(s)}\big\}, \mathfrak{Z}^{(0)}\big) \Big) }^{q}\nonumber\\
	&\leq \Lambda_\psi (K\Lambda L_\mu \sigma_{\mu_\ast}^{-1})^{c_t}\cdot \max_{k \in [m]}\max_{0\leq s\leq t}\E^{(0)}\Big(\abs{\d_\epsilon \Theta_{s;k} (\mathfrak{Z}_\epsilon^{([0:s])}) }+ \pnorm{\d_\epsilon \mathfrak{Z}^{([0:s])} }{}  \Big)\nonumber\\
	&\leq \Lambda_\psi (K\Lambda L_\mu \sigma_{\mu_\ast}^{-1})^{c_t}\cdot \epsilon^{1/c_t}.
	\end{align}
	Next, using Lemmas \ref{lem:1bit_cs_apriori_est_data} and \ref{lem:1bit_cs_data_smooth_diff}, the smoothed gradient descent iterates are stable:
	\begin{align}\label{ineq:1bit_cs_se_3}
	I_2(\epsilon)&\equiv \E^{(0)} \biggabs{\frac{1}{m}\sum_{k \in [m]}\Big( \psi_{k}\big(\big\{(A \mu^{(s-1)})_{k}, {Z}_{k}^{(s-1)}\big\}, (A \mu_\ast)_{k}\big)- \psi_{k}\big(\big\{(A \mu_\epsilon^{(s-1)})_{k}, {Z}_{\epsilon;k}^{(s-1)}\big\}, (A \mu_\ast)_{k}\big) \Big) }^{q}\nonumber\\
	&\leq \Lambda_\psi (K\Lambda)^{c_t}\cdot \max_{s \in [1:t]}\E^{(0), 1/2} \bigg(\frac{\pnorm{A \d_\epsilon \mu^{(s-1)}}{}}{\sqrt{n}}+ \frac{\pnorm{\d_\epsilon Z^{(s-1)}}{}}{\sqrt{n}}\bigg)^{2q} \nonumber\\
	&\qquad\qquad \times \max_{s \in [1:t]} \max_{\alpha=0,\epsilon}\E^{(0),1/2}\bigg( \frac{\pnorm{A \mu_\alpha^{(s-1)}}{}}{\sqrt{n}}+ \frac{\pnorm{Z_\alpha^{(s-1)}}{}}{\sqrt{n}}+\frac{\pnorm{A \mu_\ast}{}}{\sqrt{n}}\bigg)^{2q}\nonumber\\
	&\leq (K\Lambda L_\mu \sigma_{\mu_\ast}^{-1})^{c_t}\cdot \epsilon^{1/c_t}.
	\end{align}
	Combining (\ref{ineq:1bit_cs_se_1})-(\ref{ineq:1bit_cs_se_3}), we then have 
	\begin{align*}
	&\E^{(0)} \biggabs{\frac{1}{m}\sum_{k \in [m]}\Big( \psi_{k}\big(\big\{(A \mu^{(s-1)})_{k}, {Z}_{k}^{(s-1)}\big\}, (A \mu_\ast)_{k}\big)- \E^{(0)}\psi_{k}\big(\big\{\Theta_{s;k}(\mathfrak{Z}^{([0:s])}), \mathfrak{Z}^{(s)}\big\}, \mathfrak{Z}^{(0)}\big) \Big) }^{q}\\
	&\lesssim_q I_0(\epsilon)+I_1(\epsilon)+I_2(\epsilon)\leq \big( K\Lambda \Lambda_\psi L_\mu\sigma_{\mu_\ast}^{-1}\big)^{c_t}\cdot \big(\epsilon^{-c_t}n^{-1/c_t}+ \epsilon^{1/c_t}\big).
	\end{align*}
	Optimizing $\epsilon>0$ to conclude. 
	
	\noindent (2). The claim follows by a similar argument as above by invoking Theorem \ref{thm:consist_tau_rho}, Lemma \ref{lem:1bit_cs_se_smooth_diff} and Lemma \ref{lem:1bit_cs_data_smooth_diff}. \qed

\subsection{Verification of (\ref{ineq:111112bLound}) for logistic regression}\label{subsubsection:verification_logistic}

	We now verify that the logistic regression model satisfies the boundedness condition in (\ref{ineq:111112bLound}). Recall $\mathsf{L}(x,y)=\rho(-xy)=\log(1+e^{-xy})$. As the role of $x,y$ is symmetric, we only need to check
\begin{itemize}
	\item $\partial_{11}\mathsf{L}(x,y) = \frac{y^2 e^{xy}}{(1+e^{xy})^2}$, $\partial_{12}\mathsf{L}(x,y)  = \frac{e^{xy}(xy-1)-1}{(e^{xy}+1)^2}$,
	\item $\partial_{111}\mathsf{L}(x,y) = -\frac{y^3 e^{xy}(e^{xy}-1)}{(e^{xy}+1)^3}$, $\partial_{112}\mathsf{L}(x,y)=-\frac{ye^{xy}(-xy+e^{xy}(xy-2)-2)}{(e^{xy}+1)^3}$.
\end{itemize}
Consequently, the quantity of interest $\sup_{x \in \R, y \in [-1,1]}\max_{\abs{\alpha}=2, 3}\abs{\partial_\alpha\mathsf{L}(x,y)} $ can be bounded, by an absolute constant multiple of
\begin{align*}
1+\sup_{u \in \R} \frac{(1+\abs{u}^3) e^u}{(e^u+1)^2}<\infty.
\end{align*}
This verifies (\ref{ineq:111112bLound}).

\subsection{Proof of Lemma \ref{lem:1bit_cs_delta_general_loss}}

	Let $\{\varphi_\epsilon\}_{\epsilon>0}$ be smooth approximations of $\sign(\cdot)$ as before. Note that $\Theta_{\epsilon;t;k}(z_{[0:t]})$ depends on $z_0$ only through $\varphi_\epsilon(z_0,\xi_k)$. More precisely, with $\Theta_t^\circ$ defined in (\ref{def:Theta_circ}), we have $\Theta_{\epsilon;t;k}(z_{[0:t]}) = \Theta_t^\circ\big(\varphi_\epsilon(z_0,\xi_k),z_1,\ldots,z_t\big)$. Moreover, for $a=1,2,$ we further let $\mathfrak{G}^{\circ,(a)}_{\partial_1\mathsf{L};t}:\R^{[0:t]}\to \R$ defined via
	\begin{align*}
	\mathfrak{G}^{\circ,(a)}_{\partial_1\mathsf{L};t}(z_0,z_{[1:t]}) &\equiv \partial_{1a} \mathsf{L}\big( \Theta_t^\circ(z_0,z_{[1:t]}),z_0   \big). 
	\end{align*}
	Then for $a=1,2$,
	\begin{align*}
	\partial_{1a} \mathsf{L} \big( \Theta_{\epsilon;t;k}(z_{[0:t]}), \varphi_\epsilon(z_0+\xi_k)  \big)= \mathfrak{G}^{\circ,(a)}_{\partial_1\mathsf{L};t}\big(\varphi_\epsilon(z_0+\xi_k),z_{[1:t]}\big).
	\end{align*}
	Now taking derivative on (S1) with respect to $z_0$,
	\begin{align*}
	\partial_{(0)} \Upsilon_{\epsilon;t;k}(z_{[0:t]}) &= -\eta\cdot \mathfrak{G}^{\circ,(1)}_{\partial_1\mathsf{L};t}\big(\varphi_\epsilon(z_0+\xi_k),z_{[1:t]}\big)\cdot \bigg(\sum_{r \in [1:t-1]}\rho_{\epsilon;t-1,r}\partial_{(0)} \Upsilon_{\epsilon;r;k}(z_{[0:r]}) \bigg)\\
	&\qquad -\eta\cdot \mathfrak{G}^{\circ,(2)}_{\partial_1\mathsf{L};t}\big(\varphi_\epsilon(z_0+\xi_k),z_{[1:t]}\big)\cdot \varphi_\epsilon'(z_0+\xi_k).
	\end{align*}
	Let 
	\begin{align*}
	\bm{G}^{[1]}_{\epsilon;t;k}(z_0,z_{[1:t]}) &\equiv \mathrm{diag}\big(\big\{  \mathfrak{G}^{\circ,(1)}_{\partial_1\mathsf{L};s}(\varphi_\epsilon(z_0+\xi_k),z_{[1:s]})  \big\}_{s\in [1:t]}\big)\in \R^{t\times t},\\
	\bm{g}^{[2]}_{\epsilon;t;k}(z_0,z_{[1:t]})&\equiv \big(\big\{  \mathfrak{G}^{\circ,(2)}_{\partial_1\mathsf{L};s}(\varphi_\epsilon(z_0+\xi_k),z_{[1:s]}) \varphi_\epsilon'(z_0+\xi_k)  \big\}_{s\in [1:t]}\big) \in \R^t.
	\end{align*}
	Then we may solve
	\begin{align*}
	\big(\partial_{(0)} \Upsilon_{\epsilon;s;k}(z_{[0:s]})\big)_{s \in [1:t]} =-\eta\cdot \big(I_t + \eta \bm{G}^{[1]}_{\epsilon;t;k}(z_0,z_{[1:t]}) \mathfrak{O}_t(\bm{\rho}_\epsilon^{[t-1]})\big)^{-1} \bm{g}^{[2]}_{\epsilon;t;k}(z_0,z_{[1:t]}).
	\end{align*}
	The elements of the vector in the second line above is a linear combination of at most $2^t$ terms of the following form indexed by $I \in \{0,1\}^t$,
	\begin{align*}
	H_I\big(\varphi_\epsilon(z_0+\xi_k),z_{[1:t]}\big)\varphi_\epsilon'(z_0+\xi_k),\quad H_I = \prod_{s \in [1:t]: I_s\neq 0, s\leq \pnorm{I}{0}-1}\mathfrak{G}^{\circ,(1)}_{\partial_1\mathsf{L};s}\cdot \mathfrak{G}^{\circ,(2)}_{\partial_1\mathsf{L};\pnorm{I}{0}},
	\end{align*}
	with coefficients $w_I$'s bounded by $(K\Lambda)^{c_t}$. On the other hand, with 
	\begin{align*}
	\mathfrak{b}_t= (\Sigma_{\mathfrak{Z}}^{[t]})_{[1:t]^2}^{-1} (\Sigma_{\mathfrak{Z}}^{[t]})_{[1:t],0},\quad \mathfrak{v}_t^2= \sigma_{\mu_\ast}^2- ((\Sigma_{\mathfrak{Z}}^{[t]})_{[1:t],0})^\top  (\Sigma_{\mathfrak{Z}}^{[t]})_{[1:t]^2}^{-1} (\Sigma_{\mathfrak{Z}}^{[t]})_{[1:t],0},
	\end{align*}	
	we have $\mathfrak{Z}^{(0)}|\mathfrak{Z}^{[1:t]}\sim \mathcal{N}\big(\iprod{\mathfrak{b}_t}{ \mathfrak{Z}^{[1:t]}}, \mathfrak{v}_t^2\big)$. So for a bounded generic function $\mathsf{H}:\R^{[0:t]}\to \R$, if $\mathfrak{v}_t>0$,
	\begin{align}\label{ineq:1bit_cs_delta_general_loss_1}
	&\E^{(0)} \mathsf{H}\big(\varphi_\epsilon(\mathfrak{Z}^{(0)}+\xi_{\pi_m}),\mathfrak{Z}^{([1:t])} \big)\varphi_\epsilon'(\mathfrak{Z}^{(0)}+\xi_{\pi_m})\nonumber\\
	& = \frac{1}{\epsilon} \E_{ \mathfrak{Z}^{([1:t])},\pi_m } \int  \mathsf{H}\big(\varphi(\epsilon^{-1}(z+\xi_{\pi_m})),\mathfrak{Z}^{([1:t])}\big) \varphi'\big(\epsilon^{-1}(z+\xi_{\pi_m})\big) \mathfrak{g}_{\mathfrak{v}_t}\big(z-\iprod{\mathfrak{b}_t}{ \mathfrak{Z}^{[1:t]}}\big)\,\d{z}\nonumber\\
	& = \E_{ \mathfrak{Z}^{([1:t])},\pi_m } \int  \mathsf{H}(\varphi(v), \mathfrak{Z}^{([1:t])}) \varphi'(v)\cdot \mathfrak{g}_{\mathfrak{v}_t}\big(\epsilon v-\xi_{\pi_m}-\iprod{\mathfrak{b}_t}{ \mathfrak{Z}^{[1:t]}}\big)\,\d{v}\nonumber\\
	&\to \E_{ \mathfrak{Z}^{([1:t])}} \big(\E_{\pi_m} \mathfrak{g}_{\mathfrak{v}_t}(\xi_{\pi_m}+\iprod{\mathfrak{b}_t}{ \mathfrak{Z}^{[1:t]}}) \big)\cdot \int  \mathsf{H}(\varphi(v), \mathfrak{Z}^{([1:t])}) \varphi'(v)\,\d{v}
	\end{align}
	as $\epsilon\to 0$. As $\varphi:[-1,1]\to [-1,1]$ is a smooth bijection,  we may compute $\int  \mathsf{H}(\varphi(v), \mathfrak{Z}^{([1:t])}) \varphi'(v) \,\d{v} = \int  \mathsf{H}(\varphi(v), \mathfrak{Z}^{([1:t])})\, \d{\varphi(v)} = \int_{-1}^1 \mathsf{H}(y, \mathfrak{Z}^{([1:t])})\, \d{y} $. Moreover, with some calculations,
	\begin{align*}
	&\cov\left[\binom{ \iprod{ \mathfrak{b}_t}{ \mathfrak{Z}^{[1:t]} }  }{\mathfrak{Z}^{[1:t]}},  \binom{ \iprod{ \mathfrak{b}_t}{ \mathfrak{Z}^{[1:t]} }  }{\mathfrak{Z}^{[1:t]}} \right] = \Sigma_{\mathfrak{Z}}^{[t]}-\mathrm{diag}\big(\mathfrak{v}_t^2,0_{[1:t]}\big)\\
	& = \var(\mathfrak{Z}^{([0:t])})-\E^{(0)} \var(\mathfrak{Z}^{([0:t])}|\mathfrak{Z}^{([1:t])}) = \var\big(\E^{(0)}\big[\mathfrak{Z}^{([0:t])}|\mathfrak{Z}^{([1:t])}\big]\big).
	\end{align*}
	Therefore, with $(\mathsf{Z}_0,\mathsf{Z}_{[1:t]})\sim \mathcal{N}\big(0, \var\big(\E^{(0)}\big[\mathfrak{Z}^{([0:t])}|\mathfrak{Z}^{([1:t])}\big]\big)\big)$,  if $\mathfrak{v}_t>0$,
	\begin{align*}
	& \E^{(0)} \mathsf{H}\big(\varphi_\epsilon(\mathfrak{Z}^{(0)}+\xi_{\pi_m}),\mathfrak{Z}^{([1:t])} \big)\varphi_\epsilon'(\mathfrak{Z}^{(0)}+\xi_{\pi_m})\\
	& \to 2 \E^{(0)}  \mathfrak{g}_{\mathfrak{v}_t}(\xi_{\pi_m}+\mathsf{Z}_0) \cdot  \mathsf{H}(U, \mathsf{Z}_{[1:t]}),\hbox{ as } \epsilon\to 0.
	\end{align*}
	Combined with the stability estimates in Lemma \ref{lem:1bit_cs_se_smooth_diff} for $\pnorm{\d_\epsilon\bm{\delta}^{[t]}}{}$ and $\pnorm{\d_\epsilon\bm{\rho}^{[t]}}{}$,
	we then conclude that if $\mathfrak{v}_t>0$, with $\mathfrak{L}_1(U,\mathsf{Z}_{[1:t]})= \mathrm{diag}\big(\big\{\partial_{11} \mathsf{L}\big( \Theta_t^\circ(U,\mathsf{Z}_{[1:s]}),U   \big)\}_{s \in [t]}\big)$ and $\mathfrak{l}_2(U,\mathsf{Z}_{[1:t]})= \big(\partial_{12} \mathsf{L}( \Theta_t^\circ(U,\mathsf{Z}_{[1:s]}),U)\big)_{s \in [t]}$ defined in the statement of the lemma,
    \begin{align*}
    \delta_t&= -2\phi\eta\cdot \E^{(0)} \Big\{ \mathfrak{g}_{\mathfrak{v}_t}(\xi_{\pi_m}+\mathsf{Z}_0)   e_t^\top\big(I_t + \eta\cdot  \mathfrak{L}_1(U,\mathsf{Z}_{[1:t]}) \mathfrak{O}_t(\bm{\rho}^{[t-1]})\big)^{-1}\mathfrak{l}_2(U,\mathsf{Z}_{[1:t]})\Big\}. 
    \end{align*}
    For the squared loss, as $\partial_{11}\mathsf{L}\equiv \partial_{12} \mathsf{L}\equiv 1$,
    \begin{align*}
    \delta_t&= -2\phi\eta\cdot \E^{(0)}\mathfrak{g}_{\mathfrak{v}_t}(\xi_{\pi_m}+\mathsf{Z}_0)\cdot e_t^\top (I_t+\eta \mathfrak{O}_t(\bm{\rho}^{[t-1]})\big)^{-1}\bm{1}_t.
    \end{align*}
    On the other hand, we may compute
    \begin{align}\label{ineq:1bit_cs_delta_general_loss_2}
    \E^{(0)}\mathfrak{g}_{\mathfrak{v}_t}(\xi_{\pi_m}+\mathsf{Z}_0)& = \E^{(0)} \mathfrak{g}_{\mathfrak{v}_t}(\xi_{\pi_m}+\{\sigma_{\mu_\ast}^2-\mathfrak{v}_t^2\}^{1/2} \mathsf{Z})= \E^{(0)} \mathfrak{g}_{\sigma_{\mu_\ast}}(\xi_{\pi_m}).
    \end{align}
    The term $e_t^\top (I_t+\eta \mathfrak{O}_t(\bm{\rho}^{[t-1]})\big)^{-1}\bm{1}_t$ is computed in Lemma \ref{lem:1bit_cs_delta_tau_smooth}-(2) below. \qed

\subsection{Proof of Proposition \ref{prop:1bit_cs_bias_var}}

\begin{lemma}\label{lem:1bit_cs_tau_general_loss}
Suppose the conditions in Proposition \ref{prop:1bit_cs_bias_var} hold. Then there exists some $c_t=c_t(t)>1$ such that for any $\epsilon>0$, 
\begin{align*}
\pnorm{\bm{\tau}_\epsilon^{[t]} - \bm{\tau}^{[t]} }{\op}\leq \big(K\Lambda L_\mu(1\wedge\sigma_{\mu_\ast})^{-1}\big)^{c_t}\cdot \epsilon^{1/c_t}.
\end{align*}
Here $\bm{\tau}^{[t]}$ is defined in (\ref{def:1bit_tau}).
\end{lemma}
\begin{proof}
Using the same notation as in the proof of Lemma \ref{lem:1bit_cs_delta_general_loss} to translate the identity in (\ref{ineq:rho_tau_bound_Upsilon}), we obtain
\begin{align*}
\big(\partial_{(s)} \Upsilon_{\epsilon;r;k}(z_{[0:r]})\big)_{s,r \in [1:t]} =-\eta\cdot \big(I_t + \eta \bm{G}^{[1]}_{\epsilon;t;k}(z_0,z_{[1:t]}) \mathfrak{O}_t(\bm{\rho}_\epsilon^{[t-1]})\big)^{-1} \bm{G}^{[1]}_{\epsilon;t;k}(z_0,z_{[1:t]}).
\end{align*}
Here recall 
\begin{align*}
\bm{G}^{[1]}_{\epsilon;t;k}(z_0,z_{[1:t]})& = \mathrm{diag}\big(\big\{  \mathfrak{G}^{\circ,(1)}_{\partial_1\mathsf{L};s}(\varphi_\epsilon(z_0+\xi_k),z_{[1:s]})  \big\}_{s\in [1:t]}\big)\\
& = \mathrm{diag}\big(\big\{ \partial_{11} \mathsf{L}(\Theta_{\epsilon;s;k}(z_{[0:t]}),\varphi_\epsilon(z_0+\xi_k)) \big\}_{s\in [1:t]}\big).
\end{align*}
Using the stability estimates in Lemma \ref{lem:1bit_cs_se_smooth_diff}, we have
\begin{align*}
&\bigpnorm{\big(\partial_{(s)} \Upsilon_{\epsilon;r;k}(z_{[0:r]})\big)_{s,r \in [1:t]}-\big(\partial_{(s)} \Upsilon_{r;k}(z_{[0:r]})\big)_{s,r \in [1:t]}}{\op}\\
&\leq \big(K\Lambda L_\mu(1\wedge\sigma_{\mu_\ast})^{-1}\big)^{c_t}\cdot \Big[\big(1+\pnorm{z^{([0:t])} }{}\big)\cdot \epsilon^{1/c_t}+\abs{\d_\epsilon \varphi (z^{(0)}+\xi_k)}\Big].
\end{align*} 
The claim follows by taking the expectation and using the apriori estimates in Lemma \ref{lem:1bit_cs_apriori_est}.
\end{proof}

\begin{proof}[Proof of Proposition \ref{prop:1bit_cs_bias_var}-(1)]
We assume for notational simplicity that $\sigma_{\mu_\ast}\vee \tau_\ast^{(t)}\leq 1$. We use the same proof method as in Theorem \ref{thm:1bit_cs}. To this end, with $\bm{\tau}^{[t]}$ defined in (\ref{def:1bit_tau}), and $\bm{\delta}^{[t]}$ defined in Lemma \ref{lem:1bit_cs_delta_general_loss}, let  $\bm{\omega}^{[t]}\equiv (\bm{\tau}^{[t]})^{-1}$ and
\begin{align*}
b^{(t)}_{\mathrm{db}}\equiv -  \iprod{\bm{\omega}^{[t]}\bm{\delta}^{[t]} }{e_t},\quad (\sigma^{(t)}_{\mathrm{db}})^2\equiv \bm{\omega}_{t\cdot}^{[t]}\Sigma_{\mathfrak{W}}^{[t]}\bm{\omega}_{t\cdot}^{[t],\top}.
\end{align*}
Note that
\begin{align*}
&\E^{(0)} \bigabs{\E_{\pi_n} \psi_{\pi_n} \big(\mu^{(t)}_{\mathrm{db};\pi_n}\big)- \E \psi_{\pi_n}\big(b^{(t)}_{\mathrm{db}}\cdot \mu_{\ast,\pi_n}+ {\sigma}^{(t)}_{\mathrm{db}}\mathsf{Z}\big) }^q\\
&\lesssim_q \E^{(0)} \bigabs{\E_{\pi_n} \psi_{\pi_n} \big(\mu^{(t)}_{\epsilon;\mathrm{db};\pi_n}\big)- \E^{(0)} \psi_{\pi_n}\big(b^{(t)}_{\epsilon;\mathrm{db}}\cdot \mu_{\ast,\pi_n}+ \sigma^{(t)}_{\epsilon;\mathrm{db}}\mathsf{Z}\big)}^q\\
&\qquad + \E^{(0)} \bigabs{\E \psi_{\pi_n}\big({b}^{(t)}_{\mathrm{db}}\cdot \mu_{\ast,\pi_n}+ {\sigma}^{(t)}_{\mathrm{db}}\mathsf{Z}\big)- \E^{(0)} \psi_{\pi_n}\big(b^{(t)}_{\epsilon;\mathrm{db}}\cdot \mu_{\ast,\pi_n}+ \sigma^{(t)}_{\epsilon;\mathrm{db}}\mathsf{Z}\big) }^q\\
&\qquad + \E^{(0)} \bigabs{ \E_{\pi_n} \psi_{\pi_n} \big(\mu^{(t)}_{\epsilon;\mathrm{db};\pi_n}\big)- \E_{\pi_n} \psi_{\pi_n} \big(\mu^{(t)}_{\mathrm{db};\pi_n}\big) }^q\\
&\equiv I_0(\epsilon)+I_1(\epsilon)+I_2(\epsilon).
\end{align*}
For $I_0(\epsilon)$, we may apply Theorem \ref{thm:db_gd_oracle} to obtain
\begin{align*}
I_0(\epsilon)\leq \big(\tau_\ast^{(t),-1} \epsilon^{-1}\cdot K\Lambda \Lambda_\psi L_\mu\big)^{c_t} \cdot n^{-1/c_t}.
\end{align*}
For $I_1(\epsilon)$, using the stability estimate in Lemma \ref{lem:1bit_cs_se_smooth_diff}, Lemma \ref{lem:1bit_cs_tau_general_loss}, along with the apriori estimates in Lemma \ref{lem:1bit_cs_apriori_est} in combination with Lemma \ref{lem:lower_tri_mat_inv}, 
\begin{align*}
I_1(\epsilon)&\leq \big(K\Lambda \Lambda_\psi L_\mu/(\sigma_{\mu_\ast}\tau_\ast^{(t)})\big)^{c_t}\cdot \big(\abs{b^{(t)}_{\epsilon;\mathrm{db}}-b^{(t)}_{\mathrm{db}}} + \abs{\sigma^{(t)}_{\epsilon;\mathrm{db}}-\sigma^{(t)}_{\mathrm{db}}} \big)^q\\
&\leq \big(K\Lambda \Lambda_\psi L_\mu/(\sigma_{\mu_\ast}\tau_\ast^{(t)})\big)^{c_t}\cdot\epsilon^{1/c_t}.
\end{align*}
For $I_2(\epsilon)$, using Lemmas \ref{lem:1bit_cs_apriori_est_data} and \ref{lem:1bit_cs_data_smooth_diff},
\begin{align*}
I_2 \leq \big(K\Lambda \Lambda_\psi L_\mu/(\sigma_{\mu_\ast}\tau_\ast^{(t)})\big)^{c_t}\cdot \epsilon^{1/c_t}.
\end{align*}
Combining the above estimates to conclude upon choosing appropriately $\epsilon>0$.
\end{proof}

Let us now deal with the squared loss case. 
\begin{lemma}\label{lem:1bit_cs_delta_tau_smooth}
	Consider the squared loss $\mathsf{L}(x,y)\equiv (x-y)^2/2$. Suppose the conditions in Proposition \ref{prop:1bit_cs_bias_var} hold. The following hold for any $\epsilon>0$.
	\begin{enumerate}
		\item $\bm{\delta}_\epsilon^{[t]} = \phi\eta\cdot \E^{(0)} \varphi_\epsilon'(\sigma_{\mu_\ast}\mathsf{Z}+\xi_{\pi_m})\cdot \big[I_t+\eta \cdot  \mathfrak{O}_t(\bm{\rho}_\epsilon^{[t-1]})\big]^{-1} \bm{1}_t$.
		\item $\bm{\tau}_\epsilon^{[t]}= -\phi\eta \cdot \big[I_t+\eta \cdot  \mathfrak{O}_t(\bm{\rho}_\epsilon^{[t-1]})\big]^{-1}$.
		\item $b^{(t)}_{\epsilon;\mathrm{db}}=\E^{(0)} \varphi_\epsilon'(\sigma_{\mu_\ast}\mathsf{Z}+\xi_{\pi_m})$.
		\item $(\sigma^{(t)}_{\epsilon;\mathrm{db}})^2 = \phi^{-1}\cdot \E^{(0)}\big(\mathfrak{Z}_\epsilon^{(t)}-\varphi_\epsilon(\mathfrak{Z}_\epsilon^{(0)}+\xi_{\pi_m})\big)^2$.
	\end{enumerate}
\end{lemma}

\begin{proof}
	Fix $k \in [m]$. The state evolution in (S1) reads
	\begin{align}\label{ineq:1bit_cs_delta_tau_smooth_1}
	\Upsilon_{\epsilon;t;k}(z_{[0:t]})&= -\eta\cdot  \bigg(z_t+\sum_{r \in [1:t-1]}\rho_{\epsilon;t-1,r} \Upsilon_{\epsilon;r;k}(z_{[0:r]})- \varphi_\epsilon(z_0+\xi_k)\bigg).
	\end{align}
	In the matrix form, with ${\bm{\Upsilon}}^{(t)}_{\epsilon;k}(z_{[0:t]})\equiv \big( \Upsilon_{\epsilon;r;k}(z_{[0:r]})\big)_{r \in [1:t]}$, we have
	\begin{align*}
	\bm{\Upsilon}^{(t)}_{\epsilon;k}(z_{[0:t]})= -\eta\cdot \big(z_r-\varphi_\epsilon(z_0+\xi_k)\big)_{r \in [1:t]} - \eta\cdot \mathfrak{O}_t(\bm{\rho}_\epsilon^{[t-1]}) {\bm{\Upsilon}}^{(t)}_{\epsilon;k}(z_{[0:t]}).
	\end{align*}
	Consequently, we may solve
	\begin{align}\label{ineq:1bit_cs_delta_tau_smooth_2}
	\bm{\Upsilon}^{(t)}_{\epsilon;k}(z_{[0:t]}) = -\eta\cdot \big[I_t+\eta \cdot  \mathfrak{O}_t(\bm{\rho}_\epsilon^{[t-1]})\big]^{-1} \big(z_r-\varphi_\epsilon(z_0+\xi_k)\big)_{r \in [1:t]}.
	\end{align}
	
	\noindent (1). This claim is already contained in the Lemma \ref{lem:1bit_cs_delta_general_loss}. In the squared case, we may directly take derivative with respect to $z_0$ on both sides of (\ref{ineq:1bit_cs_delta_tau_smooth_2}) to conclude.  
	
	\noindent (2). Taking derivative with respect to $z_s$ on both sides of (\ref{ineq:1bit_cs_delta_tau_smooth_1}), with ${\bm{\Upsilon}}^{';(t)}_{\epsilon;k}(z_{[0:t]})\equiv \big(\partial_{(s)} \Upsilon_{\epsilon;r;k}(z_{[0:r]})\big)_{r,s \in [1:t]}$, we have $
	{\bm{\Upsilon}}^{';(t)}_{\epsilon;k}(z_{[0:t]})= - \eta I_t - \eta\cdot \mathfrak{O}_t(\bm{\rho}_\epsilon^{[t-1]}) {\bm{\Upsilon}}^{';(t)}_{\epsilon;k}(z_{[0:t]})$. 
	Solving for ${\bm{\Upsilon}}^{';(t)}_{\epsilon;k}(z_{[0:t]})$ we obtain
	\begin{align*}
	{\bm{\Upsilon}}^{';(t)}_{\epsilon;k}(z_{[0:t]}) = -\eta \big[I_t+\eta \cdot  \mathfrak{O}_t(\bm{\rho}_\epsilon^{[t-1]})\big]^{-1}.
	\end{align*}
	Taking expectation to conclude. 
	
	\noindent (3). Using the definition, we have
	\begin{align*}
	b^{(t)}_{\epsilon;\mathrm{db}} = -e_t^\top (\bm{\tau}_\epsilon^{[t]})^{-1} \bm{\delta}_\epsilon^{[t]}=\E^{(0)} \varphi_\epsilon'(\sigma_{\mu_\ast}\mathsf{Z}+\xi_{\pi_m}),
	\end{align*}
	as desired.
	
	\noindent (4). By (\ref{ineq:1bit_cs_delta_tau_smooth_2}), we have
	\begin{align*}
	(\sigma^{(t)}_{\epsilon;\mathrm{db}})^2 
	& = e_t^\top \bm{\omega}_\epsilon^{[t]}\big(\phi\cdot \E^{(0)} \bm{\Upsilon}^{(t)}_{\epsilon;k}(\mathfrak{Z}_\epsilon^{([0:t])}) \bm{\Upsilon}^{(t),\top}_{\epsilon;k}(\mathfrak{Z}_\epsilon^{([0:t])})   \big)\bm{\omega}_\epsilon^{[t],\top} e_t \\
	& = \phi^{-1}\cdot \E^{(0)}\big(\mathfrak{Z}_\epsilon^{(t)}-\varphi_\epsilon(\mathfrak{Z}_\epsilon^{(0)}+\xi_{\pi_m})\big)^2,
	\end{align*}
	completing the proof.
\end{proof}

\begin{proof}[Proof of Proposition \ref{prop:1bit_cs_bias_var}-(2)]
	We assume for notational simplicity that $\sigma_{\mu_\ast}\leq 1$. Note that for the squared loss, $\tau_\ast^{(t)}=1$. 
   
   We first consider bias. As $b^{(t)}_{\epsilon;\mathrm{db}}= \E^{(0)} \varphi_\epsilon'(\sigma_{\mu_\ast}\mathsf{Z}+\xi_{\pi_m})$, we have $\err_\xi(\epsilon)\equiv \abs{\E^{(0)} \varphi_\epsilon'(\sigma_{\mu_\ast}\mathsf{Z}+\xi_{\pi_m})-2\E^{(0)} \mathfrak{g}_{\sigma_{\mu_\ast}}(\xi_{\pi_m}) }\to 0$ as $\epsilon\to 0$. So
	\begin{align*}
	&\abs{b^{(t)}_{\mathrm{db}}-2\E^{(0)} \mathfrak{g}_{\sigma_{\mu_\ast}}(\xi_{\pi_m})} \leq \abs{b^{(t)}_{\mathrm{db}}-b^{(t)}_{\epsilon;\mathrm{db}}}+\err_\xi(\epsilon)\\
	&\leq \pnorm{ \bm{\omega}^{[t]}}{\op}\pnorm{ \bm{\omega}_\epsilon^{[t]}}{\op}\cdot \pnorm{\d_\epsilon \bm{\tau}^{[t]}}{\op} \pnorm{ \bm{\delta}_\epsilon^{[t]}}{}+\pnorm{ \bm{\omega}^{[t]}}{\op} \pnorm{\d_\epsilon \bm{\delta}^{[t]}}{}+ \err_\xi(\epsilon).
	\end{align*}
	So if $\pnorm{\d_\epsilon \bm{\tau}^{[t]}}{\op} \leq \tau_\ast^{(t)}/2$, using Lemmas \ref{lem:1bit_cs_apriori_est} and \ref{lem:1bit_cs_se_smooth_diff},
	\begin{align*}
	\abs{b^{(t)}_{\mathrm{db}}-2\E^{(0)} \mathfrak{g}_{\sigma_{\mu_\ast}}(\xi_{\pi_m}) } \leq \big(K\Lambda L_\mu/\sigma_{\mu_\ast}\big)^{c_t}\cdot \epsilon^{1/c_t} + \err_\xi(\epsilon).
	\end{align*}
	Now we may let $\epsilon\to 0$ to conclude. 
	
	We next consider variance under the squared loss. By Lemma \ref{lem:1bit_cs_delta_tau_smooth},
	\begin{align*}
	&\bigabs{(\sigma^{(t)}_{\mathrm{db}})^2 -  \phi^{-1}\E^{(0)}\big(\mathfrak{Z}^{(t)}-\sign(\mathfrak{Z}^{(0)}+\xi_{\pi_m})\big)^2 }\leq \bigabs{(\sigma^{(t)}_{\epsilon;\mathrm{db}})^2-(\sigma^{(t)}_{\mathrm{db}})^2}\\
	&\quad + \phi^{-1}\bigabs{ \E^{(0)}\big(\mathfrak{Z}_\epsilon^{(t)}-\varphi_\epsilon(\mathfrak{Z}_\epsilon^{(0)}+\xi_{\pi_m})\big)^2-  \E^{(0)}\big(\mathfrak{Z}^{(t)}-\sign(\mathfrak{Z}^{(0)}+\xi_{\pi_m})\big)^2 }\\
	& \equiv V_1+V_2.
	\end{align*}
	First we handle $V_1$. Using that for two covariance matrices $\Sigma_1,\Sigma_2 \in \R^{t\times t}$, $\pnorm{\Sigma_1^{1/2}-\Sigma_2^{1/2}}{\op}\leq t \pnorm{\Sigma_1-\Sigma_2}{\op}^{1/2}$ (cf. \cite[Lemma A.3]{bao2025leave}),
	\begin{align*}
	\bigabs{\sigma^{(t)}_{\epsilon;\mathrm{db}}-\sigma^{(t)}_{\mathrm{db}}}&\leq \bigabs{\pnorm{\Sigma_{\epsilon;\mathfrak{W}}^{[t],1/2} \bm{\omega}_\epsilon^{[t],\top} e_t }{}-\pnorm{\Sigma_{\mathfrak{W}}^{[t],1/2} \bm{\omega}^{[t],\top} e_t }{} }\\
	&\leq t\cdot \pnorm{\d_\epsilon \Sigma_{\mathfrak{W}}^{[t]}}{\op}\cdot \pnorm{ \bm{\omega}_\epsilon^{[t]} }{}+ \pnorm{\Sigma_{\mathfrak{W}}^{[t]}}{\op}\cdot \pnorm{\d_\epsilon \bm{\omega}^{[t]}}{}\\
	&\leq \big(K\Lambda L_\mu/\sigma_{\mu_\ast}\big)^{c_t}\cdot \epsilon^{1/c_t}.
	\end{align*}
	Here the last line follows from similar arguments for the bias. Now using the apriori estimates $\abs{\sigma^{(t)}_{\epsilon;\mathrm{db}}}\vee \abs{\sigma^{(t)}_{\mathrm{db}}}\leq \big(K\Lambda L_\mu/\sigma_{\mu_\ast}\big)^{c_t}$, we conclude that
	\begin{align*}
	V_1\leq \big(K\Lambda L_\mu/\sigma_{\mu_\ast}\big)^{c_t}\cdot \epsilon^{1/c_t}.
	\end{align*}
	Next we handle $V_2$. With some calculations using Lemmas \ref{lem:1bit_cs_apriori_est} and \ref{lem:1bit_cs_se_smooth_diff},
	\begin{align*}
	V_2 &\lesssim \phi^{-1}\cdot \big(t\cdot \pnorm{\d_\epsilon \Sigma_{\mathfrak{Z}}^{[t]}}{\op }^{1/2}+ \E^{(0),1/2} \abs{\d_\epsilon \varphi (\mathfrak{Z}^{(0)}+\xi_{\pi_m})}^2\big)\\
	&\qquad \times \max_{\alpha = 0,\epsilon} \E^{(0),1/2}\big(1+ \abs{\mathfrak{Z}_\alpha^{[t]}}\big)^2\leq \big(K\Lambda L_\mu/\sigma_{\mu_\ast}\big)^{c_t}\cdot \epsilon^{1/c_t}.
	\end{align*}
	The claim follows by taking $\epsilon \to 0$.
\end{proof}

\subsection{Proof of Proposition \ref{prop:gen_error_linear}}

For any $\epsilon>0$, let
\begin{align*}
\mathscr{E}_{\epsilon;\mathsf{H}}^{(t)}(A,Y)&\equiv \E\big[\mathsf{H}\big(\iprod{A_{\mathrm{new}}}{\mu_\epsilon^{(t)}},\varphi_\epsilon(\iprod{A_{\mathrm{new}}}{\mu_\ast}+\xi_{\pi_m})\big)|(A,Y)\big],\\
\hat{\mathscr{E}}_{\epsilon;\mathsf{H}}^{(t)}&\equiv m^{-1} \iprod{\mathsf{H} (\hat{Z}_\epsilon^{(t)},Y_\epsilon)}{\bm{1}_m}.
\end{align*}
Then we have 
\begin{align*}
&\E^{(0)}\abs{\hat{\mathscr{E}}_{\mathsf{H}}^{(t)}- \mathscr{E}_{\mathsf{H}}^{(t)}(A,Y) }^q\lesssim_q \E^{(0)}\abs{\hat{\mathscr{E}}_{\epsilon;\mathsf{H}}^{(t)}- \mathscr{E}_{\epsilon;\mathsf{H}}^{(t)}(A,Y) }^q\\
&\qquad + \E^{(0)}\abs{\hat{\mathscr{E}}_{\mathsf{H}}^{(t)}- \hat{\mathscr{E}}_{\epsilon;\mathsf{H}}^{(t)} }^q + \E^{(0)}\abs{\mathscr{E}_{\mathsf{H}}^{(t)}(A,Y)- \mathscr{E}_{\epsilon;\mathsf{H}}^{(t)}(A,Y) }^q\\
&\equiv I_0(\epsilon)+I_1(\epsilon)+I_2(\epsilon).
\end{align*}
For $I_0(\epsilon)$, we may apply Theorem \ref{thm:gen_error} to obtain 
\begin{align*}
I_0(\epsilon)\leq (\epsilon^{-1}\cdot K\Lambda L_\mu)^{c_t}\cdot n^{-1/c_t}.
\end{align*}
For $I_1(\epsilon)$ and $I_2(\epsilon)$, using the stability estimates in Lemmas \ref{lem:1bit_cs_se_smooth_diff} and \ref{lem:1bit_cs_data_smooth_diff},
\begin{align*}
I_1(\epsilon)+I_2(\epsilon)\leq \big(K\Lambda L_\mu(1\wedge \sigma_{\mu_\ast})^{-1}\big)^{c_t}\cdot \epsilon^{1/c_t}.
\end{align*}
Combining the above displays to conclude by choosing appropriately $\epsilon>0$. \qed

\appendix

\section{GFOM state evolution theory in \cite{han2025entrywise}}\label{section:GFOM_se}

This section reviews some basics for the theory of general first order methods (GFOM) in \cite{han2025entrywise} that will be used in the proof in this paper. We shall only present a simplified version with the design matrix $A$ satisfying (A2) in Assumption \ref{assump:setup}. The reader is referred to \cite{han2025entrywise} for a more general theory allowing for $A$ with a general variance profile.

Consider an asymmetric GFOM initialized with $(u^{(0)},v^{(0)})\in \R^{m}\times \R^{n}$, and subsequently updated according to 
\begin{align}\label{def:GFOM_asym}
\begin{cases}
u^{(t)} = A \mathsf{F}_t^{\langle 1\rangle}(v^{([0:t-1])})+ \mathsf{G}_{t}^{\langle 1\rangle}(u^{([0:t-1])}) \in \R^{m},\\
v^{(t)}= A^\top \mathsf{G}_t^{\langle 2\rangle}(u^{([0:t])})+\mathsf{F}_t^{\langle 2\rangle}(v^{([0:t-1])})\in \R^{n}.
\end{cases}
\end{align}
Here we denote $A$ as an $m\times n$ random matrix, and the row-separate functions $\mathsf{F}_t^{\langle 1 \rangle}, \mathsf{F}_t^{\langle 2 \rangle}:\R^{n\times [0:t-1]} \to \R^{n}$, $\mathsf{G}_t^{\langle 1 \rangle}:\R^{m\times [0:t-1]} \to \R^{m}$ and $\mathsf{G}_t^{\langle 2 \rangle}:\R^{m\times [0:t]} \to \R^{m}$ are understood as applied row-wise. 

The state evolution for the asymmetric GFOM (\ref{def:GFOM_asym}) is iteratively described, in the following definition, by (i) two row-separate maps $\Phi_t: \R^{m\times [0:t]} \to \R^{m\times [0:t]}$ and $\Xi_t: \R^{n\times [0:t]} \to \R^{n\times [0:t]}$, and (ii) two centered Gaussian matrices $\mathfrak{U}^{([1:\infty))}\in \R^{ [1:\infty)}$ and $\mathfrak{V}^{([1:\infty))}\in \R^{ [1:\infty)}$.

\begin{definition}\label{def:GFOM_se_asym}
	Initialize with $\Phi_0=\mathrm{id}(\R^m)$, $\Xi_0\equiv \mathrm{id}(\R^n)$, and $\mathfrak{U}^{(0)}\equiv u^{(0)}$, $\mathfrak{V}^{(0)}\equiv v^{(0)}$. For $t=1,2,\ldots$, with 
	\begin{align*}
	\{ \mathsf{F}_{t,\ell}^{\langle 1\rangle}\circ \Xi_{t-1,\ell}\}^\circ \big(v^{([1:t-1])}\big) &\equiv \{\mathsf{F}_{t,\ell}^{\langle 1\rangle}\circ \Xi_{t-1,\ell}\} \big( \mathfrak{V}_\ell^{(0)}, v^{([1:t-1])}\big),\\
	\big\{ \mathsf{G}_{t,k}^{\langle 2\rangle}\circ\Phi_{t,k}\big\}^\circ (u^{([1:t])})& \equiv \big\{ \mathsf{G}_{t,k}^{\langle 2\rangle}\circ\Phi_{t,k}\big\} (\mathfrak{U}_k^{(0)},u^{([1:t])}),
	\end{align*}
	$\E^{(0)}\equiv \E\big[\cdot|(\mathfrak{U}^{(0)},\mathfrak{V}^{(0)})\big]$, and $\pi_n$ denoting the uniform distribution on $[n]$, we execute the following steps:
	\begin{enumerate}
		\item Let $\Phi_{t}:\R^{m\times [0:t]}\to \R^{m\times [0:t]}$ be defined as follows: for $w \in [0:t-1]$, $\big[\Phi_{t}(\mathfrak{u}^{([0:t])})\big]_{\cdot,w}\equiv \big[\Phi_{w}(\mathfrak{u}^{([0:w])})\big]_{\cdot,w}$, and for $w=t$,
		\begin{align*}
		\big[\Phi_{t}(\mathfrak{u}^{([0:t])})\big]_{\cdot,t} \equiv \mathfrak{u}^{(t)}+\sum_{s \in [1:t-1]}  \mathsf{G}_{s}^{\langle 2\rangle}\big(\Phi_{s}(\mathfrak{u}^{([0:s])}) \big) \mathfrak{f}_{s}^{(t-1)}+\mathsf{G}_{t}^{\langle 1\rangle}\big(\Phi_{t-1}(\mathfrak{u}^{([0:t-1])}) \big),
		\end{align*}
		where the correction coefficients $\{\mathfrak{f}_{s}^{(t-1) } \}_{s \in [1:t-1]}\subset \R$ are determined by
		\begin{align*}
		\mathfrak{f}_{s}^{(t-1) } \equiv   \E^{(0)} \partial_{\mathfrak{V}^{(s)}} \big\{ \mathsf{F}_{t,\pi_n}^{\langle 1\rangle}\circ \Xi_{t-1,\pi_n}\big\}^\circ (\mathfrak{V}^{([1:t-1])}).
		\end{align*}
		\item Let the Gaussian law of $\mathfrak{U}^{(t)}$ be determined via the following correlation specification: for $s \in [1:t]$,
		\begin{align*}
		\mathrm{Cov}\big(\mathfrak{U}^{(t)}, \mathfrak{U}^{(s)} \big)\equiv   \E^{(0)} \prod_{*\in \{t,s\} } \big\{ \mathsf{F}_{*,\pi_n}^{\langle 1\rangle}\circ \Xi_{*-1,\pi_n}\big\}^\circ (\mathfrak{V}^{([1:*-1])}).
		\end{align*}
		\item Let $\Xi_{t}:\R^{n\times [0:t]}\to \R^{n\times [0:t]}$ be defined as follows: for $w \in [0:t-1]$, $\big[\Xi_{t}(\mathfrak{v}^{([0:t])})\big]_{\cdot,w}\equiv \big[\Xi_{w}(\mathfrak{v}^{([0:w])})\big]_{\cdot,w}$, and for $w=t$,
		\begin{align*}
		\big[\Xi_{t}(\mathfrak{v}^{([0:t])})\big]_{\cdot,t} \equiv \mathfrak{v}^{(t)}+\sum_{s \in [1:t]}  \mathsf{F}_{s}^{\langle 1\rangle}\big(\Xi_{s-1}(\mathfrak{v}^{([0:s-1])}) \big) \mathfrak{g}_{s}^{(t)}+\mathsf{F}_{t}^{\langle 2\rangle}\big(\Xi_{t-1}(\mathfrak{v}^{([0:t-1])}) \big),
		\end{align*}
		where the correction coefficients $\{\mathfrak{g}_{s}^{(t)}\}_{s \in [1:t]}\subset \R$ are determined via
		\begin{align*}
		\mathfrak{g}_{s}^{(t)}\equiv \phi\cdot  \E^{(0)} \partial_{\mathfrak{U}^{(s)}} \big\{ \mathsf{G}_{t,\pi_m}^{\langle 2\rangle}\circ\Phi_{t,\pi_m}\big\}^\circ  (\mathfrak{U}^{([1:t])}).
		\end{align*}
		\item Let the Gaussian law of $\mathfrak{V}^{(t)}$ be determined via the following correlation specification: for $s \in [1:t]$,
		\begin{align*}
		\mathrm{Cov}(\mathfrak{V}^{(t)},\mathfrak{V}^{(s)})\equiv  \phi\cdot  \E^{(0)} \prod_{*\in \{t,s\} } \big\{ \mathsf{G}_{*,\pi_m}^{\langle 2\rangle}\circ\Phi_{*,\pi_m}\big\}^\circ  (\mathfrak{U}^{([1:*])}).
		\end{align*}
	\end{enumerate}
\end{definition}

The next two theorems provide distributional characterizations for $\{u^{(t)},v^{(t)}\}$ in both an entrywise and an averaged sense. 

\begin{theorem}\label{thm:GFOM_se_asym}
	Fix $t \in \N$ and $n \in \N$. Suppose the following hold:
	\begin{enumerate}
		\item[($D^\ast$1)] $A\equiv A_0/\sqrt{n}$, where the entries of $A_0\in \R^{m\times n}$ are independent mean $0$ variables such that $\max_{i,j \in [n]}\pnorm{A_{0,ij}}{\psi_2}\leq K$ holds for some $K\geq 2$.
		\item[($D^\ast$2)] For all $s \in [t], \#\in \{1,2\}, k \in [m], \ell \in [n]$, $\mathsf{F}_{s,\ell}^{\langle \# \rangle},\mathsf{G}_{s,k}^{\langle 1 \rangle} \in C^3(\R^{[0:s-1]})$ and $\mathsf{G}_{s,k}^{\langle 2 \rangle} \in C^3(\R^{[0:s]})$. Moreover, there exists some $\Lambda\geq 2$ and $\mathfrak{p}\in \N$ such that 
		\begin{align*}
		&\max_{s \in [t]}\max_{\# =1,2}\max_{k\in[m], \ell \in [n]}\Big\{ \abs{ \mathsf{F}_{s,\ell}^{\langle \# \rangle}(0) }+ \abs{\mathsf{G}_{s,k}^{\langle \# \rangle}(0)}\\
		&\quad +\max_{0\neq a \in \mathbb{Z}_{\geq 0}^{[0:s-1]},0\neq b \in \mathbb{Z}_{\geq 0}^{[0:s]},\abs{a}\vee \abs{b}\leq 3 } \Big(\bigpnorm{ \partial^a \mathsf{F}_{s,\ell}^{\langle \# \rangle} }{\infty}+\bigpnorm{ \partial^a \mathsf{G}_{s,k}^{\langle 1 \rangle}  }{\infty}+ \bigpnorm{  \partial^b \mathsf{G}_{s,k}^{\langle 2 \rangle}   }{\infty} \Big)   \Big\}\leq \Lambda.
		\end{align*}
	\end{enumerate}
	Further suppose $1/K\leq m/n\leq K$. Then for any $\Psi \in C^3(\R^{ [0:t]})$ satisfying
	\begin{align}\label{cond:Psi_asym}
	\max_{a \in \mathbb{Z}_{\geq 0}^{ [0:t]}, \abs{a}\leq 3} \sup_{x \in \R^{ [0:t] } }\bigg(\sum_{ \tau \in  [0:t] }(1+\abs{x_{\tau}})^{\mathfrak{p}}\bigg)^{-1} \abs{\partial^a \Psi(x)} \leq \Lambda_{\Psi}
	\end{align}
	for some $\Lambda_{\Psi}\geq 2$, it holds for some universal $c_0>0$ and $c_1\equiv c_1(\mathfrak{p})>0$, such that with $\E^{(0)}\equiv \E[\cdot|(u^{(0)},v^{(0)})]$, 
	\begin{align*}
	& \max_{k \in [m]}\bigabs{ \E^{(0)} \Psi\big(u_{k}^{([0:t])}(A)\big)-\E^{(0)}\Psi\big(\Phi_{t,k}(\mathfrak{U}_k^{(0)},\mathfrak{U}^{([1:t])})\big) } \\
	& \quad \vee \max_{\ell \in [n]} \bigabs{\E^{(0)} \Psi\big(v_{\ell}^{([0:t])}(A)\big) -\E^{(0)} \Psi\big(\Xi_{t,\ell}(\mathfrak{V}_\ell^{(0)},\mathfrak{V}^{([1:t])})\big) } \\
	& \leq \Lambda_\Psi \cdot \big(K\Lambda \log n\cdot (1+\pnorm{u^{(0)}}{\infty}+\pnorm{v^{(0)}}{\infty})\big)^{c_1 t^5} \cdot n^{-1/c_0^t}.
	\end{align*}
\end{theorem}

\begin{theorem}\label{thm:GFOM_se_asym_avg}
	Fix $t \in \N$ and $n \in \N$, and suppose $1/K\leq m/n\leq K$ for some $K\geq 2$. Suppose $(D^\ast 1)$ in Theorem \ref{thm:GFOM_se_asym} holds and $(D^\ast2)$ therein is replaced by
	\begin{enumerate}
		\item[($D^\ast $2)'] 
		$
		\max\limits_{s \in [t]}\max\limits_{  \# = 1,2   }\max\limits_{k \in [m], \ell \in [n]} \big\{\pnorm{ \mathsf{F}_{s,\ell}^{\langle \# \rangle} }{\mathrm{Lip}}+ \pnorm{\mathsf{G}_{s,k}^{\langle \# \rangle}}{\mathrm{Lip}}+\abs{\mathsf{F}_{s,\ell}^{\langle \# \rangle} (0)}+ \abs{\mathsf{G}_{s,k}^{\langle \# \rangle} (0)}\big\}\leq \Lambda$ for some $\Lambda\geq 2$.
	\end{enumerate}
	Fix a sequence of $\Lambda_\psi$-pseudo-Lipschitz functions $\{\psi_k:\R^{[0:t]} \to \R\}_{k \in [m\vee n]}$ of order $\mathfrak{p}$, where $\Lambda_\psi\geq 2$. Then for any $r_0 \in \N$, there exists some $C_0=C_0(\mathfrak{p},r_0)>0$ such that with $\E^{(0)}\equiv \E[\cdot|(u^{(0)},v^{(0)})]$,  
	\begin{align*}
	&\E^{(0)} \bigg[\biggabs{\frac{1}{m}\sum_{k \in [m]} \psi_k\big(u_k^{([0:t])}(A)\big) - \frac{1}{m}\sum_{k \in [m]}  \E^{(0)}  \psi_k\big(\Phi_{s,k}(\mathfrak{U}_k^{(0)},\mathfrak{U}^{([1:s])})\big)  }^{r_0}\bigg]\\
	&\quad \vee \E^{(0)}  \bigg[\biggabs{\frac{1}{n}\sum_{\ell \in [n]} \psi_\ell\big(v_\ell^{([0:t])}(A)\big) - \frac{1}{n}\sum_{\ell \in [n]}  \E^{(0)}  \psi_\ell\big(\Xi_{t,\ell}(\mathfrak{V}_\ell^{(0)},\mathfrak{V}^{([1:t])})\big)   }^{r_0}\bigg]  \\
	&\leq \big(K\Lambda\Lambda_\psi\log n\cdot (1+\pnorm{u^{(0)}}{\infty}+\pnorm{v^{(0)}}{\infty})\big)^{C_0 t^5}\cdot n^{-1/C_0^t}. 
	\end{align*}
\end{theorem}

\section{Auxiliary technical results}\label{section:aux_result}

\begin{lemma}\label{lem:lindeberg}
	Let $X=(X_1,\ldots,X_n)$ and $Y=(Y_1,\ldots,Y_n)$ be two random vectors in $\R^n$ with independent components such that $\E X_i^\ell=\E Y_i^\ell$ for $i \in [n]$ and $\ell=1,2$. Then for any $f \in C^3(\R^n)$,
	\begin{align*}
	\bigabs{\E f(X) - \E f(Y)}&\leq \max_{U_i \in \{X_i,Y_i\}}\biggabs{\sum_{i=1}^n\E U_i^3 \int_0^{1} \partial_i^3 f(X_{[1:(i-1)]},tU_i, Y_{[(i+1):n]} )(1-t)^2\,\d{t}}.
	\end{align*}
\end{lemma}
\begin{proof}
The claim is essentially contained in \cite{chatterjee2006generalization}; see e.g., \cite[Proposition 6.1]{han2025entrywise} for a detailed proof.
\end{proof}

			\begin{figure}[t]
	\begin{minipage}[t]{0.3\textwidth}
		\includegraphics[width=\textwidth]{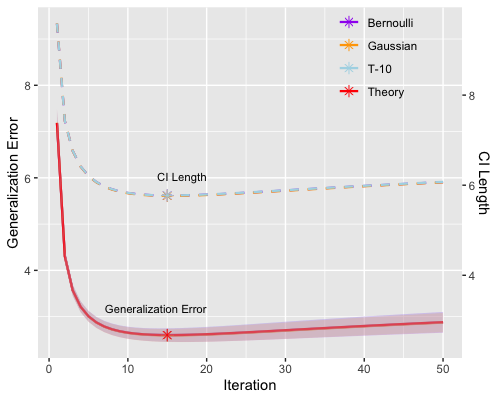}
	\end{minipage}
	\begin{minipage}[t]{0.3\textwidth}
		\includegraphics[width=\textwidth]{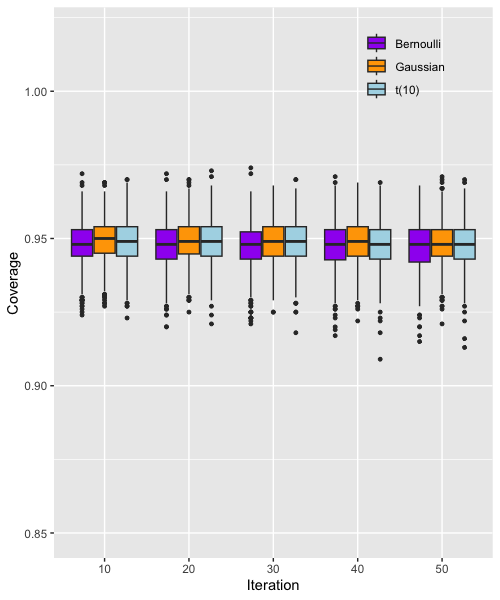}
	\end{minipage}
	\begin{minipage}[t]{0.3\textwidth}
		\includegraphics[width=\textwidth]{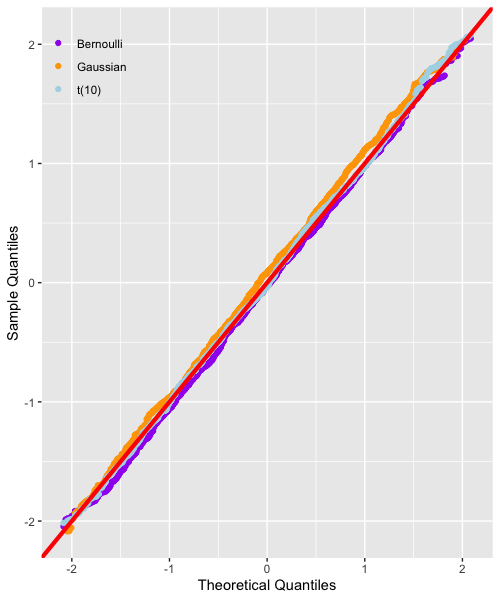}
	\end{minipage}
	\begin{minipage}[t]{0.3\textwidth}
		\includegraphics[width=\textwidth]{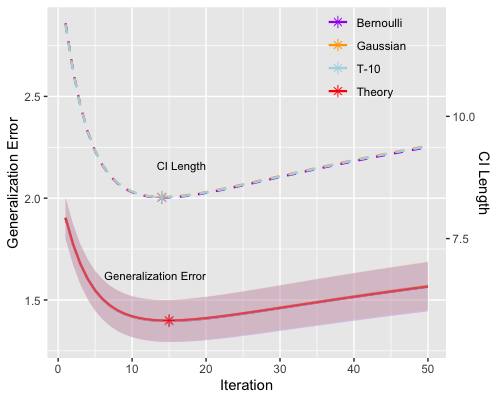}
	\end{minipage}
	\begin{minipage}[t]{0.3\textwidth}
		\includegraphics[width=\textwidth]{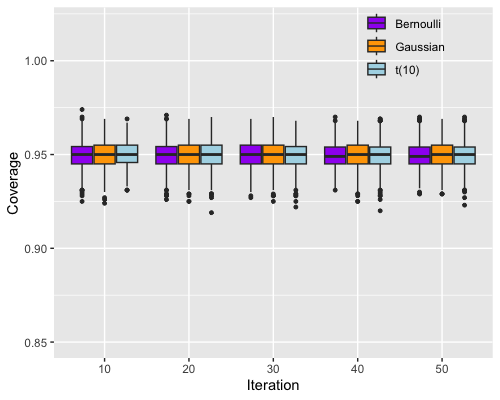}
	\end{minipage}
	\begin{minipage}[t]{0.3\textwidth}
		\includegraphics[width=\textwidth]{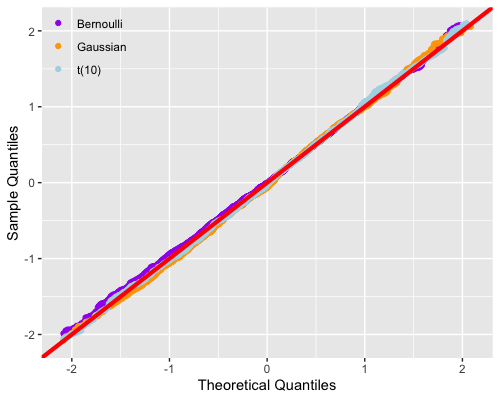}
	\end{minipage}
	\caption{Linear regression with $\ell_1$ penalty. \emph{Top row}: Squared loss. \emph{Bottom row}: Pseudo-Huber loss.}
	\label{fig:4}
\end{figure}

\begin{lemma}\label{lem:lower_tri_mat_inv}
	Let $M \in \R^{t\times t}$ be a lower triangular matrix. Then its inverse $L\equiv M^{-1}$ satisfies, for some universal constant $c_0>0$,
	\begin{align*}
	\pnorm{L}{\op}\leq \bigg(\frac{c_0t\cdot \pnorm{M}{\op}}{\min_{s \in [t]} \abs{M_{ss}}}\bigg)^t.
	\end{align*}
\end{lemma}
\begin{proof}
	Note that $L$ is also lower triangular. Using $ L M= I_t$, for any $r,s\in [t]$ we have $\delta_{r,s}=\sum_{v \in [t]} L_{r,v} M_{v,s}=\sum_{v\in [s:r]} L_{r,v} M_{v,s}$. Consequently, for $s=r$, we have $L_{rr}=1/M_{rr}$. For general $k \in [1:r]$, as $0=\sum_{r-k+1\leq v\leq r} L_{r,v} M_{v,r-k}+L_{r,r-k}M_{r-k,r-k}$, we obtain the bound
	\begin{align*}
	\abs{L_{r,r-k}} \leq \frac{1}{ \abs{M_{r-k,r-k}} } \sum_{r-k+1\leq v\leq r} \abs{L_{r,v}} \abs{ M_{v,t-k}}\leq \frac{\pnorm{M}{\op}}{\min_{s \in [r]} \abs{M_{ss}}}\cdot \sum_{r-k+1\leq v\leq r} \abs{L_{r,v}}.
	\end{align*}
	Iterating the bound to conclude. 
\end{proof}

	\section{Additional numerical experiments}\label{section:additional_simulation}

	\subsection{Simulations with the $\ell_1$ regularization}
	In this subsection, we revisit all previously considered simulation settings in Section \ref{sec:numerics} with an added $\ell_1$ penalty $\mathsf{f}(x) = \lambda \abs{x}$ to assess the performance of our inference procedure under regularization.  In addition, we include a new setting: 
	\begin{itemize}
		\item One-bit compressed sensing under Gaussian errors, using squared loss with and without $\ell_1$ regularization.
	\end{itemize}
	In all simulations in this section we take $\lambda = 0.1$ in the regularization.
	
	\begin{figure}[t]
		\begin{minipage}[t]{0.3\textwidth}
			\includegraphics[width=\textwidth]{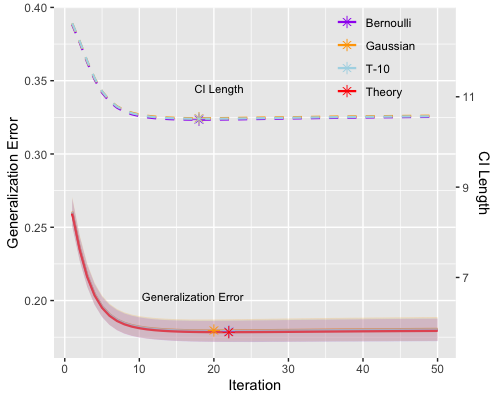}
		\end{minipage}
		\begin{minipage}[t]{0.3\textwidth}
			\includegraphics[width=\textwidth]{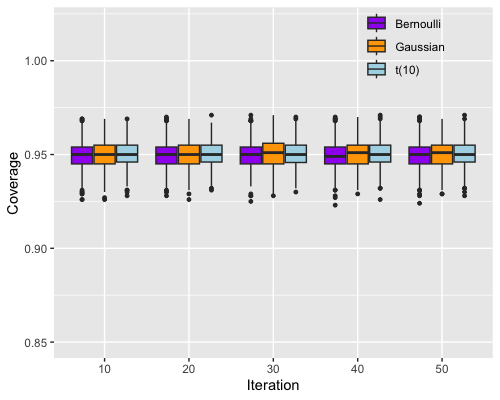}
		\end{minipage}
		\begin{minipage}[t]{0.3\textwidth}
			\includegraphics[width=\textwidth]{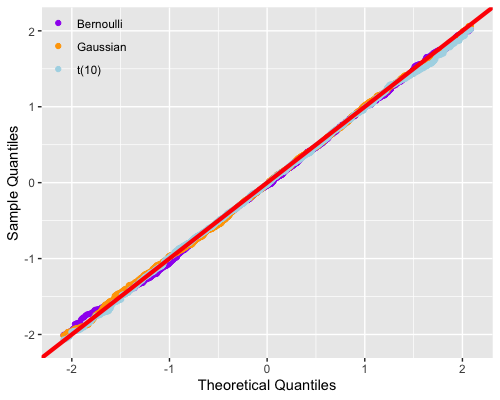}
		\end{minipage}
		\begin{minipage}[t]{0.3\textwidth}
			\includegraphics[width=\textwidth]{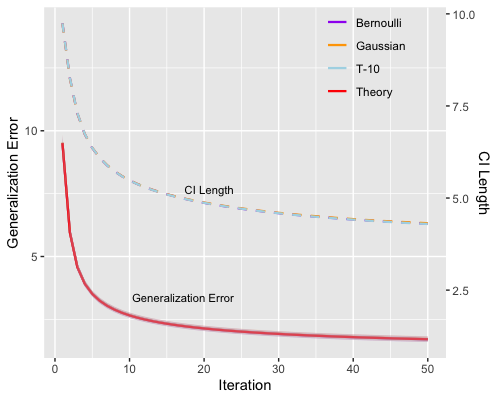}
		\end{minipage}
		\begin{minipage}[t]{0.3\textwidth}
			\includegraphics[width=\textwidth]{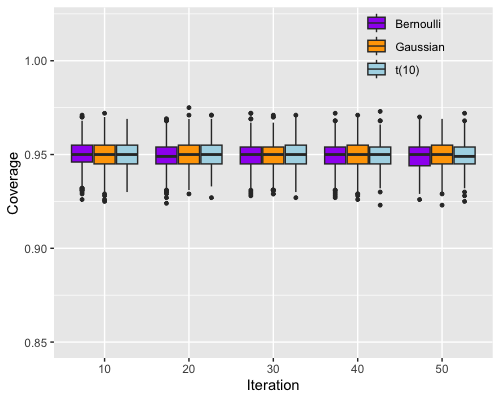}
		\end{minipage}
		\begin{minipage}[t]{0.3\textwidth}
			\includegraphics[width=\textwidth]{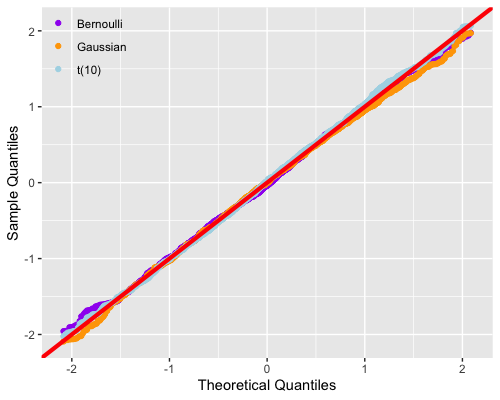}
		\end{minipage}
		\caption{Single-index regression with $\ell_1$ penalty and squared loss. \emph{Top row}: sigmoid link $\varphi_\ast(x) = 1/(1 + e^{-x})$. \emph{Bottom row}: nonlinear link $\varphi_\ast(x) = x + \sin(x)$.}
		\label{fig:5}
	\end{figure}

	\subsection*{Different simulation settings for one-bit compressed sensing in (3):}
	We examine the performance of the gradient descent inference algorithm for Example \ref{model:1bit} with the squared loss function, under i.i.d. $\mathcal{N}(0,1)$ noises $\{\xi_i\}$, considering both the unregularized and $\ell_1$-regularized cases. 
	
	In this setting, the signal strength $\sigma_{\mu_\ast}$ can be estimated by
	\begin{align}\label{def:1bit_signal_strength_est_gaussan}
	\hat{\sigma}_{\mu_\ast} \equiv  \bigg( \frac{ \hat{Q}(A,Y)}{1 - \hat{Q}(A,Y) } \bigg)^{1/2},\quad \hat{Q}(A,Y)\equiv\frac{\pi}{2\phi}\bigg(\frac{\pnorm{A^\top Y}{}^2}{m}-1\bigg)_+\wedge 1.
	\end{align}
	To justify this estimator $\hat{\sigma}_{\mu_\ast}$ for $\sigma_{\mu_\ast}$, consider the initialization $\mu^{(0)}=0$ and step size $\eta=1$. At iteration $1$, the debiased gradient descent iterate is given by $\mu^{(1)}_{\mathrm{db}}=\phi^{-1} A^\top Y$. Moreover, Proposition \ref{prop:1bit_cs_bias_var}-(2) shows that $b^{(1)}_{\mathrm{db}}=2 \E^{(0)} \mathfrak{g}_{\sigma_{\mu_\ast}}(\xi_{\pi_m})\approx 2\E \mathfrak{g}_1(\sigma_{\mu_\ast}\mathsf{Z})=2\big(2\pi(1+\sigma_{\mu_\ast}^2)\big)^{-1/2}$ and $(\sigma^{(1)}_{\mathrm{db}} )^2 =\phi^{-1}$. Then applying Proposition \ref{prop:1bit_cs_bias_var}-(1) leads to 
	\begin{align*}
	\phi^{-1}\frac{\pnorm{A^\top Y}{}^2}{ m}=\frac{\pnorm{\mu^{(1)}_{\mathrm{db}} }{}^2}{n}\approx (b^{(1)}_{\mathrm{db}} )^2 \sigma_{\mu_\ast}^2 + (\sigma^{(1)}_{\mathrm{db}} )^2 \approx \frac{2}{\pi}\cdot \frac{\sigma_{\mu_\ast}^2}{1+\sigma_{\mu_\ast}^2}+\phi^{-1}.
	\end{align*}
	The proposal (\ref{def:1bit_signal_strength_est_gaussan}) follows by inverting the above approximation.

	\begin{figure}[t]
		\begin{minipage}[t]{0.3\textwidth}
			\includegraphics[width=\textwidth]{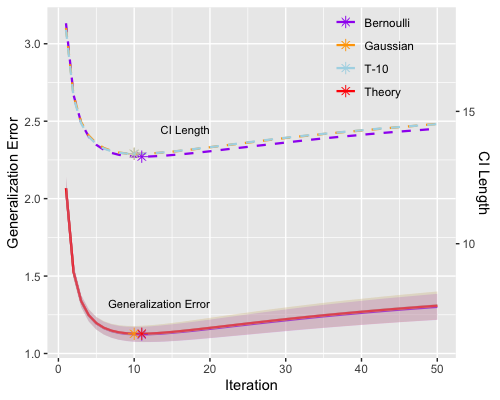}
		\end{minipage}
		\begin{minipage}[t]{0.3\textwidth}
			\includegraphics[width=\textwidth]{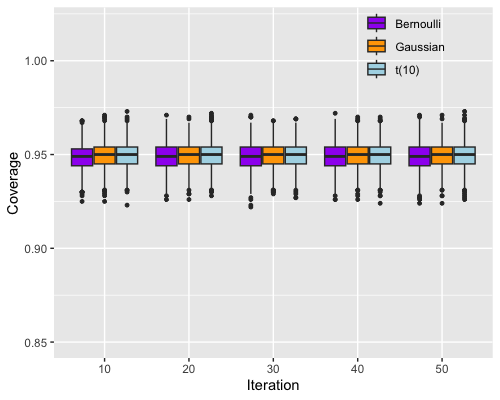}
		\end{minipage}
		\begin{minipage}[t]{0.3\textwidth}
			\includegraphics[width=\textwidth]{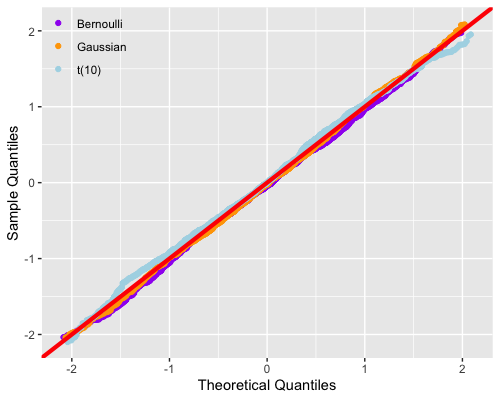}
		\end{minipage}
		\begin{minipage}[t]{0.3\textwidth}
			\includegraphics[width=\textwidth]{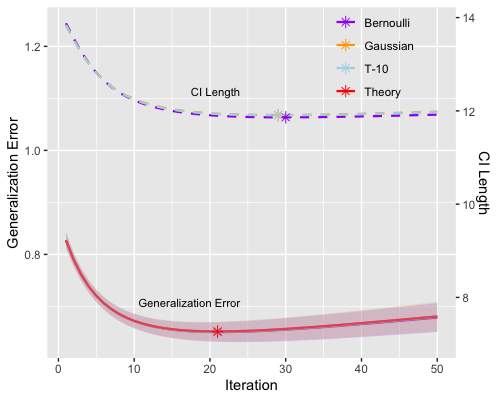}
		\end{minipage}
		\begin{minipage}[t]{0.3\textwidth}
			\includegraphics[width=\textwidth]{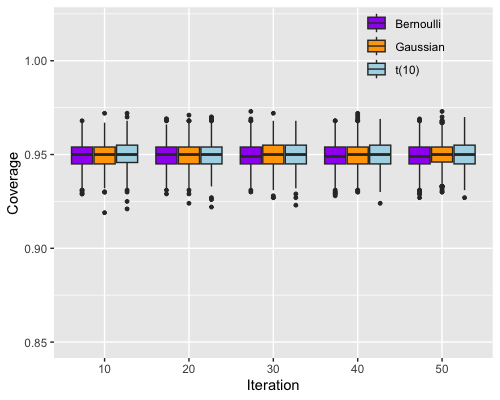}
		\end{minipage}
		\begin{minipage}[t]{0.3\textwidth}
			\includegraphics[width=\textwidth]{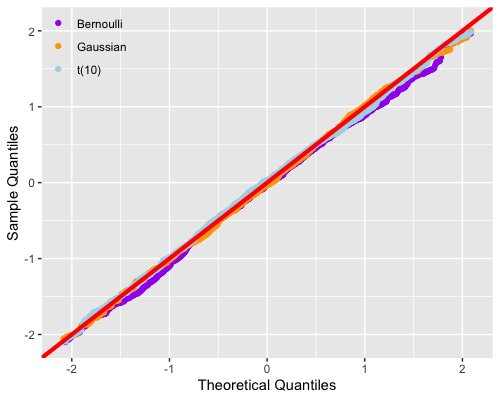}
		\end{minipage}
		\caption{Logistic regression with $\ell_1$ penalty. \emph{Top row}: Squared loss. \emph{Bottom row}: Logistic loss.}
		\label{fig:6}
	\end{figure}
	
	\subsection*{Numerical findings:}
	We present the simulation results for linear regression model in Figure \ref{fig:4}, for single-index regression model in Figure \ref{fig:5}, for logistic regression model in Figure \ref{fig:6}, and for one-bit compressed sensing in Figure \ref{fig:7}.  These figures further confirm the key findings discussed in Section \ref{sec:numerics}:
	\begin{itemize}
		\item The left panels in all Figures \ref{fig:4}-\ref{fig:7} highlight the early stopping phenomenon in gradient descent. The minimizing points for generalization error and CI length may differ across models and loss functions, and there does not appear to be a general pattern governing their relative positions.
		\item The middle panels confirm that the CIs in all settings achieve the nominal coverage level, demonstrating robust performance across varying settings. 
		\item The right panels validate the approximate normality of  $\hat{\mu}^{(t)}_{\mathrm{db}}$; specifically, we present $(\hat{\mu}^{(t)}_{\mathrm{db};1} - \mu_{\ast;1}) / \hat{\sigma}^{(t)}_{\mathrm{db}}$ at the final iteration, showing strong agreement with theoretical predictions.
	\end{itemize} 
	These findings are consistently observed across a wide range of simulation parameters. To avoid redundancy, we omit additional figures of a similar nature.

	\begin{figure}[t]
		\begin{minipage}[t]{0.3\textwidth}
			\includegraphics[width=\textwidth]{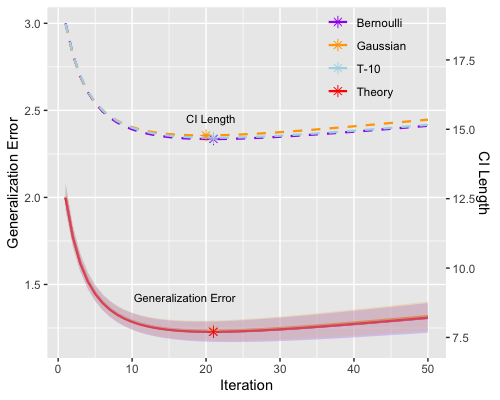}
		\end{minipage}
		\begin{minipage}[t]{0.3\textwidth}
			\includegraphics[width=\textwidth]{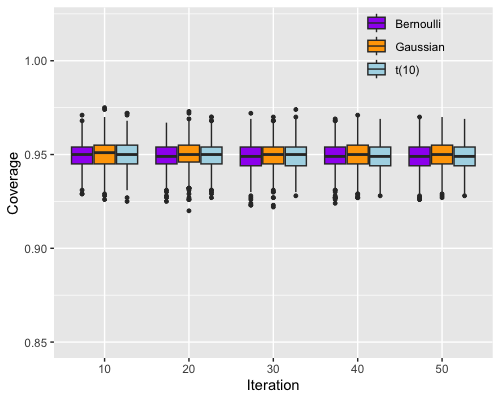}
		\end{minipage}
		\begin{minipage}[t]{0.3\textwidth}
			\includegraphics[width=\textwidth]{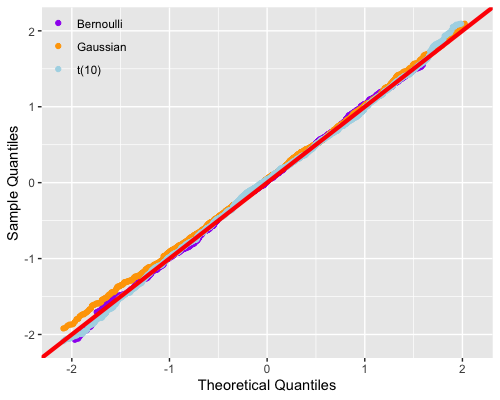}
		\end{minipage}
		\begin{minipage}[t]{0.3\textwidth}
			\includegraphics[width=\textwidth]{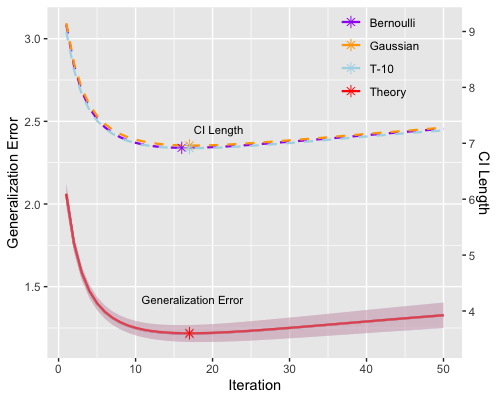}
		\end{minipage}
		\begin{minipage}[t]{0.3\textwidth}
			\includegraphics[width=\textwidth]{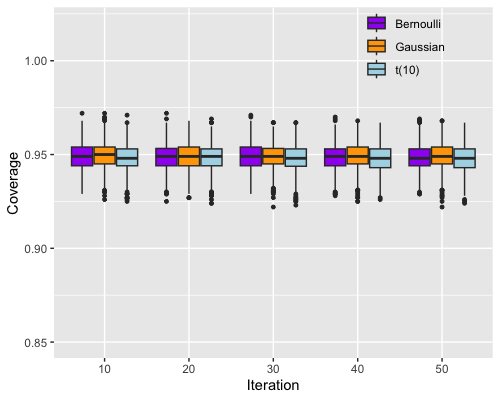}
		\end{minipage}
		\begin{minipage}[t]{0.3\textwidth}
			\includegraphics[width=\textwidth]{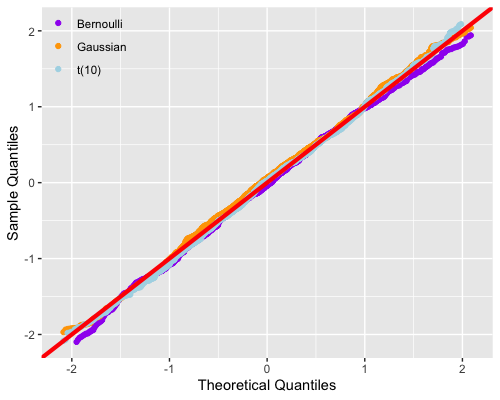}
		\end{minipage}
		\caption{One-bit compressed sensing under Gaussian errors with squared loss. \emph{Top row}: Unregularized case. \emph{Bottom row}: Regularized case with $\ell_1$ penalty. }
		\label{fig:7}
	\end{figure}

	\subsection{Comparison with leave-one-out cross-validation}\label{subsection:loo_gen_est}
	In this subsection, we compare our proposed generalization error estimator $\hat{\mathscr{E}}^{(\cdot)}_{\abs{\cdot}^2/2}$ with the leave-one-out cross-validation (LOOCV) method in \cite{patil2024failures}.
	
	We consider two models: (i) linear regression and (ii) single-index regression with a sigmoid link, both fitted using gradient descent on the squared loss as follows: starting from an initialization $\mu^{(0)} \in \R^n$, for a fixed step size $\eta > 0$, we iteratively update
	\begin{align*}
	\mu^{(t)} = \mu^{(t-1)} - \eta \cdot A^\top ( A \mu^{(t-1)} - Y), \quad t = 1,2,\ldots.
	\end{align*}
	The LOOCV estimator is defined as
	\begin{align*}
	\hat{\mathscr{E}}^{(t)}_{\mathsf{loo}} \equiv \frac{1}{2m}\sum_{i \in [m]} (Y_i - A_i^\top \mu_{[-i]}^{(t)} )^2,
	\end{align*}
	where $\mu_{[-i]}^{(t)}$ denotes the $t$-th iterate computed on the dataset with the $i$-th observation removed, using the same initialization and step size. Note that for the single-index regression with a sigmoid link, the above gradient descent mis-specifies the model (and therefore the loss function).

	Figure \ref{fig:8} shows that both estimators $\hat{\mathscr{E}}^{(\cdot)}_{\abs{\cdot}^2/2}$ and $\hat{\mathscr{E}}^{(\cdot)}_{\mathsf{loo}}$ produce similar generalization error estimates across iterations. However, our method is significantly more efficient: at iteration $t$, computing all $m$ LOOCV iterates requires $\mathcal{O}(n^3)$ operations, whereas our proposal has complexity at most $\mathcal{O}(nt^2+n^2)$. As illustrated in the figure, our method is several hundred times faster than LOOCV when the number of iterations is moderate. This observation is consistent with the theoretical complexity comparison and highlights the practical benefit of our approach when $t \ll n$.

	\begin{figure}[t]
	\begin{minipage}[t]{0.4\textwidth}
		\includegraphics[width=\textwidth]{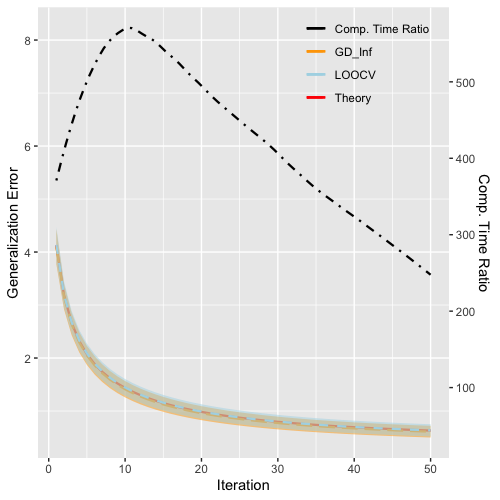}
	\end{minipage}
	\quad \quad \quad 
	\begin{minipage}[t]{0.4\textwidth}
		\includegraphics[width=\textwidth]{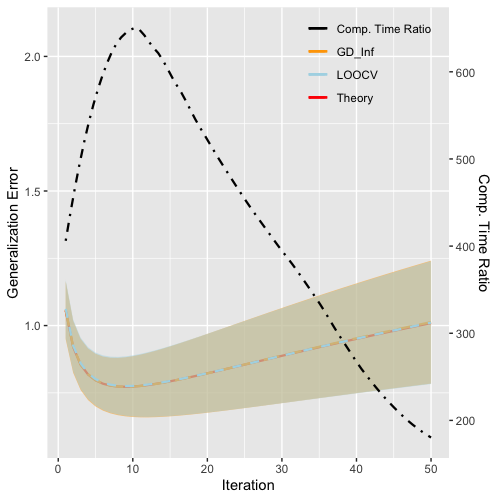}
	\end{minipage}
	\caption{Comparison with leave-one-out cross-validation. \emph{Left panel}: Linear regression model. \emph{Right panel}: Single-index regression model with sigmoid link function. \emph{Simulation parameters}: $\eta = 0.2$, $m = 120$, $n = 100$, $\mu_\ast \in \R^n$ are i.i.d. $\abs{\mathcal{N}(0,5)}$, $\xi \in \R^m$  has i.i.d. entries  drawn from $\mathcal{N}(0,0.1)$, $\sqrt{n}A$ has i.i.d. entries following $\mathcal{N}(0,1)$.
		The algorithms are run for $50$ iterations with Monte Carlo repetition $B = 1000$.}
	\label{fig:8}
\end{figure}

 \section*{Acknowledgments}
 The authors would like to thank the Editor, an Associate Editor and three referees for their helpful comments and suggestions that significantly improved the quality of the paper.

\bibliographystyle{alpha}
\bibliography{mybib}

\end{document}